\documentclass[a4paper,11pt]{article}

\usepackage{amsmath,amssymb,graphics,epsfig,color,cite}
\usepackage{mathrsfs,bbm}
\usepackage{amsthm} 
\usepackage{bm}
\usepackage[toc,page]{appendix}
\usepackage{tikz}
\usetikzlibrary{arrows}
\theoremstyle{plain}
\usetikzlibrary{decorations.pathreplacing,decorations.markings}
\tikzset{arrow data/.style 2 args={%
      decoration={%
         markings,
         mark=at position #1 with \arrow{#2}},
         postaction=decorate}
      }%

 \usepackage{enumitem}

\usepackage{hyperref}
 
\usepackage{geometry}
\DeclareMathAlphabet{\mathpzc}{OT1}{pzc}{m}{it}

\newcounter{savecntr}

\newcommand{\mc}[1]{\mathcal{#1}}

\newcommand{\mbf}[1]{\mathbf{#1}}

\newcommand{\lp}{\langle}
\newcommand{\rp}{\rangle}

\DeclareMathOperator{\Hess }{Hess}

\DeclareMathOperator*{\argmin}{arg\,min}

\def\Hess {{\rm Hess\,}}

\def\supp{\mathop{\rm supp} \nolimits} 
\def\Ran{{\rm Ran}}
\def\sspan{{\rm Span}}

\def\and {{\rm \; and \;}}

\def\dim {{\rm \; dim  \;}}

\newcommand{\ft}[1]{\mathsf{#1}} 
\newcommand {\pa}{\partial}

\newtheorem{theorem}{Theorem}
\newtheorem*{theorem*}{Theorem}
\newtheorem{proposition}{Proposition}
\newtheorem{definition}[proposition]{Definition}
\newtheorem*{definition*}{Definition}
\newtheorem{lemma}[proposition]{Lemma}
\newtheorem{corollary}[proposition]{Corollary}
\newtheorem{remark}[proposition]{Remark}

\newtheorem*{assumption*}{Assumption}

\newtheorem{manualtheoreminner}{Theorem}
\newenvironment{manualtheorem}[1]{%
  \renewcommand\themanualtheoreminner{#1}%
  \manualtheoreminner
}{\endmanualtheoreminner}

\newtheorem{manualassumptioninner}{Assumption}
\newenvironment{manualassumption}[1]{%
  \renewcommand\themanualassumptioninner{#1}%
  \manualassumptioninner
}{\endmanualassumptioninner}

\date{}

\usepackage{titlesec}
\titleformat*{\section}{\Large\bfseries}
\titleformat*{\subsection}{\large\bfseries}
\titleformat*{\subsubsection}{\normalsize\bfseries}
\titleformat*{\paragraph}{\normalsize\bfseries}
\titleformat*{\subparagraph}{\normalsize\bfseries}

\usepackage[hang,font={normalsize,it},labelfont=sf]{caption}

\title{\Large{Small eigenvalues of the Witten Laplacian with Dirichlet boundary conditions: the case with critical points on the boundary}}
\author{Dorian Le Peutrec\setcounter{savecntr}{\value{footnote}}\thanks{Laboratoire de Math\'ematiques d'Orsay, Univ. Paris-Sud, CNRS, Universit\'e Paris-Saclay, 91405 Orsay, France. E-mail: dorian.lepeutrec@math.u-psud.fr}$\, \ $ and Boris Nectoux  \setcounter{savecntr}{\value{footnote}}\thanks{Institut f\"{u}r Analysis und Scientific Computing, TU Wien, Wiedner Hauptstr. 8, 1040 Wien, Austria. E-mail: boris.nectoux@enpc.fr}}

\begin{document} 

 \maketitle
 
\begin{abstract}
In this work, we give  sharp asymptotic equivalents in the limit $h\to 0$ of the small eigenvalues of the Witten Laplacian, that is the operator associated with   the quadratic form 
$$\psi\in H^1_0(\Omega)\mapsto  h^2 \int_\Omega \big \vert \nabla \big (e^{\frac 1hf} \psi\big )\big \vert^2\, e^{-\frac 2hf},$$
where $\overline\Omega=\Omega\cup \pa\Omega$
is an  oriented $C^\infty$ compact and  connected  Riemannian
manifold with non empty boundary $\pa\Omega$
 and    $f: \overline \Omega\to \mathbb R$  is a $C^\infty$ Morse function.
 The function $f$  is allowed to admit critical points on $ \pa\Omega$,
 which is the main novelty of this work in  comparison with the existing literature.\\ 
  \textbf{Keywords:} Witten Laplacian, overdamped Langevin dynamics, semiclassical analysis, metastability, spectral theory, Eyring-Kramers formulas.  \\
  \textbf{MSC 2010:} 35P15,   35P20, 47F05,  35Q82. 
 \end{abstract}

  
 \noindent
 \tableofcontents

 \noindent
 \section{Introduction}
 \subsection{Setting}
 
Let $(\overline\Omega,g)$ be an oriented $C^{\infty}$ compact and  connected  Riemannian manifold of dimension $d$ with interior $\Omega$ and non empty boundary $\partial \Omega$, and let  $f:\overline \Omega\to \mathbb R$ be a $C^\infty$ function. 
Let us moreover denote by $d$ the 
exterior derivative acting on functions on~$\overline \Omega$
and by $d^*$ its formal adjoint (called the  co-differential) acting  on $1$-forms  (which are naturally identified with vector fields).
For any $h>0$, the semiclassical Witten Laplacian acting on functions on $\overline \Omega$ is then 
the Schr\"odinger operator defined
by 
$$
\Delta_{f,h}\ :=\ d^*_{f,h}d_{f,h}\ =\ h^2\Delta_H+\vert \nabla f\vert ^2 +h\Delta_H f\,,
$$
where
$\Delta_H=d^{*}d$ is the Hodge Laplacian acting on functions,
that is the negative of the Laplace--Beltrami operator,
and 
$$d_{f,h}\ :=\ h\, e^{-\frac fh}d\, e^{\frac fh}
\quad\text{and}\quad
d_{f,h}^*=h\, e^{\frac fh}d^*\, e^{-\frac fh}$$
are respectively the 
distorted exterior derivative and co-differential.
This operator was
originally  introduced by Witten in~\cite{Wit}
and acts more generally on the algebra of differential forms. 
Note also the relation 
\begin{equation}
\label{eq.unit}
\Delta_{f,h}\ =\ h\,e^{-\frac V{2h}}\,\big( h\Delta_H +  \nabla V \cdot \nabla \big)\,e^{\frac V{2h}}\quad\text{where}
\quad V=2f,
\end{equation}
where the notation $\nabla V \cdot \nabla$ stands for $g(\nabla V,\nabla\cdot)$.
It is then equivalent to study the Witten Laplacian 
$\Delta_{f,h}$ acting in the flat space $L^{2}(\Omega)=L^{2}(\Omega,d\text{Vol}_{\Omega})$
or the weighted Laplacian $$L_{V,h}:= h\Delta_H +  \nabla V \cdot \nabla=d_{V,h}^* \,d$$
acting in the weighted space $L^{2}(\Omega,e^{-\frac Vh}d\text{Vol}_{\Omega})$.\medskip

\noindent
Let us now consider the usual self-adjoint Dirichlet realization 
$\Delta_{f,h}^{D}$ of the Witten Laplacian $\Delta_{f,h}$ on the Hilbert space
$L^{2}(\Omega)$.
Its domain is given by 
 $$D\big( \Delta_{f,h}^{D}\big)=H^2(\Omega)\cap H^1_0(\Omega),$$
where, for $p\in\mathbb N^{*}$, we denote by  $H^p(\Omega)$
the usual Sobolev space with order $p$ and by $H^1_0(\Omega)$
the set made of the functions in  $H^1(\Omega)$  with vanishing
trace on $\pa\Omega$. 
We refer for instance to \cite{GSchw} for more material about Sobolev spaces
on manifolds with boundary.
The operator
$\Delta_{f,h}^{D}$ 
has a compact resolvent, and thus  its spectrum $\sigma(\Delta_{f,h}^{D})$
is discrete. This operator is moreover nonnegative since it satisfies:
\begin{equation}\label{eq.qfh'}
\forall \psi\in D\big( \Delta_{f,h}^{D}\big)\,,\ \  \langle \Delta_{f,h}^{D} \psi, \psi\rangle_{L^{2}(\Omega)}\ =\  
\| d_{f,h} \psi\|^{2}_{\Lambda^{1}L^{2}(\Omega)}
\ =\ h^{2}\int_\Omega \vert d(e^{\frac fh}\psi) \vert^2e^{-\frac {2f}h}\,,
\end{equation}
where  $\Lambda^{1}L^{2}(\Omega)$ denotes 
the space of $1$-forms in $L^{2}(\Omega)$
and
$\| d_{f,h} \psi\|^{2}_{\Lambda^{1}L^{2}(\Omega)}
 = \int_\Omega \vert d_{f,h}\psi \vert^2$.
Let us also mention here that
 the (closed) quadratic form
$Q_{f,h}$
associated
with $\Delta_{f,h}^{D}$ has domain $H^1_0( \Omega)$
and satisfies, for every $\psi\in H^1_0( \Omega)$,
\begin{equation}\label{eq.qfh}
Q_{f,h}(\psi) := Q_{f,h}(\psi,\psi) =  \| d_{f,h} \psi\|^{2}_{\Lambda^{1}L^{2}(\Omega)} = h^2 \int_\Omega  \vert d\psi \vert^2  +  \int_\Omega    \big(\vert \nabla f\vert^2 +h\Delta_H f\big)\, \big  \vert  \psi   \big \vert^2\,.
\end{equation}

\begin{remark}\label{re.elli}
From standard results on elliptic operators,
the principal eigenvalue of  $\Delta_{f,h}^D$, which is positive since $e^{-\frac fh}\notin H^{1}_{0}(\Omega)$ (see 
\eqref{eq.qfh'}), is moreover non degenerate and any
associated  eigenfunction  has a sign on $\Omega$
(see for example~\cite{MR1814364, Eva}).
\end{remark}

\subsection{Spectral approach of metastability in statistical physics}

The operator $L_{V,h}=\frac1h e^{\frac fh}\,\Delta_{f,h} e^{-\frac fh} $, where
we recall that $V=2f$ (see \eqref{eq.unit}),
is the infinitesimal generator  of the overdamped Langevin process 
\begin{equation}\label{eq.langevin}
dX_t=-\nabla V(X_t)dt+\sqrt {2h} \, dB_t
\end{equation} 
which is  for instance  used  to describe the motion of the atoms of a molecule or the diffusion of impurities in a crystal. When the temperature  of the system  is small, i.e. when $h\ll 1$, the process~\eqref{eq.langevin} is 
typically metastable: it is trapped during a long period of time in  a neighborhood of a local minimum of $V$, called a metastable
region, before   reaching another metastable region.
 \medskip

\noindent
When one looks at the process~\eqref{eq.langevin}
on a metastable region $\Omega$ with absorbing boundary
conditions,    the evolution of  observables is in particular
given by the semigroup $e^{-t L^{D}_{V,h}}$, where  
$L^{D}_{V,h}:= \frac1h e^{-\frac fh}\,\Delta^{D}_{f,h} e^{-\frac fh}$ is the Dirichlet realization of the weighted Laplacian $L_{V,h}$ in the weighted space $L^{2}(\Omega,e^{-\frac Vh}d\text{Vol}_{\Omega})$, see \eqref{eq.unit}.
A first description of the metastability  of the process
\eqref{eq.langevin} with absorbing boundary conditions  
is then given by the behaviour of the low-lying spectrum of
 the Dirichlet realization $\Delta^{D}_{f,h}$ of the Witten Laplacian in the limit $h\to 0$. The metastable behaviour
of the dynamics 
is more precisely characterized by the fact that the low-lying spectrum of
$\Delta^{D}_{f,h}$ contains
exponentially small eigenvalues, i.e.   eigenvalues of
order $O(e^{-\frac ch})$ where $C>0$.
The first mathematical results in this direction probably go back
to  the works of Freidlin-Wentzell in the 
framework
of their large deviation theory developed in the 70's and we refer in particular
to their book \cite{FrWe} for an overview on this topic.
In this context, 
when $\lambda_{h}$ is some exponentially small eigenvalue
of $\Delta^{D}_{f,h}$, the limit
of $h\ln \lambda_{h}$
 has been investigated 
 assuming that   (see \cite[Section~6.7]{FrWe})
\begin{equation}\label{eq.PCA}
\vert \nabla f\vert \neq 0 \text{ on } \pa \Omega.
\end{equation}
The results of \cite{FrWe} imply in particular that, when $\pa_{n}f>0$ on $\pa\Omega$
and $\Omega$ contains a unique critical point of $f$ which is non degenerate and is hence the global minimum
of $f$ in $\overline \Omega$,  the principal eigenvalue $\lambda_{1,h}$
of $\Delta^{D}_{f,h}$ satisfies
$$
\lim_{h\to 0} h \ln \lambda_{1,h}\ =\ -2\,(\min_{\pa\Omega} f - \min_{\overline\Omega} f)\,.
$$
The asymptotic logarithmic behaviour of the low-lying spectrum of $\Delta^{D}_{f,h}$
has also been studied in 
\cite{mathieu-95} dropping the assumption~\eqref{eq.PCA}. 
When   $f$ and $f|_{\pa \Omega}$ are smooth   Morse functions and~\eqref{eq.PCA} holds, 
precise asymptotic formulas in the limit $h\to 0$ have been given 
by Helffer-Nier
in 
\cite{HeNi1} where they prove in particular
that under additional generic hypotheses on the function $f$,
any exponentially  small eigenvalue $\lambda_{h}$  of $\Delta^{D}_{f,h}$
satisfies the following
Eyring-Kramers type formula when $h\to 0$:
\begin{equation}\label{eq.precise-a}
\lambda_{h}\ =\  A \, h^{\gamma} \, e^{-\frac 2h E} \,\big(1+\varepsilon(h)\big),
\end{equation}
where $A>0$, $E>0$, and $\gamma\in\{\frac12,1\}$ are explicit, and the error  term $\varepsilon(h)$ is of the order $O(h)$ and admits a full asymptotic expansion in $h$.
The constants $E$'s involved in~\eqref{eq.precise-a}
are the depths of some characteristic wells of the potential $f$ in $\Omega$.
The results of \cite{HeNi1}, obtained by a semiclassical approach, were following similar results
obtained  in the case without boundary
in \cite{HKS,miclo-95,BEGK,BGK}
by a probabilistic approach 
and
in \cite{HKN} by a  semiclassical approach.
We also refer to \cite{herau-hitrick-sjostrand-11,michel2017small}
for a generalization of the results obtained in \cite{HKN} in the
case without boundary (see also \cite{BeGe2010,BD15,landim2017dirichlet}
for related results),
to \cite{DLLN-saddle1}\footnote{This work corresponds to the first part of the preprint~\cite{DLLN-saddle0}.} for a generalization
of the results obtained in \cite{HeNi1} in the case of Dirichlet boundary conditions (see also
\cite{LeNi,DoNe,nectoux2017sharp,borisov} for related results),
and  to~\cite{mathieu-95, Lep} in the case of Neumann  boundary conditions.
Finally, we refer to~\cite{Ber,Lep-HDR} for a comprehensive review on this topic.

\subsection{Motivation and results}

{\bf Motivation.} These  past few years,   several efficient algorithms have been designed  to  accelerate the sampling of  the exit event from a metastable region $\Omega$, such as for instance the Monte Carlo methods~\cite{schuette-98, schuette-sarich-13, voter-05,wales-03,cameron-14b,fan-yip-yildiz-14} or  the   \textit{accelerated dynamics} algorithms~\cite{sorensen-voter-00,voter-97,voter-98}. These algorithms rely on a very precise asymptotic understanding of the metastable behaviour of the  process $(X_t)_{t\ge 0}$  in a metastable
region  $\Omega$ when  $h\to 0$, and in particular
on the validity of Eyring-Kramers type formulas
of the type~\eqref{eq.precise-a} in the limit $h\to 0$.
Moreover, though the hypothesis \eqref{eq.PCA} considered in \cite{HeNi1,DLLN-saddle1}
is generic, 
in most applications of the accelerated  algorithms mentioned above, the domain~$\Omega$ is  the basin of attraction of some local minimum of $f$  for the  dynamics $\dot{X}=-\nabla f(X)$ so that the function $f$ admits critical points on the boundary of 
$\Omega$.\medskip

\noindent
In this work, we precisely aim at giving a precise description  of 
the low-lying spectrum of~$\Delta^{D}_{f,h}$ in the limit $h\to 0$ of the type \eqref{eq.precise-a} 
in  a rather general geometric setting covering the latter case (though we assume
$\Omega$ to have a smooth boundary). This establishes the first step
to precisely describe
the metastable behaviour of the  overdamped Langevin process~\eqref{eq.langevin}
with absorbing boundary conditions 
in $\Omega$
 when $\pa\Omega$ contains critical points of $f$.
 Let us also point out that, 
though the spectrum of~$\Delta^{D}_{f,h}$ (or equivalently of~$L^{D}_{V,h}$)  has been widely studied these past few decades, up to our knowledge, this setting has not   been treated in the mathematical literature.  
Our
techniques come from semiclassical analysis and, in Section~\ref{sec.strategy} below, we detail various difficulties arising when considering critical points of~$f$ on~$\pa \Omega$ with such techniques. 
 \medskip

\noindent
{\bf Results.}
We recall that we assume that $\overline\Omega$ is a $C^{\infty}$ oriented compact and  connected Riemannian manifold of dimension $d$ with interior $\Omega$ and boundary $\partial \Omega\neq\emptyset$, and  that $f:\overline \Omega\to \mathbb R$ is a $C^\infty$ Morse function. For $\mu\in \mathbb R$,
we will use the notation
$$\{f\le \mu\}=\{x\in \overline \Omega, \ f(x)\le \mu\},\ \ \{f< \mu\}=\{x\in \overline \Omega,\  f(x)< \mu\},$$
and 
$$\{f=\mu\}=\{x\in \overline \Omega, \ f(x)=\mu\}.$$
Moreover, for all $z\in \pa \Omega$, $\ft n _{\Omega}(z)$ will denote the unit outward  vector to $\Omega$ at $z$. 
Finally,  for $r>0$ and $y\in \overline \Omega$, $B(y,r)$ will denote the open ball of radius $r$ centered at $y$ in $\overline \Omega$:
$$B(y,r):=\{z\in \overline \Omega, \, \vert y-z\vert <r\},$$
where, for $y\in \overline \Omega$, $\vert y - z\vert$ is the geodesic distance between $y$ and $z$ in $\overline \Omega$.
 \medskip
 
\noindent 
Since stating  our main results, which are Theorems~\ref{th.thm2} and 
\ref{th.thm3} (see Section~\ref{sub.sharp}), 
requires substantial  additional material, we just 
give here  simplified (and weaker) versions of these results.
We first give a preliminary result
stating that,
when
$f:\overline\Omega\to \mathbb R$ is
 a Morse function,
  the number of small eigenvalues of $\Delta_{f,h}^{D}$ is
  the number of local minima of $f$
 in $\Omega$. This requires the following definition.

\begin{definition}\label{de.U0}
 Let us assume that $f:\overline \Omega\to \mathbb R$ is a $C^\infty$ Morse function. 
The set of  local minima of $f$ in $\Omega$ is then denoted by $\ft U_0 $ and one defines
$$\ft m_0:= {\rm  Card}\big ( \ft U_0  \big)\ \in\ \mathbb N.$$
\end{definition}

\begin{theorem}\label{thm.count}
Let us assume that $f:\overline \Omega\to \mathbb R$ is a $C^\infty$ Morse function. Then,  there exist $c_0>0$ and $h_0>0$ such that for all $h\in (0,h_0)$: 
$$\dim \Ran \, \pi_{[0,c_0h]}\big (\Delta_{f,h}^{D}\big )
\ =\ 
\dim \Ran \, \pi_{(0,e^{-\frac{c_0}{h}})}\big (\Delta_{f,h}^{D}\big )
\ =\ \ft  m_0\,,$$
where, 
for a Borel set $E\subset \mathbb R$, $\pi_{E}(\Delta_{f,h}^{D}\big )$
denotes the spectral projector associated with 
$\Delta_{f,h}^{D}$
and $E$, and the nonnegative integer
 $\ft m_0$ is defined in Definition~\ref{de.U0}. 
\end{theorem}
\noindent
 Let us emphasize  
 that
 the local minima of $f$ included in  $\pa\Omega$ are not listed in $\ft U_0$. 
%
 This
preliminary result  is expected from works such as
\cite{HeSj4,HeNi1} but we did not find any such statement in the literature
in our setting when the boundary admits critical points of $f$.
Theorem~\ref{thm.count} will  be proven in Section~\ref{se.count}.  
  \medskip

\noindent In the sequel, 
when  $\ft m_{0}>0$,
we will denote by 
$$
0\ <\ \lambda_{1,h}\ <\ \lambda_{2,h}\ \leq\  \cdots \ \leq\  \lambda_{\ft m_{0},h}
$$
the $\ft m_{0}$ exponentially small eigenvalues 
of $\Delta_{f,h}^{D}$
 in the limit $h\to 0$ (see~Theorem~\ref{thm.count}).
The  first main result of this paper  is Theorem~\ref{th.thm2}, 
which is stated and proven  in Section~\ref{sub.sharp}.
Here is a simplified version of this result, in a 
less general setting. The notation $\Hess  f(z)$ at 
a critical point $z$ of $f$ below stands for the endomorphism
of the tangent space $T_z\overline\Omega$ 
canonically associated with the usual  symmetric bilinear form
 $\Hess  f(z)$ on $T_z\overline\Omega\times T_z\overline\Omega$
 via the metric $g$.
\begin{manualtheorem}{\ref{th.thm2}'}
\label{th.thm2'}
Let us assume that the number of local minima $\ft m_{0}$ of the Morse function 
$f$ is positive, that $f|_{\pa\Omega}$ has only non degenerate
local minima, and that
at any  saddle point (i.e. critical point of index $1$) $z$   of $f$ which belongs to $\pa\Omega$, 
$\ft n _{\Omega}(z)$ is an eigenvector of $\Hess  f(z)$
associated with its unique negative eigenvalue.
Then,
there exists $C>0$ such that one has in the limit $h\to 0$:
\begin{equation}
\label{eq.sharp0}
\forall j\in\{1,\dots,\ft m_{0}\}\,,
\ \ \frac 1C \, h^{\gamma_{j}} \, e^{-\frac 2h E_{j}} \ \leq\
\lambda_{j,h}\ \leq\ 
C\, h^{\gamma_{j}} \, e^{-\frac 2h E_{j}}\,,
\end{equation}
 where, for  $j\in\{1,\dots,\ft m_{0}\}$,
  $E_{j}>0$, and $\gamma_{j}$ are explicit  with moreover
$\gamma_{j}\in\{\frac12,1\}$.
\end{manualtheorem}

\noindent
The above constants $E_{j}$'s are  the depths
of some characteristic wells of the potential~$f$ in~$\Omega$ which are defined through the map
\begin{equation}
\label{de.j-0}
{\bf j}: \ft U_0  \to \mathcal P( \ft U_1^{\ft{ssp}}(\overline \Omega)) 
\end{equation} constructed in Section~\ref{sec.j} (see \eqref{de.j} there).
Here $\mathcal P( \ft U_1^{\ft{ssp}}(\overline \Omega))$ denotes the power set
of  
$\ft U_1^{\ft{ssp}}(\overline \Omega)$, the set of relevant 
{\it generalized} saddle points (or critical points of index $1$) of $f$ in  $\overline \Omega$
(see Definition~\ref{de.SSP} at the end of Section~\ref{sec.gene-pri2}).
To be a little more precise here, we have the inclusion 
\begin{align*}
\ft U_1^{\ft{ssp}}(\overline \Omega)\ \subset\ &\{\text{critical points of $f$ in $\overline\Omega$ of index $1$}\}\\
\cup&\ \{\text{local minima $z$ of $f|_{\pa\Omega}$ in $\pa\Omega$ such that $\pa_{\ft n_\Omega}f(z)>0$}\},
\end{align*}
where $\pa_{\ft n_\Omega}f(z)= \ft n_\Omega(z) \cdot \nabla f(z)$ denotes the normal derivative
of $f$ at $z$. Moreover, $f$ is constant on each ${\bf j}(x)$, $x\in \ft U_0 $,
and the $E_{j}$'s are precisely the $f({\bf j}(x))-f(x)$'s, where, with a slight abuse of notation, we have identified
$f({\bf j}(x))$ with its unique element, see Section~\ref{sec.association} for precise statements.
Note that the $E_{j}$'s  give the logarithmic equivalents of 
the small eigenvalues of $\Delta_{f,h}^{D}$ since 
 the relation \eqref{eq.sharp0} obviously implies:
$$
\forall j\in\{1,\dots,\ft m_{0}\}\,,
\ \ \lim_{h\to 0} h \ln \lambda_{j,h}
\ =\ -2E_{j}
\,.
$$
Note also that when $\Omega$ is the basin of attraction of 
some local minimum (or of some family of local minima)
of
some Morse function $f$ for the  flow of $\dot{X}=-\nabla f(X)$
and $z$ is a saddle
 point of $f$ which belongs to $\partial\Omega$, the following holds: $\partial \Omega$
 is a smooth manifold of dimension $d-1$ near $z$
 and $\ft n_{\Omega}(z)$ 
 is an eigenvector of $\Hess  f(z)$
associated with its unique negative eigenvalue.
More precisely, $\pa\Omega$ coincides with the stable manifold of
$z$ for the dynamics $\dot X = -\nabla f(X)$ near the saddle point $z$
(see \eqref{eq.stable} in Section~\ref{se.count}). 
In Theorem~\ref{th.thm2}, the corresponding assumption is more general since we just   assume that the boundary $\pa\Omega$ of  $\Omega$
 is, at the saddle points $z\in \mbf j(\ft U_0)\cap \pa \Omega$, tangent to the stable manifold of $z$.
 \medskip

\noindent
 Finally, the  second main result of this work is Theorem~\ref{th.thm3}, 
which is stated and proven  in Section~\ref{sub.sharp}. 
It states that, under the hypotheses of Theorem~\ref{th.thm2},
which, we recall, are a little more general than the ones
of Theorem~\ref{th.thm2'},
plus additional  very general generic hypotheses on the separation of
the characteristic wells of $f$ 
defined through the map  ${\bf j}: \ft U_0  \to \mathcal P( \ft U_1^{\ft{ssp}}(\overline \Omega)) $ (see \eqref{de.j-0} and the lines below), one has in the limit $h\to 0$ sharp asymptotic estimates of
the type \eqref{eq.precise-a} on
all or  part of 
the smallest eigenvalues of $\Delta_{f,h}^{D}$.
 We state below  a simplified version of Theorem~\ref{th.thm3},  in a 
less general setting, where we do not make explicit the pre-exponential factors (see Theorem~\ref{th.thm3} for explicit formulas).
 
\begin{manualtheorem}{\ref{th.thm3}'}
\label{th.thm3'}
Let us assume the hypotheses of Theorem~\ref{th.thm2'}.
Then, under generic hypotheses on 
the characteristic wells of $f$  defined through the map
${\bf j}: \ft U_0  \to \mathcal P( \ft U_1^{\ft{ssp}}(\overline \Omega))$ defined in \ref{de.j},
one has in the limit $h\to 0$:
\begin{equation}
\label{eq.sharp0'}
\forall j\in\{1,\dots,\ft m_{0}\}\,,
\ \ 
\lambda_{j,h}\ =\ 
A_{j}\, h^{\gamma_{j}} \, e^{-\frac 2h E_{j}}\,\big(1+O(\sqrt h)\big)\,,
\end{equation}
 where, for  $j\in\{1,\dots,\ft m_{0}\}$,
 $A_{j}>0$, $E_{j}>0$, and $\gamma_{j}$ are explicit  with moreover
$\gamma_{j}\in\{\frac12,1\}$, and the remainder term  $O(\sqrt h)$
is actually  of the order $O(h)$ when the boundary of the associated characteristic
well does not meet  both $(\vert \nabla f \vert)^{-1}(\{0\})$~and~$\pa\Omega$.\\
In addition, when $\ft m_0\geq 2$ and $E_{\ft m^*}>E_{\ft m^*+1}$ for some $\ft m^*\in \{1,\ldots, \ft m_0-1\}$,
 the previous estimates
remain valid
for  $\lambda_{1,h},\dots,\lambda_{\ft m^*,h} $ under more  general hypotheses:
\begin{equation}
\label{eq.sharp0''}
\forall j\in\{1,\dots,\ft m^*\}\,,
\ \ 
\lambda_{j,h}\ =\ 
A_{j}\, h^{\gamma_{j}} \, e^{-\frac 2h E_{j}}\,\big(1+O(\sqrt h)\big)\,,
\end{equation}
 where, for  $j\in\{1,\dots,\ft m_{0}\}$,
 $A_{j}>0$, $E_{j}>0$, and $\gamma_{j}$ are explicit  with moreover
$\gamma_{j}\in\{\frac12,1\}$, and the remainder term  $O(\sqrt h)$
is actually of the order $O(h)$ when the boundary of the associated characteristic
well does not meet both  $(\vert\nabla f \vert)^{-1}(\{0\})$~and~$\pa\Omega$.
\end{manualtheorem}

\noindent
Let us  now comment about  this result.
\medskip

\noindent
First, the
 above error terms $O(\sqrt h)$ or $O( h)$ 
  are in general optimal,  see indeed Remark~\ref{re.optimality} below.
   \medskip
  
\noindent   
 Moreover, even when $\ft m^*=\ft m_0$ in  Theorem~\ref{th.thm3'}, 
the geometric assumptions on the characteristic wells of $f$ are still  
more general than the generic hypotheses made e.g. in \cite{BGK,HKN}
in the case without boundary or in
\cite{HeNi1} in the case with boundary, see indeed \cite[Assumption 5.3.1]{HeNi1}. For instance,  
our hypotheses neither imply that
the $E_{j}$'s  are  distinct, nor that the ${\bf j}(x)$,  $x\in \ft U_0$, are singletons, as assumed
in \cite{HeNi1}.
More precisely, the main result of 
\cite{HeNi1} is a particular case of Theorem~\ref{th.thm3'} when $\vert \nabla f\vert \neq 0$ on $\pa\Omega$, except that 
we do not prove in this work the possible existence of a full
asymptotic expansion of the low-lying spectrum of $\Delta_{f,h}^{D}$.\medskip

\noindent
Furthermore, 
our results
have  the advantage to give assumptions on $f$ leading to  sharp asymptotic estimates on the sole $\ft m^*$ smallest eigenvalues $\lambda_{1,h},\dots,\lambda_{\ft m^*,h} $ of $\Delta_{f,h}^{D}$ when $\ft m^*< \ft m_0$, and the 
more $\ft m^*$ is small, the less restrictive are these assumptions. This was not allowed in \cite{HeNi1}. 
In particular, in the case when $\ft m^*=1$ is given a sharp equivalent  of the sole  principal  eigenvalue $\lambda_{1,h}$. 
This is appreciable since it gives the leading term of the semigroup
$(e^{-t \Delta_{f,h}^{D}})_{t\ge 0}$ under very general assumptions.
On this point,  Theorem~\ref{th.thm3}  also generalizes
 \cite[Theorem 3]{DLLN-saddle1}  
when $f$ admits critical points on   $\pa\Omega$. 
To be a little more precise here, we have for example the following corollary of Theorem~\ref{th.thm3}
(see also Remark~\ref{re.cor3} in this connection).  

\begin{corollary}
\label{cor.thm3'}
Let us assume that $f$ is a Morse function, that $\{f<\min_{\pa\Omega}f\}$ is non empty, connected, 
 contains all the local minima of $f$ in $\Omega$,
and that 
$$
\overline{\{f<\min_{\pa\Omega}f\}}\cap \pa\Omega
\ =\ \{z_{1},\dots,z_{N}\}\,,$$
where $N\in\mathbb N^{*}$ and, for $k\in\{1,\dots,N\}$, $z_{k}$ is a saddle point of $f$ such that 
$\ft n_{\Omega}(z_{k})$ 
 is an eigenvector of $\Hess  f(z_{k})$
associated with its unique negative eigenvalue $\lambda(z_{k})$.
The principal eigenvalue of $\Delta_{f,h}^{D}$ then satisfies the following Eyring-Kramers
formula in the limit $h\to0$: 
\begin{equation}
\label{eq.sharp1}
\lambda_{1,h}\ =\    \frac 2\pi \,\frac{\sum_{k=1}^{N}\vert \lambda(z_{k})\vert \, \vert   \det\Hess f(z_{k})  \vert   ^{-\frac 12}  }{   \sum \limits_{y\in \argmin_{ \Omega}f }     \big(  \det\Hess f(y)\big)^{-\frac 12}      } \, h \,e^{-\frac 2h   (\min_{\pa\Omega}f - \min_{\Omega}f)}\,\big(1+O(\sqrt h)\big)\,.
\end{equation}
\end{corollary}

\noindent
Let us also mention the work \cite{michel2017small}, where the author
treats the case of general  Morse functions in the case without boundary.
We believe that the analysis done in \cite{michel2017small} can 
be adapted to our setting, which would lead to
the existence of an Eyring-Kramers type formula 
for each small eigenvalue of $\Delta_{f,h}^{D}$
under the sole assumptions of  Theorem~\ref{th.thm2'}.
Nevertheless, we made the choice to not follow this way
here
since these precise formulas are in general very complicated to make explicit.
Indeed, the  pre-exponential factors are
not computed in general in \cite{michel2017small}, 
but are shown to be computable
by following an arbitrary long algorithm.
This follows from the fact that in the general case, some tunneling effect
 between the characteristic wells of $f$  mixes
 their corresponding pre-exponential factors, see \cite{michel2017small} for more details. 
Our hypotheses remain however  very general and  lead to explicit Eyring-Kramers type formulas in Theorem~\ref{th.thm3}.




\subsection{Strategy and organization of the paper}
\label{sec.strategy}

In  works such as \cite{HKN, HeNi1,Lep,herau-hitrick-sjostrand-11,michel2017small,DLLN-saddle1},
a part of the analysis relies on the construction of
$0$-forms (i.e. functions) quasi-modes
supported in some characteristic wells of the potential $f$
and of
$1$-forms quasi-modes supported near the saddle points of $f$, and,
in \cite{HeNi1,Lep,DLLN-saddle1},  near its so-called generalized saddle points
on the boundary. Very accurate   WKB approximations of  these local $1$-forms quasi-modes 
then finally
lead 
to the asymptotic expansions of the low-lying spectrum of the Witten Laplacian acting
on functions.
This approach is based on the supersymmetric structure of the latter  operator,
once  restricted to the interplay between $0$- and $1$-forms.

 \medskip
 
\noindent
Near the generalized saddle points on the boundary as considered in~\cite{HeNi1,Lep}, where one recalls that
$\vert \nabla f\vert \neq 0$ there and  actually where  the normal derivative $\pa_{\ft n_\Omega}f$ does not vanish, 
this construction  means solving non characteristic transport equations
with prescribed initial boundary conditions, see in particular \cite{HeNi1,Lep-WKB,Lep}.
Near a usual saddle point $z$ in $\Omega$
(i.e. a critical point $z$ with index $1$), this construction follows from
the work \cite{HeSj4} of Helffer-Sj\"ostrand
and means solving transport equations which are degenerate at $z$ (see in particular Section~2 there).
In this case, the problem is well-posed
only
for prescribed initial condition at the single point $z$.
In particular, when one drops the assumption \eqref{eq.PCA} and $z$ is a usual saddle point which belongs to $\pa\Omega$,  the corresponding  transport equations, which are
the same as for interior saddle points, are uniquely solved as in \cite{HeSj4}, 
but the resulting WKB ansatz does      not  in general satisfy  the required boundary conditions,
except its leading term when the boundary~$\pa\Omega$  has a specific shape near $z$.
To be more precise,
and to make the connection with the hypotheses
of Theorems~\ref{th.thm2'} and \ref{th.thm3'} (and Theorems~\ref{th.thm2} and \ref{th.thm3}),
the leading term of this WKB ansatz 
satisfies the required boundary conditions 
if and only if $\pa\Omega$ coincides near $z$
with the stable manifold of $z$ for the dynamics $\dot{X}=-\nabla f(X)$
(see~\eqref{eq.stable} in Section~\ref{se.count}).
This compatibility condition imposes in particular that
$\ft n_{\Omega}(z)$ spans the negative direction of
$\Hess  f(z)$. 
The fact that the remaining part of the WKB 
ansatz does in general not satisfy 
 the required boundary conditions for a compatible boundary $\pa\Omega$ arises
 from the curvature of this boundary.\\
 
 \noindent
 The above considerations show that, when $z\in\pa\Omega$ is a saddle point of $f$ and   $\ft n_{\Omega}(z)$ does not span the negative direction of
$\Hess  f(z)$, the classical WKB ansatz constructed near $z$ will not be an accurate approximation 
of the  local $1$-form quasi-mode associated with $z$. They also imply
that the potential 
existence of  full asymptotic expansions of
the small eigenvalues of $\Delta_{f,h}^{D}$
will in general not follow from the existence of these WKB ans\"atze
when $f$ admits saddle points on the boundary.
 Moreover, we expect that sharp asymptotic equivalents such as \eqref{eq.sharp1} are not valid in general  when $\ft n_{\Omega}(z)$ does not span the negative direction of
$\Hess  f(z)$ at the relevant saddle points $z\in\pa\Omega$. In the latter case,  we expect that the corresponding possible sharp asymptotic equivalents should also rely on the angle between    $\ft n_{\Omega}(z)$ and the negative direction of $\Hess  f(z)$. 
 \medskip


\noindent
In this work, we follow a different strategy based on  the constructions of very  accurate quasi-modes for 
$\Delta_{f,h}^{D}$. This approach, which is partly inspired
by the quasi-modal construction made in \cite{DiLe} (see also~\cite{BEGK,landim2017dirichlet,nectoux2017sharp}),
requires  a  careful construction of these functions quasi-modes around
the relevant (possibly generalized) saddle points $z$ of $f$, whereas
these points were not in the supports of the corresponding quasi-modes 
constructed in \cite{HKN, HeNi1,Lep,herau-hitrick-sjostrand-11,michel2017small,DLLN-saddle1}.
 One advantage of this method  
 is to avoid a careful study of the Witten Laplacian  acting on $1$-forms
 near the boundary $\pa\Omega$,
 which would finally lead to  more stringent hypotheses
 on $f$ and  on $f|_{\pa \Omega}$, that is precisely
 to the hypotheses made in the statement of 
 Theorem~\ref{th.thm2'}\footnote{ 
 For example,
 in the statement of Corollary~\ref{cor.thm3'},
 the ``$1$-form approach'' would require that {\em all} the local minima of
$f|_{\pa \Omega}$ are non degenerate
 and that $\ft n_{\Omega}(z)$ spans the negative direction of
$\Hess  f(z)$ at {\em any} saddle point $z\in\pa\Omega$.}
\medskip

\noindent
The rest of the paper is organized as follows. In Section~\ref{se.count},
we prove Theorem~\ref{thm.count}
about the number of small eigenvalues of
$\Delta_{f,h}^{D}$. This is done using  spectral and localization
arguments.
Then, in Section~\ref{sec.association}, we construct
the map ${\bf j}$ characterizing the relevant wells of the potential
function $f$. This permits to construct   very accurate quasi-modes
in Section~\ref{se.quasi} and then to state and prove
our main results, namely Theorems~\ref{th.thm2} and 
\ref{th.thm3}, in Section~\ref{sec.eigenvalues}.
As in \cite{HKN, HeNi1,Lep,herau-hitrick-sjostrand-11,michel2017small,DLLN-saddle1},
the analysis of
the precise asymptotic behaviour of the low-lying spectrum of 
$\Delta_{f,h}^{D}= d^{D,*}_{f,h}d^{D}_{f,h}$ 
is finally reduced to the computation of the small singular values
of  $d^{D}_{f,h}$.

\section{On the number of small eigenvalues  of $\Delta_{f,h}^{D}$}
\label{se.count}

This section is dedicated to the proof of Theorem~\ref{thm.count}. 
Before going into its proof, we briefly recall
basic facts about smooth Morse functions on a $C^{\infty}$ compact Riemannian manifold 
with boundary
$\overline \Omega=\Omega\cup \pa\Omega$ of dimension $d$.
 \medskip
 
\noindent
Let $z\in \mathbb \pa \Omega$. Let us consider  a neighborhood  $\ft V_z$ of $z$ in $\overline \Omega$ and a coordinate system $
p\in \ft V_z \mapsto x=(x',x_d)\in \mathbb R^d_-=\mathbb R^{d-1}\times \mathbb R_-$
  such that:   $x(z)=0$,  $
\{p\in \ft V_z, \, x_d(p)<0\}=  \Omega\cap  \ft V_z $ and $\{p\in \ft V_z,\,  x_d(p)=0\}=\pa \Omega \cap \ft V_z$. 
By definition, the function $f$ is    $C^\infty$ on $\ft V_z$ if, in the $x$-coordinates,  the function $f:x(\ft V_z)\to \mathbb R$   is the restriction of a     $C^\infty $ function   defined on an open subset~$\ft O$ of $\mathbb R^d$ containing $x(\ft V_z)$.
Moreover,
 $z\in \pa \Omega$ is a non degenerate critical point of
 $f:\overline \Omega \to \mathbb R$ of index $p\in \{0,\ldots,d\}$  if it is a non degenerate critical point of index $p$  for this extension. Notice that this definition is independent of the choice of the extension.
A~$C^\infty$ function $f:\overline \Omega\to \mathbb R$ is then
said to be a Morse function if all its critical points in~$\overline \Omega$ are non degenerate. In the following, we 
will also say that $z\in\overline \Omega$ is a saddle point of the Morse function $f$ if it is a critical point of $f$ with index $1$.\medskip

\noindent
Let now $f:\overline \Omega\to \mathbb R$ be a Morse fonction. By the above,
there exist a $C^{\infty}$   Riemannian manifold $\widetilde \Omega$ (without boundary) of dimension $d$
and a $C^{\infty}$ Morse  function   $\tilde f:\widetilde \Omega\to \mathbb R$ such that
$$ \tilde f|_{\overline\Omega}=f
\quad\text{and}\quad 
 \overline\Omega\ \subset\ \widetilde\Omega 
 \,.$$
For a critical point $z\in\overline \Omega$ of $\tilde f$,
the sets~$\ft W^+(z)$ and~$\ft W^-(z)$ will  respectively denote the so-called stable and unstable manifolds of $z$ for  the dynamics $\dot X=-\nabla \tilde f(X)$ in $\widetilde\Omega$. In other words, denoting by $X_y(t)$ the solution to $\frac{d}{dt}X_y(t)=-\nabla \tilde f( X_y(t))$ with initial condition $X_y(0)=y\in  \widetilde\Omega$, one has (see for example \cite[Definition~7.3.2]{Jost}):
\begin{equation}
\label{eq.stable}
\ft W^{\pm}(z)\ =\ \{y\in \widetilde\Omega\ \ \text{s.t.\ \ $X_{y}(t)\in \widetilde\Omega$\  for every $\pm t\geq 0$ \  and} \ \lim_{t\to \pm \infty}X_y(t)= z\}.
\end{equation}
We recall that when $z$ has index $p\in \{0,\ldots,d\}$, the sets $\ft W^+(z)$ and~$\ft W^-(z)$
are indeed  smooth  submanifolds of $\widetilde \Omega$; they moreover intersect orthogonally at $z$  and have respective dimensions $d-p$ and $p$
(see for example \cite[Theorem 7.3.1 and Corollary~7.4.1]{Jost}). Note lastly that the part of $\ft W^{\pm}(z)$
leaving outside $\overline \Omega$ of course depends on the choice of the extension $\tilde f$.

\subsection{Preliminary results} 

In order to prove Theorem~\ref{thm.count}, one will make use of the following proposition
which results 
 from~\cite[Th\'eor\`eme 1.4]{HeSj4}. 
   
\begin{proposition}\label{pr.hs4} 
Let $\overline{\ft O}$ be an 
oriented $C^{\infty}$ compact and  connected  Riemannian manifold of dimension $d$ with 
interior $\ft O$ and non empty
boundary
$\pa\ft O$,
let 
$\phi: \overline{\ft O}\to \mathbb R $
be a $C^\infty$ Morse function, and let $x_{0}$
be a critical point of $\phi$ in $\ft O$ with index  $\ell\in \{0,\ldots,d\}$
such that $x_0$ is the only critical point of $\phi$ in  $\overline{\ft O}$.
Then,
the Dirichlet realization $\Delta_{\phi,h}^D(\ft O)$ of the Witten Laplacian 
acting on functions on $\ft O$
satisfies the following estimate:
 there exist $\eta_0>0$   and $h_0>0$ such that for all $h\in (0,h_0)$, 
$${\rm dim \, Ran }\,  \pi_{[0,\eta_0 h]}\big(\Delta_{\phi,h}^D(\ft O)  \big)=\delta_{\ell,0}.$$
\end{proposition}

%

\noindent
The following result is a direct consequence of Proposition~\ref{pr.hs4}. 
\begin{corollary}  \label{coropsij}
Let  $\overline{\ft O}$, $\phi$, $x_{0}$, and $\ell\in \{0,\ldots,d\}$  be as in Proposition~\ref{pr.hs4}. 
Let us assume that $\ell=0$, i.e. that $x_{0}$ is a local minimum of $\phi$ in $\ft O$, and that $\phi$ 
only attains its minimal value on $\overline{\ft O}$ at $x_0$. Let
moreover, for every $h$ small enough,
 $\Psi\ge 0$ be the  $L^2(\ft O)$-normalized  eigenfunction  of  $\Delta_{\phi,h}^D(\ft O)$ associated with its unique eigenvalue $\lambda_{h}$ in $(0,\eta_0 h]$ (see Proposition~\ref{pr.hs4} and Remark~\ref{re.elli}). 
 Lastly, let $\xi\in C^\infty_c(\ft O,[0,1])$ be a cut-off function 
    such that $\xi=1$ in a neighborhood  of $x_{0}$ in $\ft O$.
    Then, 
 defining 
 $\chi:=\frac{ \xi\,  e^{-\frac1h  \phi } }{ \big \Vert\xi \, e^{-\frac1h  \phi }  \big \Vert_{L^2(\ft O)}   }$,  
     there exists $c>0$ such that   for every $h$ small enough:
\begin{equation}
\label{estu}
 \Psi=\chi + O\big (e^{-\frac ch}\big ) \ \text{ in } \  L^2(\ft O)
 \quad\text{and}
 \quad
 0<\lambda_{h}\le 
  \big \Vert d_{\phi,h}  \chi  \big\Vert_{ \Lambda^1L^2(\ft O)}^2  \le e^{-\frac ch}\,.
\end{equation}
\end{corollary}

\begin{proof} 
The proof of~\eqref{estu} is standard but we give it for the sake of completeness. 
As in the statement of Corollary~\ref{coropsij},
let us define 
$$\chi:=\frac{ \xi\,  e^{-\frac1h  \phi } }{ \big \Vert\xi \, e^{-\frac1h  \phi }  \big \Vert_{L^2(\ft O)}   }\,.$$
From the definition of $\xi$ and the Laplace method together with the fact that $\phi$ only attains its minimal value on $\overline{\ft O}$ at $x_{0}$, it holds
$$ \big \Vert \xi\,  e^{-\frac1h  \phi }  \big \Vert_{L^2(\ft O)}^2=\frac{(\pi h)^{\frac d2}}{\sqrt{\det\Hess \phi(x_{0})} } e^{-\frac 2h \phi(x_{0})} \big(1+O(h)\big).$$
According to Proposition~\ref{pr.hs4}, there exist $\eta_0>0$ and $h_0>0$ such that for all $h\in (0,h_0)$,  $\pi_{[0,\eta_0 h]}\big(\Delta_{f,h}^D(\ft O)  \big)$ is the orthogonal projector on $\sspan \{\Psi\}$.  Moreover, using 
the following spectral estimate, valid 
for any 
 nonnegative self-adjoint operator $(T,D(T))$  on a Hilbert space $\left(\mc H, \Vert\cdot\Vert\right)$
 with  associated quadratic form  $(q_T,Q(T))$,
\begin{equation}
\label{quadra}
\forall b >0\,, \ \forall u\in Q\left (T\right )\,,\ 
\left\Vert \pi_{[b,+\infty)} (T) \, u\right\Vert^2 \leq \frac{q_T(u)}{b},
\end{equation}
it holds (see \eqref{eq.qfh'} and \eqref{eq.qfh})
$$\Big\Vert  \chi-\pi_{[0,\eta_0 h]}\big(\Delta_{\phi,h}^D(\ft O) \big) \chi    \Big \Vert_{  L^2(\ft O)}^2\le 
\frac{\big \Vert d_{\phi,h} \chi  \big\Vert_{ \Lambda^1L^2(\ft O)}^2 }{\eta_0 h}= \frac{h}{\eta_0 } \frac{ \int_{\ft O} \vert d \xi\vert^2 e^{-\frac 2h \phi}  }{   \big \Vert\xi \, e^{-\frac1h  \phi }  \big \Vert_{L^2(\ft O)}^2} .$$
Hence, since $\xi=1$ in a neighborhood of $x_{0}$ and thus, for some $c>0$, 
$\phi(y)\ge \phi (x_{0}) +c$ for   every 
$y\in \supp d\xi$, one has for every  $h>0$ small enough,
\begin{equation}\label{eq.psi-00}
\big \Vert d_{\phi,h} \chi  \big\Vert_{ \Lambda^1L^2(\ft O)}^2
\le e^{-\frac ch}
\quad\text{and}
\quad
\Big\Vert \chi  -\pi_{[0,\eta_0 h]}\big(\Delta_{\phi,h}^D(\ft O)  \big)  \chi      \Big \Vert_{ L^2(\ft O)}^2\le e^{-\frac ch},
\end{equation}
where $c>0$ is independent of $h$. 
 Since $\Vert\chi  \Vert_{L^2(\ft O)} =1$, 
the first relation in \eqref{eq.psi-00}
together with
the Min-Max principle 
leads to (see \eqref{eq.qfh'})
$$
\lambda_{h}\ \leq\ 
\langle \Delta_{\phi,h}^D(\ft O) \chi,\chi \rangle_{L^{2}(\ft O)}
\ =\
\big \Vert d_{\phi,h} \chi  \big\Vert_{ \Lambda^1L^2(\ft O)}^2 
\  \le\  e^{-\frac ch}\,.
$$
Moreover,
  using the second relation in \eqref{eq.psi-00} and  the Pythagorean theorem, one obtains for 
 every $h>0$ small enough: 
\begin{equation}\label{eq.psi-001}
\big\Vert  \pi_{[0,\eta_0 h]}\big(\Delta_{\phi,h}^D(\ft O)  \big)  \chi      \big \Vert_{ L^2(\ft O)}=1+O( e^{-\frac ch}).
\end{equation}
In conclusion, from \eqref{eq.psi-00}, \eqref{eq.psi-001}, and  since $\chi$ and  $\Psi$
are nonnegative,   it holds, in $L^2(\ft O)$,
for
some $c>0$ and every $h>0$ small enough: 
\begin{align*}
\Psi&= \frac{\pi_{[0,\eta_0 h]}\big(\Delta_{\phi,h}^D(\ft O)  \big)  \chi }{\big\Vert  \pi_{[0,\eta_0 h]}\big(\Delta_{\phi,h}^D(\ft O)  \big)  \chi      \big \Vert_{ L^2(\ft O)}} =\chi + O\big (e^{-\frac ch}\big ).
\end{align*}
This concludes the proof of~\eqref{estu}
and then the proof of  Corollary~\ref{coropsij}.
\end{proof}

\noindent
We are now in position to prove Theorem~\ref{thm.count}. 

\subsection{Proof of  Theorem~\ref{thm.count}}

\noindent
Let $\{x_1,\ldots,x_n\}$ be the set of the critical points of $f$ in $\overline \Omega$, i.e.  
$$ \big \{x_1,\ldots,x_n \big \}\ =\  \big \{x\in \overline \Omega, \,\,  \vert \nabla f(x)\vert =0 \big \}\,.$$
From the preliminary discussion in the beginning of Section~\ref{se.count},
there exist an oriented $C^{\infty}$ compact and  connected  Riemannian manifold $\overline{\widetilde \Omega}$ of dimension $d$ with interior $\widetilde \Omega$ and boundary $\partial \widetilde\Omega$, and 
a $C^\infty$ Morse function   $\tilde f:\overline{\widetilde \Omega}\to \mathbb R$ 
such that
$$ \tilde f|_{\overline\Omega}=f\,,\ \  \overline\Omega\ \subset\ \overline{\widetilde\Omega} \quad\text{and}\quad
\big \{x_1,\ldots,x_n \big \}\ \subset\ \widetilde\Omega\,.$$
 We recall that $\ft m_0$ denotes the number of local minima of $f $ in $\Omega$ (see Definition~\ref{de.U0}), and thus
that $0\le \ft m_0\le n$. When $\ft m_0>0$, the elements $x_{1},\dots,x_{n}$
are moreover ordered such that
$$
 \big \{x_1,\ldots,x_{\ft m_{0}} \big \}\ =\ \ft U_{0}.
$$
In addition, one introduces for every $j\in\{1,\dots,\ft m_{0}\}$
a smooth open neighborhood $\ft O_{j}$ of $x_{j}$
such that  $\overline{\ft O_{j}}\subset \Omega$
and such that $x_{j}$ is the only critical point
of $f$ in $\overline{\ft O_{j}}$ as well as the only point where 
$f$ attains its minimal value  in $\overline{\ft O_{j}}$.
Similarly,
when $x_{j}\in \Omega$ is not a local minimum of $f$, one introduces
a smooth open neighborhood $\ft O_{j}$ of $x_{j}$
such that  $\overline{\ft O_{j}}\subset \Omega$
and such that $x_{j}$ is the only critical point
of $f$ in $\overline{\ft O_{j}}$.
Lastly, when $x_{j}\in \pa\Omega$,
one now introduces
a smooth open neighborhood $\ft O_{j}$ of $x_{j}$
in $\overline {\widetilde\Omega}$
such that  $\overline{\ft O_{j}}\subset \widetilde\Omega$
and such that $x_{j}$ is the only critical point
of $\tilde f$ in $\overline{\ft O_{j}}$.
When such a $x_{j}$ is a local minimum of $f$,  the set
$\ft O_{j}$ is moreover chosen small enough such that
the minimal value of $\tilde f$ in $\overline{\ft O_{j}}$
is only attained at $x_{j}$.
Let us also introduce   a quadratic  partition of unity $(\chi_j)_{j\in \{1,\ldots,n+1\}}$ such that:
\begin{enumerate}
\item  For all $j\in \{1,\ldots,n+1\}$, $\chi_j\in C^\infty(\overline{ \widetilde\Omega},[0,1])$ and
$\sum_{j=1}^{n+1} \chi_j^2=1$ on $\overline{ \widetilde\Omega}$.
\item For all $j\in \{1,\ldots,n\}$,  $\chi_{j}=1$ near $x_{j}$ and 
$\supp\chi_{j}\subset {\ft O}_{j}$. In particular,
$\supp\chi_{j}\subset \Omega$ when $x_{j}\in\Omega$.
\item  For all $(i,j)\in \{1,\ldots,n\}^2$,  $i\neq j$ implies $\supp \chi_{i}  \cap \supp \chi_{j} =\emptyset$.\end{enumerate}
\noindent
In the following, we will also use 
 the so-called IMS localization formula (see for example \cite{CFKS}):    for all $\psi \in H^1_0(\Omega)$, it holds
  \begin{align}
\label{eq.IMS}
Q_{f,h}(   \psi ) &=\sum_{j=1}^{n+1} Q_{f,h}( \chi_j\, \psi)    -\sum_{j=1}^{n+1}  h^2\,\big \Vert  | \nabla \chi_j|\,\psi   \big \Vert_{L^2(\Omega)}^2\,,
\end{align}
where $Q_{f,h}$ is the quadratic form defined in \eqref{eq.qfh}.\\
 

\noindent
\textbf{Step 1.} Let us first show that there exists $c_{0}>0$ such that
for every $h$ small enough, it holds
\begin{equation}
\label{eq.easy}
\dim \Ran \, \pi_{(0,e^{-\frac{c_0}{h}})}\big (\Delta_{f,h}^{D}\big )
\ \geq\ \ft  m_0\,.
\end{equation}
This relation is obvious when $\ft  m_0=0$.
When $\ft m_{0} >0$, 
the family $(\overline{\ft O_{j}}, f|_{\overline{\ft O_{j}}},x_{j} )$
satisfies, for every $j\in\{1,\dots,\ft m_{0}\}$,
the hypotheses of 
Corollary~\ref{coropsij}. Then, according to \eqref{estu}, the function
$$\psi_{j}:= \frac{\chi_{j} e^{-\frac fh}}{\|\chi_{j} e^{-\frac fh}\|_{L^{2}(\ft O_{j})}}$$
satisfies, for some $c_{j}>0$ and every $h>0$ small enough (see \eqref{eq.qfh}),
$$
Q_{f,h}(\psi_{j})\leq e^{-\frac {c_{j}}h}\,.
$$
Since the $\psi_{j}$'s, $j\in\{1,\dots,\ft m_{0}\}$, 
are unitary in $L^{2}(\Omega)$ and have disjoint supports, 
it follows from the Min-Max principle that $\Delta_{f,h}^{D}$ admits at least $\ft m_0$ exponentially small eigenvalues when $h\to 0$, which proves \eqref{eq.easy}.



\medskip

\noindent
\textbf{Step 2.} Let us now show that there exists $c'_{0}>0$ such that
for every $h$ small enough,  it holds
\begin{equation}
\label{eq.easy'}
\dim \Ran \, \pi_{[0,c'_{0}h]}\big (\Delta_{f,h}^{D}\big )
\ \leq\ \ft  m_0\,.
\end{equation}
According to the Min-Max principle, it is sufficient to show that
there exist $h_{0}>0$ and $C>0$ such that for every $h\in(0,h_{0}]$,
there exist $u_{1},\dots,u_{\ft m_{0}}$ in $L^{2}(\Omega)$
such that for any 
$\psi \in D(Q_{f,h})=  H^1_0(\Omega)$, it holds
\begin{equation}
\label{eq.MinMax}
Q_{f,h}(\psi)\ \geq\ C h\,\|\psi\|^{2}_{L^{2}(\Omega)}-\sum_{i=1}^{\ft m_{0}}\langle \psi,u_{i} \rangle^{2}_{L^{2}(\Omega)}\,.
\end{equation}

\noindent
{\bf Analysis on $\supp \chi_{n+1}$.}\\[0.1cm]
\noindent Since $\supp \chi_{n+1}\cap \overline \Omega$ does not meet  
$\{x_{1},\dots,x_{n}\}$, there exists $C>0$ such that
$|\nabla f|\geq 3C$ on $\supp \chi_{n+1}\cap \overline \Omega$.
It then follows from~\eqref{eq.qfh} that there exists $C>0$ such that for every $h$ small enough
and for every $\psi \in H^{1}_{0}(\Omega)$, it holds
\begin{align}
\nonumber
Q_{f,h}(\chi_{n+1}\psi)&\ \ge\   \big  \lp \chi_{n+1} \psi, \big(\vert \nabla f\vert^2 +h\Delta_H f\big) \chi_{n+1}\psi \big \rp_{L^2(\Omega)}\\
 \label{eq.MinMax1}
&\ \ge\  2C\, \Vert \chi_{n+1}\psi\Vert_{L^2(\Omega)}^2.
\end{align}

\noindent
{\bf Analysis on $\supp \chi_{j}$, $j\in \{1,\dots,\ft m_{0}\}$.}\\[0.1cm]
\noindent
We assume here that $\ft m_{0}>0$. We recall that 
for every $j\in \{1,\dots,\ft m_{0}\}$, 
$(\overline{\ft O_{j}}, f|_{\overline{\ft O_{j}}},x_{j} )$ satisfies the hypotheses of
Corollary~\ref{coropsij}, and we denote, for $h>0$, by $\Psi_{j}\geq 0$
the~$L^2(\ft O_{j})$-normalized  eigenfunction  of  $\Delta_{f,h}^D(\ft O_{j})$ associated with its principal
eigenvalue~$\lambda^{j}_{h}$ (which is positive, and exponentially small when $h\to 0$).
It then follows from 
Proposition~\ref{pr.hs4} and
Corollary~\ref{coropsij} that for some $C>0$ and every $h>0$ small enough, it holds,
for every $j\in\{1,\dots,\ft m_{0}\}$ and
for every $\psi\in  H^1_0(\Omega)$,
\begin{align}
\nonumber
Q_{f,h}(\chi_{j}\psi)&\ \geq\ 
\lambda^{j}_{h}\langle \chi_{j} \psi, \Psi_{j} \rangle^{2}_{L^{2}(\Omega)}
+ 2Ch\|\chi_{j}\psi- \langle \chi_{j} \psi, \Psi_{j} \rangle \Psi_{j} \|_{L^2(\Omega)}^{2} \\
\nonumber
&\ \geq\ 
 2Ch\,\|\chi_{j}\psi\|_{L^2(\Omega)}^{2}-  2Ch\,\langle \chi_{j} \psi, \Psi_{j} \rangle_{L^2(\Omega)}^{2}\\
 \label{eq.MinMax1'}
 &\ = \ 2Ch\,\|\chi_{j}\psi\|_{L^2(\Omega)}^{2}-  \langle  \psi, u_{j} \rangle_{L^2(\Omega)}^{2}\,,
\end{align}
where one has defined $u_{j}:=\sqrt{2Ch}\, \chi_{j} \Psi_{j}$.
 \medskip

\noindent
{\bf Analysis on $\supp \chi_{j}$, when $x_{j}\in \Omega$ is not a local minimum of $f$.}\\[0.1cm]
\noindent In this case,
applying Proposition~\ref{pr.hs4} with $\ft O_{j}$  and 
$\Delta_{f,h}^D(\ft O_{j})$, it follows that 
for some $C>0$ and every $h>0$ small enough, it holds,
for every $\psi\in  H^1_0(\Omega)$,
\begin{align}
\label{eq.MinMax2}
Q_{ f,h}(\chi_{j}\psi)
&\ \geq\ 
 2Ch\,\|\chi_{j}\psi\|_{L^2(\Omega)}^{2}\,.
\end{align}

\noindent
{\bf Analysis on $\supp \chi_{j}$, when $x_{j}\in \pa\Omega$ is not a local minimum of $f$.}\\[0.1cm]
\noindent In this case,
applying as previously Proposition~\ref{pr.hs4} with $\ft O_{j}$  but here with
$\Delta_{\tilde f,h}^D(\ft O_{j})$ and denoting by 
$Q_{\tilde f,h,\ft O_{j}}$ its associated quadratic form, it follows that 
for some $C>0$ and every $h>0$ small enough, it holds,
for every $ \psi\in  H^1_0(\widetilde\Omega)$,
\begin{align*}
Q_{\tilde f,h,\ft O_{j}}(\chi_{j}\psi)
\ =\ 
\big \Vert d_{\tilde f,h} \chi_{j} \psi  \big\Vert_{ \Lambda^1L^2(\ft O_{j})}^2
&\ \geq\ 
 2Ch\,\|\chi_{j}\psi\|_{L^2(\widetilde\Omega)}^{2}\,.
\end{align*}
Let us now consider the application $\psi \in L^{2}(\Omega)\mapsto \overline{\psi}\in L^{2}(\widetilde \Omega)$,
where $\overline{\psi}$ extends $\psi$
on $\widetilde \Omega$  by $\overline{\psi}|_{\widetilde \Omega \setminus \overline\Omega} =0$.
Since  $\overline{\psi}$
belongs to $H^1_0(\widetilde\Omega)$ for every
$\psi\in H^1_0(\Omega)$ with moreover $(d \overline{\psi})|_{\widetilde \Omega \setminus \overline\Omega}=0$, 
it holds, for every $h$ small enough and for every $\psi\in H^1_0(\Omega)$,
\begin{align}
\nonumber
Q_{f ,h}(\chi_{j}\psi)
\ =\ \big \Vert d_{ f,h}  (\chi_{j}\psi) \big\Vert_{ \Lambda^1L^2(\Omega)}^2
&\ =\ \big \Vert d_{\tilde f,h}  (\chi_{j}\overline{\psi})  \big\Vert_{ \Lambda^1L^2(\ft O_{j})}^2\\
&\ \geq\ 
 2Ch\,\|\chi_{j}\overline{\psi}\|_{L^2(\widetilde\Omega)}^{2}
 \label{eq.MinMax3}
 \ =\ 2Ch\,\|\chi_{j}\psi\|_{L^2(\Omega)}^{2}\,.
\end{align}

\noindent
{\bf Analysis on $\supp \chi_{j}$, when $x_{j}\in \pa\Omega$ is a local minimum of $f$.}\\[0.1cm]
 \noindent
Let us now consider, as previously, the extension map $\psi \in H^{1}_{0}(\Omega)\mapsto \overline{\psi}\in H^{1}_{0}(\widetilde \Omega)$ by $0$ outside $\overline\Omega$,
and
let $\Psi_j\ge 0$ be the  $L^2(\ft O_j)$-normalized  eigenfunction  of  $\Delta_{ \tilde f,h}^D(\ft O_j)$ associated with its principal eigenvalue $\lambda_h^j$ (see Remark~\ref{re.elli}). 
Then, according to Proposition~\ref{pr.hs4},
one has for some $C>0$, for every $h$ small enough, and for every $\psi\in H^{1}_{0}(\Omega)$,
\begin{align}
\nonumber
Q_{f,h}(\chi_{j}\psi)&\ =\
Q_{\tilde f,h,\ft O_j}(\chi_{j}\overline{\psi})\\
 \nonumber
&\ \geq\ 
\lambda^{j}_{h}\langle \chi_{j} \overline{\psi}, \Psi_{j} \rangle^{2}_{L^{2}(\ft O_j)}
+ 6Ch\|\chi_{j}\overline{\psi}- \langle \chi_{j} \overline{\psi}, \Psi_{j} \rangle \Psi_{j} \|_{L^2(\ft O_j)}^{2} \\
\nonumber
&\ \geq\ 
 6Ch\,\|\chi_{j}\overline{\psi}\|_{L^2(\ft O_j)}^{2}-  6Ch\,\langle \chi_{j} \overline{\psi}, \Psi_{j} \rangle_{L^2(\ft O_j)}^{2}\\
\label{eq.eqtilde}
 &\ = \ 6Ch\,\|\chi_{j}\psi\|_{L^2(\Omega)}^{2}-  6Ch\,\langle \chi_{j} \psi, \Psi_{j} \rangle_{L^2(\Omega\cap \ft O_j)}^{2}\,.
\end{align}
Moreover, applying 
Corollary~\ref{coropsij}
with $\ft O=\ft O_j$, $\phi=\tilde f|_{\overline{\ft O_j}}$, and $\xi=\chi_j$,
it follows from 
\eqref{estu}  that for every $h$ small enough, one has
 $$\big\Vert \Psi_j\big\Vert^2_{L^2(\Omega\cap \ft O_j)} =\frac{\big \Vert  \chi_j e^{-\frac1h  \tilde f }\big \Vert^2_{L^2(\Omega\cap \ft O_j)}   }{ \big \Vert\chi_j e^{-\frac1h \tilde f}  \big \Vert^2_{L^2(\ft O_j)}   }+ O\big (e^{-\frac ch}\big ) .$$
From the Laplace method together with the fact that $\tilde{f}$ only attains its minimal value on $\overline{\ft O_j}$ at $x_j$,  
it then holds in the limit  $h\to 0$:
 $$
\big\Vert \Psi_j\big\Vert^2_{L^2(\Omega\cap \ft O_j)} 
\  =\ \frac{1}{2}+ o(1)\,.$$
According to \eqref{eq.eqtilde}, this implies, using  the Cauchy-Schwarz inequality
$$\langle \chi_{j} \psi, \Psi_{j} \rangle_{L^2(\Omega)}^{2}
\  \le \
 \big \Vert  \chi_j\psi \big \Vert^2_{L^2(\Omega)}
\big \Vert  \Psi_{j} \big \Vert^2_{L^2(\Omega\cap \ft O_j)}\, , $$
that for some $C>0$, for every $h$ small enough, and 
for every $\psi\in H^{1}_{0}(\Omega)$, it holds:
\begin{align}
\label{eq.MinMax4}
Q_{f,h}(\chi_{j}\psi)
 &\ \geq \ 2Ch\,\|\chi_{j}\psi\|_{L^2(\Omega)}^{2}\,.
\end{align}

\noindent
{\bf Conclusion.}\\[0.1cm]
\noindent Adding the estimates \eqref{eq.MinMax1} to \eqref{eq.MinMax3}
and \eqref{eq.MinMax4}, we deduce
from the
IMS localization formula \eqref{eq.IMS}
that there exists $C>0$ such that for every $h$ 
small enough and 
for every $\psi\in H^{1}_{0}(\Omega)$, it holds
\begin{align*}
Q_{f,h}(   \psi ) &\ =\ \sum_{j=1}^{n+1} Q_{f,h}( \chi_j\, \psi)    -\sum_{j=1}^{n+1}  h^2\,\big \Vert  | \nabla \chi_j|\,\psi   \big \Vert_{L^2(\Omega)}^2\\
&\ \geq\ \sum_{j=1}^{n+1} 2Ch \|\chi_{j}\psi\|_{L^2(\Omega)}^{2} - 
\sum_{j=1}^{\ft m_0} \langle  \psi, u_{j} \rangle_{L^2(\Omega)}^{2} + O(h^2) \|\psi\|_{L^2(\Omega)}^{2}\\
&\ \geq\ 
Ch\,\|\psi\|_{L^2(\Omega)}^{2} 
- \sum_{j=1}^{\ft m_0} \langle  \psi, u_{j} \rangle_{L^2(\Omega)}^{2}\,,
\end{align*}
where, for $j\in\{1,\dots,\ft m_0\}$, we recall that $u_{j}=\sqrt{2Ch}\, \chi_{j} \Psi_{j}$.
This implies the relation~\eqref{eq.MinMax}
and then~\eqref{eq.easy'}, which concludes the proof of Theorem~\ref{thm.count}.

\section{Study of the characteristic wells of the function $f$}
\label{sec.association}

In this section, one constructs two maps, $\mbf j$ and $\ft C_{\mbf j}$. The map $\mbf j$ associates  each local minimum of
$f$ in $\Omega$  with a set of  relevant saddle points, here called {\em separating saddle points}, of $f$ in $\overline \Omega$, and the map $\ft C_{\mbf j}$ associates  each local minimum of $f$ in $\Omega$  with a characteristic well, here called a    {\em critical component}, of $f$ in $\Omega$ (see Definition~\ref{de.SSP} below).
Our construction
is strongly inspired by  a similar construction made in \cite{herau-hitrick-sjostrand-11} in the case without boundary, where
the notions of separating saddle point and of critical component were defined
in this setting.
The depths of the wells~$\ft C_{\mbf j}(x)$, $x\in \ft U_0$, which can be expressed in terms 
of~$\mbf j(x)$,
will finally give, up to some multiplicative factor~$-2$, the logarithmic equivalents of 
the small eigenvalues of $\Delta_{f,h}^D$ (see indeed Theorems~\ref{th.thm2'} and~\ref{th.thm2}).
  The maps $\mbf j$ and $\ft C_{\mbf j}$ will also be  used  in the next section to define accurate quasi-modes for $\Delta^D_{f,h}$. 
\medskip

\noindent
This section is organized as follows. 
In Section~\ref{sec.gene-pri}, one defines  the \textit{principal (characteristic) wells} of the function~$f$ in $\Omega$. Then, in  Section~\ref{sec.gene-pri2}, one  defines the   separating saddle points of~$f$ in~$\overline \Omega$ and the {critical components} of $f$. Finally, Section~\ref{sec.j} is dedicated to the  constructions of  the  maps $\mbf j$ and $\ft C_{\mbf j}$. 


\subsection{Principal wells of $f$ in   $\Omega$ }
\label{sec.gene-pri}


 \begin{definition}\label{de.principal-wells}
 Let $f:\overline \Omega\to \mathbb R$ be a  $C^\infty$ Morse function
 such that $\ft U_{0}\neq \emptyset$.
  For all $x\in \ft U_0 $ (see Definition~\ref{de.U0}) and $\lambda>f(x)$, one defines
$$
\ft C(\lambda,x)\ \ \text{as the connected component of $\{f<\lambda\}$ in~$\overline\Omega$
containing $x$} .
$$
Moreover, for every $x\in \ft U_0 $, one defines\label{page.lambdax}
 $$\lambda(x):=\sup\{\lambda>f(x)\, \text{ such that }\  \ft C(\lambda,x)\cap\pa\Omega=\emptyset \} \ \ 
\text{ and }\ \ \ft C(x):=\ft  C(\lambda(x),x).
$$

\end{definition}
\noindent
Since  for every $x\in \ft U_0$, $x$ is a non degenerate local minimum of $f$ in $\Omega$, 
notice that
 the real value  $\lambda(x)$ is well defined and belongs to $(f(x),+\infty)$. 
 The \textit{principal wells} of the function $f$ in $\Omega$
 are then defined as follows.

\begin{definition}  \label{de.1}
Let $f:\overline \Omega\to \mathbb R$ be  a $C^\infty$ Morse function
such that $\ft U_{0}\neq \emptyset$.
The set 
$$\mathcal C=\big \{\ft  C(x),\, x\in \ft U_0 \big \}$$
is called  the set of principal wells of the function $f$  in   $\Omega$. 
The number of principal wells is denoted by 
$$
 \ft N_{1}:={\rm Card} (\mathcal C)  \in  \{1,\dots,\ft m_{0}\}.
 $$
Finally,   the  principal wells of $f$ in    $\Omega$ (i.e. the elements of $\mathcal C$) are denoted by:
$$  \mathcal C=\{\ft C_{1,1},\dots,   \ft  C_{1,  \ft N_{1}}\}.$$ 
\end{definition} 

\noindent
In Remark~\ref{re.denom} below, one explains why the elements of $\mathcal C$ are called the principal wells of $f$ in $\Omega$. 
Notice that they obviously satisfy $\pa \ft  C(x) \subset \{f=\lambda(x)\}$ for every
$x\in \ft U_0$.
These  
 wells
 satisfy moreover the following property.
\begin{proposition}  \label{pr.p1}
Let $f:\overline \Omega\to \mathbb R$ be  a $C^\infty$ Morse function
such that $\ft U_{0}\neq \emptyset$
and
 let $\mathcal C=\{ \ft  C_{1,1},\dots, \ft  C_{ 1,\ft N_{1}}\}$ be the set 
 of its  principal wells defined in Definition~\ref{de.1}. Then, for
 every $k\in \{1,\dots, \ft N_{1}\}$,  it holds:
\begin{equation}\label{eq.ck-omega} 
 \left\{
    \begin{array}{ll}
       \ft C_{1,k} \, \text{ is an open subset of }\,  \Omega \text{, and }  \\
        \text{for all } \ell  \in \{1,\dots, \ft N_{1}\} \text{ with $\ell\neq k$}, \  \ft C_{1,k} \cap \ft C_{1,\ell}=\emptyset.
    \end{array}
\right.
\end{equation}
\end{proposition}
\begin{proof}
The proof of~\eqref{eq.ck-omega} is included in the proof of \cite[Proposition 20]{DLLN-saddle0}. Let us mention that in~\cite[Proposition 20]{DLLN-saddle0}, it is  also assumed   that  $f|_{\pa \Omega}$ is a Morse  function, but   this assumption is not used in the proof of \eqref{eq.ck-omega}  there.
\end{proof}

\subsection{Separating saddle points}
\label{sec.gene-pri2}

\subsubsection{Separating saddle points  of $f$ in $\Omega$}
Before giving the definition of  the separating saddle points  of $f$ in $\Omega$,  let us first   recall the 
local structure of the sublevel sets of $f$ near a point $z\in \Omega$. 

\begin{lemma}\label{le.local-structure}
Let $f:\overline \Omega\to \mathbb R$ be  a $C^\infty$ Morse function,  let  $z\in\Omega$, and let  us recall
that, for $r>0$,  $B(z,r):=\{x\in \overline \Omega \ \text{s.t.} \ |x-z|<r\}$.
For every $r>0$ small enough, the following holds:
\begin{enumerate}
\item When $\vert \nabla f(z)\vert \neq 0$,
the set $\{f<f(z)\}\cap B(z,r)$  is  
connected.
\item When $z$ is a critical point of~$f$ with index~$p\in \{0,\ldots,d\}$, one has: 
\begin{enumerate}
\item  if $p=0 $, i.e. if $z\in \ft U_0$, then $ \{f<f(z)\}\cap B(z,r)=\emptyset$,
\item  if $p=1 $, then $\{f<f(z)\}\cap B(z,r)$ has precisely two connected components,
\item  if $p\ge 2$, then $\{f<f(z)\}\cap B(z,r)$ is connected. 
\end{enumerate} 
\end{enumerate} 
\end{lemma}

\noindent
The notion of {separating saddle point}  of $f$ in $ \Omega$ was introduced in~\cite[Section 4.1]{herau-hitrick-sjostrand-11} for a Morse function on a manifold without boundary.

\begin{definition}\label{de.SSP-omega}
Let $f:\overline \Omega\to \mathbb R$ be  a $C^\infty$ Morse function.
The point $z\in \Omega$ is a separating saddle point of $f$ in $\Omega$  if  it  is a saddle point of~$f$  (i.e. a critical point of~$f$ of index $1$) and if   for every $r>0$ small enough, the two connected components of $\{f<f(z)\}\cap B(z,r)$    are contained in different connected components  of $\{ f< f(z) \}$. The set of 
 {separating saddle points}  of $f$ in $ \Omega$ is denoted by~$\ft U_1^{\ft{ssp}}(\Omega)$. 
\end{definition}

\noindent
With this definition, one has the following result which will be needed later  to construct the maps $\mbf j$ and $\ft C_{\mbf j}$ in Section~\ref{sec.j}.

\begin{proposition}
\label{pr.p2}
Let $f:\overline \Omega\to \mathbb R$ be a $C^\infty$ Morse function
such that $\ft U_{0}\neq \emptyset$.   Let us consider~$\ft C_{1,q}$ for $q\in \{1,\ldots,\ft N_1\}$. The set $\ft C_{1,q}$ and its sublevel sets satisfy the following properties. 

\begin{enumerate}
\item It holds,
\begin{equation}
\label{eq.C-ssp}
  \text{ if } \pa \ft C_{1,q}\cap \pa \Omega =\emptyset\,, \text{ i.e. if $\overline{\ft C_{1,q}}\subset   \Omega$, }  \text{ then }    \pa \ft C_{1,q} \cap \ft U_1^{\ft{ssp}}(\Omega) \neq \emptyset.
\end{equation}  
\item Let  $\lambda_q$ be such that  $\ft C_{1,q}$ is a connected component of $\{f<\lambda_q\}$ (see Definitions~\ref{de.principal-wells} and~\ref{de.1}). 
Let $\lambda\in(\min_{\overline{\ft C_{1,q}}}f,\lambda_q]$ and  $\ft C$ be a 
 connected component of $\ft C_{1,q}\cap \{f<\lambda\}$.
Then, 
$$
\big( \ft C\cap \ft U_1^{\ft{ssp}}(\Omega) \neq \emptyset\big)\ \  \text{ iff }\ \ \ 
\text{$\ft C\cap\ft U_{0}$ contains more than one point}. 
$$
Moreover, let us define 
$$  \sigma:=\max_{y\in \ft C\cap \ft U_1^{\ft{ssp}}(\Omega)}f(y)$$
 with the convention
$\sigma=\min_{\overline{\ft C}} f$ when $\ft C\cap \ft U_1^{\ft{ssp}}(\Omega) = \emptyset$. Then, the following assertions hold.
\begin{itemize}
\item For all $\mu\in(\sigma,\lambda]$, the set $\ft C\cap \{f<\mu\}$ is a connected component  of $\{f<\mu\}$.
\item If $\ft C\cap \ft U_1^{\ft{ssp}}(\Omega) \neq \emptyset$, one has 
$\ft C\cap\ft U_{0}\subset\{f<\sigma\} $
and   the boundary of any of the 
connected components of  $\ft C\cap \{f<\sigma\}$ contains a separating saddle point of~$f$ in~$\Omega$ (i.e. a point in $\ft U_1^{\ft{ssp}}(\Omega) $).
\end{itemize}
\end{enumerate}
\end{proposition}
\begin{proof}
The proof of the first item of Proposition~\ref{pr.p2} is the same as the proof of the last point of \cite[Proposition 20]{DLLN-saddle0}
(see Step~5 there), while
the proof of the second item of Proposition~\ref{pr.p2} is the same as the proof of~\cite[Proposition 22]{DLLN-saddle0}, which  follows from the study of the sublevel sets of a Morse function on a manifold without boundary (since the principal wells $\ft C_{1,k}$'s are included in $\Omega$). Again,  the assumption that  $f|_{\pa \Omega}$  is a Morse  function
made in \cite{DLLN-saddle0}     is not used in
these proofs.
\end{proof}


\subsubsection{Separating saddle points  of $f$ in $\overline \Omega$}

In this section, we specify and extend Definition~\ref{de.SSP-omega}
in our setting by
   taking into account the boundary  of $\Omega$ and the principal wells $\{\ft C_{1},\dots,   \ft  C_{  \ft N_{1}}\}$ of $f$   introduced in Definition~\ref{de.1}.  To this end, 
we first state the following result which  describes the local structure of $f$  near  $\bigcup_{k\in \{ 1,\dots, {\ft  N_{1}}\}}\pa \ft C_{1,k} \cap \pa \Omega$
and which will  be used to  state an additional assumption on $f$, assumption \autoref{H1} below, ensuring that the critical points of $f$ in $\pa \ft C_{1,k} \cap \pa \Omega$ are geometrical   saddle points of $f$ in $\overline \Omega$ (see Remark~\ref{gc} below).

\begin{proposition}\label{pr.assumption}
Let $f:\overline \Omega\to \mathbb R$ be  a $C^\infty$ Morse function
such that $\ft U_{0}\neq \emptyset$. Let  $k\in \{ 1,\dots, {\ft  N_{1}}\}$. 
Then, if $\pa \ft C_{1,k} \cap \pa \Omega\neq \emptyset$, for $z\in \pa \ft C_{1,k} \cap \pa \Omega$ (see Definition~\ref{de.1}), one has: 
\begin{enumerate} 
\item[(a)] If $\vert \nabla f(z)\vert \neq 0$,  then $z$ is a local minimum of $f|_{\pa \Omega}$ and $\pa_{\ft n_{  \Omega}}f(z)>0$.    

\item[(b)] If $\vert \nabla f(z)\vert = 0$, then $z$ is saddle  point of $f$.  
In addition, if  the unit outward normal  vector $\ft n_{  \Omega}(z)$ to  $ \Omega$ at $z$ is an eigenvector of $\Hess  f(z)$ associated with its  negative eigenvalue, then  $z$ is a non degenerate local minimum of $f|_{\pa \Omega}$  (where  $\Hess  f(z)$ denotes the endomorphism of $T_z\overline\Omega$
canonically associated  with the usual  symmetric bilinear form~$\Hess  f(z):T_z\overline\Omega\times T_z\overline\Omega \to \mathbb R$
 via the metric~$g$).
\end{enumerate}

\noindent
Besides,  it holds,
\begin{equation}\label{eq.sspCk}
\text{ for all $\ell\in \{ 1,\dots, {\ft  N_{1}}\}$ with $\ell \neq k$, }\, \overline{\ft C_{1,\ell}}\cap \overline{\ft C_{1,k}}= \pa \ft C_{1,k}\cap \pa \ft C_{1,\ell}\subset \ft U_1^{\ft{ssp}}(\Omega).
\end{equation}

\end{proposition}

 \begin{remark}
 As it will be clear from the proof of Proposition~\ref{pr.assumption}, the fact that~$f:\overline \Omega\to \mathbb R $ is a Morse function   is not needed in the proof of item (a) in  Proposition~\ref{pr.assumption}. 
 \end{remark}

\begin{proof}
Let  $z\in \pa \ft C_{1,k} \cap \pa \Omega$.  Let  $\ft V_z$ be a neighborhood of $z$ in~$\overline \Omega$ and let 
\begin{equation}\label{eq.coo-omega0}
p\in \ft V_z \mapsto x=(x',x_d)\in \mathbb R^{d-1}\times \mathbb R_-
\end{equation}
 be a    coordinate system   such that   $x(z)=0$,  
\begin{equation}\label{eq.coo-omega}
\{p\in \ft V_z, \, x_d(p)<0\}=  \Omega\cap  \ft V_z\ \ \text{ and }\  \  \{p\in \ft V_z,\,  x_d(p)=0\}=\pa \Omega \cap \ft V_z
\end{equation}
and
\begin{equation}\label{eq.coo-omega'}
\forall i,j\in\{1,\dots,d\}\,,\ g_z\big(\frac{\pa}{\pa x_i}(z),\frac{\pa}{\pa x_j}(z)\big)=\delta_{ij}
\quad\text{and}
\quad
\frac{\pa}{\pa x_d}(z) = \ft n_\Omega(z)\,.
\quad
\end{equation}

\noindent
The set $x(\ft V_z)$ is a neighborhood of $0$ in $\mathbb R^{d-1}\times \mathbb R_-$. 
With a slight abuse of notation, the function $f$ in the  coordinates $x$ is still denoted by~$f$.
 The set $x(\ft C_{1,k} \cap \ft V_z)$ is included in $\{x_d<0\}$ since $\ft C_{1,k}\subset \Omega$ (see 
Proposition~\ref{pr.p1}). 
For ease of notation,  the set 
$x(\ft C_{1,k}\cap  \ft  V_z)$ will also be    denoted by $\ft C_{1,k}$.  
Let us now  introduce a~$C^\infty$ extension of~$f: x(\ft V_z)\subset \{x\in \mathbb R^d, \, x_d \le0\}  \to \mathbb R$  to a neighborhood $\ft V_0$ of $0$  in $\mathbb R^d$
{such that $\ft V_0 \cap \{x\in \mathbb R^d, \, x_d \leq 0\}\subset x(\ft V_z) $}. In the following this extension is still  denoted by~$f$.
 {Note that according to \eqref{eq.coo-omega'},
the matrix
$\Hess  f(0)$ is then at the same time the matrix of the symmetric
bilinear form $\Hess  f(z):T_z\overline\Omega\times T_z\overline\Omega\to \mathbb R$
and of its canonically associated (via the metric $g$)
  endomorphism  $\Hess  f(z):T_z\overline\Omega \to T_z\overline\Omega$,
  in the  basis $\big(\frac{\pa }{\pa x_1}(z),\dots, \frac{\pa }{\pa x_d}(z)=\ft n_\Omega(z)\big)$ 
of $T_z\overline\Omega$.}\medskip

\noindent
 Let $r_0>0$ be such that $\{x\in \mathbb R^d,    \, \vert x\vert <r_0\} \subset \ft V_0$ and let $r\in (0,r_0)$. To prove  Proposition~\ref{pr.assumption}, {one will both work with the initial function 
 $f$ and with the above associated function still denoted by $f$,}
\begin{equation}\label{eq.xcoordi}
f:\, x=(x',x_d)\in \ft V_0\subset \mathbb R^d\mapsto f(x)\in \mathbb R.
   \end{equation}
The proof of Proposition~\ref{pr.assumption} is  divided into several  steps. 
 \medskip

\noindent
\textbf{Step 1.} Proof of  item~(a) in Proposition~\ref{pr.assumption}.  Let us assume that  $\vert \nabla f(z)\vert \neq 0$.  
According to~Lemma~\ref{le.local-structure}, for all $r>0$ small enough, the set~$\{x\in \mathbb R^d, \vert x\vert <r \text{ and }   f(x)<f(0)\} $ is connected.
   Let us also notice that it clearly holds 
   $$\emptyset \neq \ft C_{1,k} \cap \{x\in \mathbb R^d,  \vert x\vert <r\} \subset  \{x\in \mathbb R^d,  \vert x\vert <r \text{ and }   f(x)<f(0)\}.$$
Let us now prove that  
    \begin{equation}\label{eq.eq-ck-f<} 
     \{x\in \mathbb R^d,  \vert x\vert <r \text{ and }   f(x)<f(0)\}\subset \{x_d<0\}.
        \end{equation}
 If it is not the case, there exists $y_2\in \{x\in \mathbb R^{d}, \vert x\vert <r \}$ such that $x_d(y_2)\ge 0$ and $f(y_2)<f(0)$. The set~$\{x\in \mathbb R^d, \vert x\vert <r \text{ and }   f(x)<f(0)\} $ is connected and thus, since it is locally path-connected, it is  path-connected. 
Then, let $y_1\in \ft C_{1,k} \cap   \{x\in \mathbb R^d,  \vert x\vert <r\}$ and consider a continuous curve $\gamma:[0,1]\to  \{x\in \mathbb R^d,  \vert x\vert <r \text{ and }   f(x)<f(0)\}$ such that $\gamma(0)=y_1$ and  $\gamma(1)=y_2$. Let us define $t_0:=\inf\{t\ge 0, \, x_d(\gamma(t))\geq0\}$. Since  $x_d(\gamma(0))<0$ and $x_d(\gamma(1))\ge 0$, it holds $t_0 >0$. Then, for all $t\in [0,t_0]$, it holds $x_d(\gamma(t))\le 0$ (with equality if and only if $t=t_0$), $\vert \gamma(t)\vert <r$, and
 $f(\gamma(t))<f(0)$. Therefore,  since by definition   $\ft C_{1,k}$ is a connected component of  $\{q\in \overline \Omega, f(q)<f(z)\}$, 
 it holds 
$\gamma(t_0)\in \ft C_{1,k} \subset \{x_d<0\}$. 
This contradicts $x_d(\gamma(t_0))=0$ and proves~\eqref{eq.eq-ck-f<}.
Hence, since $\ft C_{1,k}$ is a connected component of 
$\{f<f(z)\}$ in $\Omega$ which intersects the connected set 
$p(\{x\in \mathbb R^d,    \vert x\vert <r \text{ and }   f(x)<f(0)\})\subset \Omega$, it holds
%
    \begin{equation}\label{eq.eq-ck0}
\ft C_{1,k} \cap \{x\in \mathbb R^d,  \vert x\vert <r\} =  \{x\in \mathbb R^d,    \vert x\vert <r \text{ and }   f(x)<f(0)\} . 
   \end{equation}
Equations~\eqref{eq.coo-omega} and~\eqref{eq.eq-ck-f<}  imply that $z$ is  a local minimum of $f|_{\pa\Omega}$. Using in addition the fact that $\vert \nabla f(z)\vert \neq 0$, it holds  $ \pa_{\ft n_ {\Omega}}f(z)\neq 0$
and hence $\pa_{\ft n_ {\Omega}}f(z)>0$, since 
$\pa_{\ft n_ {\Omega}}f(z)<0$ would imply that $z$   is  a local minimum 
of $f$ in $\overline\Omega$ which would thus  not belong to~$\overline{\ft C_{1,k}}$.
 This proves item~(a) in Proposition~\ref{pr.assumption}.
Let us mention that one can prove in addition  that  $\pa \Omega$ and $\pa \ft C_{1,k}$ are tangent at $z$. 
\medskip

\noindent
\textbf{Step 2.} Proof of  item~(b) in Proposition~\ref{pr.assumption}.   Let us now assume that   $\vert \nabla f(z)\vert =0$. 

\medskip

\noindent
\textbf{Step 2a}. Let us prove that $0$  is a saddle point of $f:\ft V_0\to \mathbb R$. The point $0$ is a non degenerate critical point of $f$. Moreover, because $0$  is not a local minimum of $f$ in $\{x_d\le 0\}$ (since  $0\in \pa \ft C_{1,k}$),  $\Hess  f(0)$  has at least one negative eigenvalue. To prove that $0$ is a saddle point of $f$, let us argue by contradiction: assume that $\Hess  f(0)$ has at least two negative eigenvalues. 
 Then, according to~Lemma~\ref{le.local-structure} (with $p\ge 2$ there), for all $r\in (0,r_0)$ small enough, the set~$\{x\in \mathbb R^d,   f(x)<f(0)\}\cap \{x\in \mathbb R^{d}, \vert x\vert <r\} $ is connected. In particular, the same arguments as those used to prove~\eqref{eq.eq-ck-f<} and~\eqref{eq.eq-ck0} imply that:
    \begin{equation}\label{eq.eq-ck}
\ft C_{1,k} \cap \{x\in \mathbb R^d,  \vert x\vert <r\} =  \{x\in \mathbb R^d, \vert x\vert <r   \text{ and } f(x)<f(0) \} \subset \{x_d<0\}. 
   \end{equation}
To conclude, let us now prove that 
  \begin{equation}\label{eq.eq-xd}
  \{x\in \mathbb R^d, \vert x\vert <r\ \text{ and }   f(x)<f(0)\}\cap \{x\in \mathbb R^{d}, x_d=0\} \neq \emptyset,
   \end{equation} 
   which will contradict~\eqref{eq.eq-ck}. To this end,
 let  $(\ft e_1, \ft e_2,\ldots, \ft e_d)\subset \mathbb R^d$ be an orthonormal  basis of eigenvectors of $\Hess  f(0)$ associated with its eigenvalues $(\mu_1,\ldots,\mu_d)$ ordered such that $\mu_1<0$ and $\mu_2<0$. 
Since 
$\{x_d=0\} $
is a $d-1$ dimensional vector space, there exists  
$\ft v \in   \{x_d=0\}    \cap \text{Span} (\ft e_1, \ft e_2)\setminus \{0\}$. 
An order $2$ Taylor expansion then shows that
$f(t\,\ft v)<f(0)$ for every $t>0$ small enough, which implies \eqref{eq.eq-xd} since $t\,\ft v\in \{x_d=0\}$.  
Thus,   $\Hess  f(0)$  has only one negative eigenvalue, i.e. $0$ is a saddle point of $f$. 

\medskip

\noindent
\textbf{Step 2b}.  Let us now end the proof of item (b) in Proposition~\ref{pr.assumption}. 
The point $0$ is clearly a critical point of   $f|_{\{x_d=0\}}$ since   it is a critical point,
and more precisely a saddle point by the above analysis, of~$f: \ft V_0\to \mathbb R$. 
Let us also emphasize here that
without any additional assumption, $0$ is not necessarily  a non degenerate critical point of~$f|_{\{x_d=0\}}$,
nor a local minimum of~$f|_{\{x_d=0\}}$ (see indeed Remark~\ref{re-ex-2D} below). 
Let us now make the following additional assumption: let us   assume that the unit outward normal  vector 
$\ft n_{ \Omega}(z)$ is an eigenvector of~$\Hess  f(z)$ associated with its  negative eigenvalue.
According to \eqref{eq.coo-omega} and \eqref{eq.coo-omega'}, this means
that $\ft e_d=(0,\dots,0,1)\in \mathbb R^d$ is  
an eigenvector of $\Hess  f(0)$ associated with its unique negative eigenvalue.
Since in the Euclidean space $\mathbb R^d$, it holds
$\{x_d=0\}=  \ft e_d^{\perp}$, 
it follows that $\Hess  f|_{\{x_d=0\}}(0)$ is positive definite
and hence that 
$0$
is a   non degenerate local minimum of $f|_{\{x_d=0\}}$.  This concludes the proof of item~(b) in Proposition~\ref{pr.assumption}.
%
%
 
\medskip

\noindent
\textbf{Step 3.} Proof of the relation \eqref{eq.sspCk}.  Let us recall that for every $k$, the  set $\ft C_{1,k}$ is an open subset of $\Omega$ such that  for all $\ell\neq k$, it holds $\ft C_{1,\ell}\cap \ft C_{1,k}=\emptyset$ (see Proposition~\ref{pr.p1}), and hence
 $\overline{\ft C_{1,\ell}}\cap \overline{\ft C_{1,k}}=\pa{\ft C_{1,\ell}}\cap \pa{\ft C_{1,k}}$. The proof of~\eqref{eq.sspCk} is divided into two steps. 
\medskip

\noindent
\textbf{Step 3a.} Let us prove that for all $\ell\in \{1,\ldots,\ft N_1\}$, $\ell\neq k$,
it holds
\begin{equation}\label{eq.neql}
\pa{\ft C_{1,\ell}}\cap \pa{\ft C_{1,k}}\subset \Omega.
\end{equation}
To this end, let us consider  $z\in  \pa{\ft C_{1,k}} \cap \pa \Omega$.   Let us  work again in the $x$-coordinates
satisfying \eqref{eq.coo-omega0} and~\eqref{eq.coo-omega}, and with the function
$$f:\, x=(x',x_d)\in \ft V_0\subset \mathbb R^d\mapsto f(x)\in \mathbb R$$
which was  introduced in~\eqref{eq.xcoordi}.
\medskip

\noindent
Let us first consider the case when  $\vert \nabla f(0)\vert \neq 0$. Let us recall that according to Lemma~\ref{le.local-structure} and~\eqref{eq.eq-ck0}, for $r>0$ small enough, $ \{x\in \mathbb R^d,    \vert x\vert <r \text{ and }   f(x)<f(0)\}$ is connected and equals $\ft C_{1,k} \cap \{x\in \mathbb R^d,  \vert x\vert <r\}$.
Let  $\ell\in \{1,\ldots,\ft N_1\}$, $\ell\neq k$. Since in addition $\ft C_{1,\ell}\cap \ft C_{1,k}= \emptyset$, one has  $0\notin  \pa{\ft C_{1,\ell}}$.   This concludes the proof of~\eqref{eq.neql} when $\vert \nabla f(0)\vert \neq 0$.

\medskip

\noindent
Let us now consider   the case when $\vert \nabla f(0)\vert =0$. According to   item (b),   $0$ is a saddle point  of~$f$.  
  According to~Lemma~\ref{le.local-structure} and since $0$ is a non degenerate saddle point of $f$,  for $r>0$ small enough,~$\{x\in \mathbb R^d,  \vert x\vert <r \text{ and }  f(x)<f(0)\} $ has two connected components which are denoted by  $\ft A_1$ and~$\ft A_2$.
To prove~\eqref{eq.neql}, let us argue by contradiction and let us assume that $0\in \pa \ft C_{1,\ell} \cap \pa \ft C_{1,k}$ for some  $\ell\in \{1,\ldots,\ft N_1\}$ with $\ell\neq k$. 
Since both $\ft C_{1,k}$ and $\ft C_{1,\ell}$ meet $\ft A_1\cup \ft A_2$, 
the same arguments as those used to prove~\eqref{eq.eq-ck-f<} and~\eqref{eq.eq-ck0}
then lead, up to switching $\ft A_1$ and $\ft A_2$,  to
$$
\ft C_{1,k} \cap \{x\in \mathbb R^d,  \vert x\vert <r\} =  \ft A_1
\quad\text{and}
\quad
\ft C_{1,\ell} \cap \{x\in \mathbb R^d,  \vert x\vert <r\} =  \ft A_2
$$
and to
   \begin{equation}\label{eq.a1a2}
\{x\in \mathbb R^d,    \vert x\vert <r \text{ and }   f(x)<f(0)\}=\ft A_1\cup \ft A_2\subset \{x_d<0\}. 
\end{equation}  
This imposes that
the eigenvector $\ft e_d$ of $\Hess    f(0)$ associated with its  negative eigenvalue
satisfies 
 \begin{equation*}\label{eq.e1xd}
\ft e_d\in \{x_d=0\}. 
\end{equation*}
Indeed, if it was not the case, an order $2$ Taylor expansion of $t\mapsto f(t\,\ft e_d)$ at $t=0$ would imply
that $f-f(0)$ admits negative values in $\{x_d>0\}\cap \{|x|<r\}$ for every $r>0$, contradicting \eqref{eq.a1a2}. Thus, $\ft e_d\in \{x_d=0\}$. 
Then, the  order $2$ Taylor expansion of  $t\mapsto f(t\,\ft e_d)$ at $t=0$
shows that $f-f(0)$ admits negative values in $\{x_d=0\}\cap \{|x|<r\}$ for every $r>0$,
which also contradicts~\eqref{eq.a1a2}. This  
concludes the proof of~\eqref{eq.neql} when $\vert \nabla f(0)\vert =0$.
 \medskip

\noindent
\textbf{Step 3b.} Proof  of~\eqref{eq.sspCk}. According to~\eqref{eq.neql},  for all $\ell \neq k$, 
it holds  $\pa \ft C_{1,k}\cap \pa \ft C_{1,\ell}\subset \Omega$. Let us now consider $z\in \pa \ft C_{1,k}\cap \pa \ft C_{1,\ell}$
when the latter set in non empty, which implies that $\ft C_{1,k}$ and $\ft C_{1,\ell}$ are two connected components of $\{f <f(z)\}$. Then, for $r>0$ small enough, 
~$\{f <f(z)\}\cap B(z,r)$ has at least two connected components, respectively included in $\ft C_{1,k}$ and in $\ft C_{1,\ell}$. From~Lemma~\ref{le.local-structure},~$z$ is then a saddle point of $f$ and, according to Definition~\ref{de.SSP-omega},
it thus belongs to
$\ft U_1^{\ft{ssp}}(\Omega)$. 
This concludes the proof of~\eqref{eq.sspCk} and then the proof of Proposition~\ref{pr.assumption}.
\end{proof}

\noindent
We are now in position to state the following assumption which will be used to construct the maps $\mbf j$ and $\ft C_{\mbf j}$ at the end of this section. Before stating it, let us recall
that from item~(b) in Proposition~\ref{pr.assumption},
any point $z$ belonging to $\pa \ft C_{1,k} \cap \pa \Omega$
for some $k\in\{1,\dots,\ft N_1\}$ and such that
$|\nabla f(z)|=0$ is a saddle   point of $f$. Using moreover \eqref{eq.sspCk},
such a $z$ does not belong to   $\overline{\ft C_{1,\ell}}$
when $\ell\in\{1,\dots,\ft N_1\}\setminus\{k\}$.

\begin{manualassumption}{\bf(H1)}
\label{H1}
The function  $f:\overline \Omega\to \mathbb R$ is a $C^\infty$ Morse function
such that $\ft U_{0}\neq \emptyset$ and whose
 principal wells $\ft C_{1,1},\dots,   \ft  C_{1,  \ft N_{1}}$    defined in Definition~\ref{de.1}
satisfy the following property:   for every $k\in \{ 1,\dots, {\ft  N_{1}}\}$ and every $z\in \pa \ft C_{1,k} \cap \pa \Omega$
such that  $\vert\nabla f(z)\vert= 0$,  the unit outward normal vector   $\ft n_{  \Omega}(z)$ to  $ \Omega$ at $z$ is an eigenvector of $\Hess  f(z)$ associated with its  negative eigenvalue,
where $\Hess  f(z)$ denotes the endomorphism of $T_z\overline\Omega$
canonically associated  with the   symmetric bilinear form
 $\Hess  f(z):T_z\overline\Omega\times T_z\overline\Omega \to \mathbb R$
 via the metric~$g$.
\end{manualassumption}

\noindent
When \autoref{H1}
is satisfied,  according to  Proposition~\ref{pr.assumption}, 
the sublevel sets $\{f<f(z)\}$
have the following local structure  near the points  $z\in\bigcup_{k=1}^{\ft N_1} \pa \ft C_{1,k} \cap \pa \Omega$.

 \begin{corollary}\label{co.hypoA}
Let  $f:\overline \Omega\to \mathbb R$ be  a $C^\infty$ Morse function satisfying    \autoref{H1}. Then, for all $k\in \{ 1,\dots, {\ft  N_{1}}\}$ such that $\pa \ft C_{1,k} \cap \pa \Omega\neq \emptyset$ and for all 
$z\in \pa \ft C_{1,k} \cap \pa \Omega$, one has:
\begin{enumerate} 
\item[(a)] If $\vert \nabla f(z)\vert \neq 0$,  $z$ is a   local minimum of $f|_{\pa \Omega}$ and $\pa_{\ft n_{  \Omega}}f(z)>0$ (see Figure~\ref{fig:z_nonpc}).

\item[(b)] If $\vert \nabla f(z)\vert = 0$, $z$ is a  saddle point of $f$ and 
the unit outward normal  vector $\ft n_{  \Omega}(z)$ to  $ \Omega$ at $z$ is an eigenvector of $\Hess  f(z)$ associated with its  negative eigenvalue. Moreover, the point  $z$ is a non degenerate local minimum of $f|_{\pa \Omega}$ (see Figure~\ref{fig:z_PC}).
\end{enumerate}
 \end{corollary}

\noindent
Note that when  \autoref{H1} is satisfied, it follows from
 Corollary~\ref{co.hypoA}
 that
the  points $z\in \bigcup_{k=1}^{\ft N_1}\pa \ft C_{1,k}\cap \pa \Omega$ such that $\vert \nabla f(z)\vert = 0$ are isolated in 
$\bigcup_{k=1}^{\ft N_1}\pa \ft C_{1,k}\cap \pa \Omega$.
Indeed, they are 
 non degenerate critical points of $f|_{\pa \Omega}$
 and $ \bigcup_{k=1}^{\ft N_1}\pa \ft C_{1,k}\cap \pa \Omega$ is composed of critical points of $f|_{\pa \Omega}$.
 Note also that this is in general not the case for the points  
 $z\in \bigcup_{k=1}^{\ft N_1}\pa \ft C_{1,k}\cap \pa \Omega$ such that $\vert \nabla f(z)\vert \neq 0$.

%
%


\begin{figure}[h!]
\begin{center}
\begin{tikzpicture}[scale=0.5]
\tikzstyle{vertex}=[draw,circle,fill=black,minimum size=6pt,inner sep=0pt]
 
 \draw[thick, ->] (0,0)--(1.5,0);
 \draw (5.5,0) node[]{$\nabla f(z)=\pa_{\ft n_{  \Omega}}f(z)\, \ft n _{\Omega}(z)$};
 
  \draw (0,0)  node[vertex,label=north east: {$z$}](v){};
    \draw (-11,0) node[]{$ \Omega$};
  \draw (-5,0) node[]{$ \ft C_{1,k}$};
       \draw[thick]  (0,4)--(0,-4);
            \draw (0,4.5) node[]{$\pa \Omega$};
    \draw (-4.5,3.8) node[]{$ \big\{f>f(z)\big\}$};
      \draw (-4.5,-3.8) node[]{$\big \{f>f(z)\big\}$};
  \draw[thick, dashed]   (-7,3) ..controls  (2.2,2)  and (2.2,-2)  ..  (-7,-3) ;
 
   \draw[thick,dashed] (3,3)--(2,3);
   \draw (5,3) node[]{$\big \{f=f(z)\big\}$};
\end{tikzpicture}
\caption{Behaviour  of   $f$ in a neighborhood of $z\in \pa \ft C_{1,k}\cap \pa \Omega$ when $\vert \nabla f(z)\vert \neq 0$ and  $z$ is isolated in $\pa \ft C_{1,k}\cap \pa \Omega$.  }
 \label{fig:z_nonpc}
 \end{center}
\end{figure}
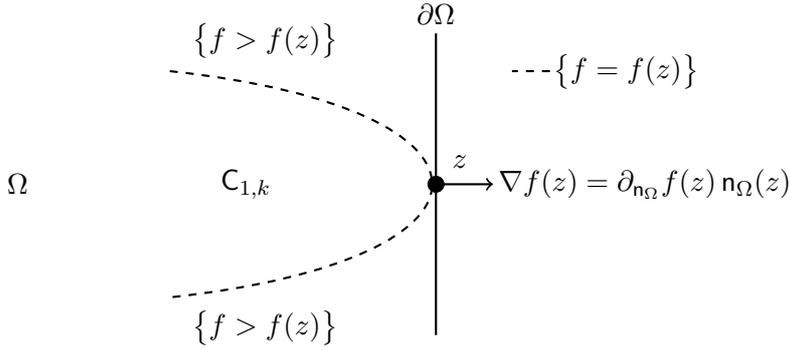
%
%
%
%


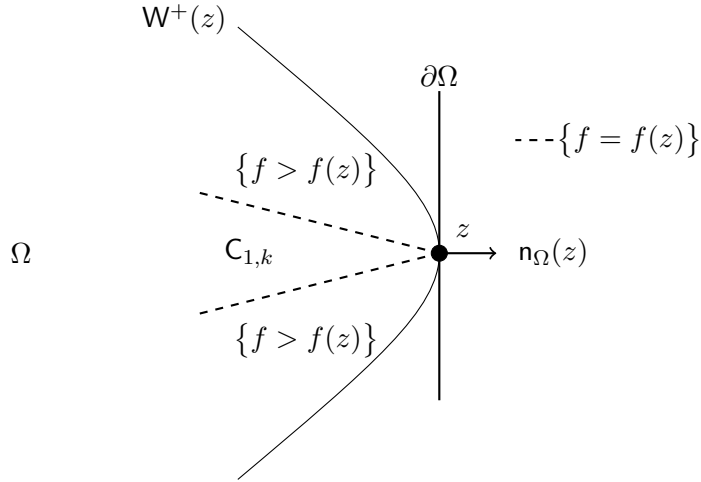
\begin{figure}[h!]
\begin{center}
\begin{tikzpicture}[scale=0.5]
\tikzstyle{vertex}=[draw,circle,fill=black,minimum size=6pt,inner sep=0pt]
\draw[thick] (0,-3.9)--(0,4.3);
\draw[thick, ->] (0,0)--(1.5,0);
 \draw (3,0) node[]{$\ft n_{  \Omega}(z)$};
 \draw (0,4.7) node[]{$\pa \Omega$};
           \draw (-11,0) node[]{$ \Omega$};
  \draw (0,0)  node[vertex,label=north east: {$z$}](v){};
  \draw (-5,0) node[]{$ \ft C_{1,k}$};
      \draw   (-5.3,6) ..controls  (1.77,0)  ..  (-5.3,-6) ;
            \draw (-6.7,6.2) node[]{$\ft W^+(z)$};
    \draw (-3.5,2.2) node[]{$ \big\{f>f(z)\big\}$};
      \draw (-3.5,-2.3) node[]{$\big \{f>f(z)\big\}$};
  \draw[thick, dashed]  (-6.3,1.6) to (0,0) ;
   \draw[thick, dashed]  (-6.3,-1.6) to (0,0) ;
   \draw[thick,dashed] (3,3)--(2,3);
   \draw (5,3) node[]{$\big \{f=f(z)\big\}$};
\end{tikzpicture}
\caption{Behaviour  of $f$ in a neighborhood of $z\in \pa \ft C_{1,k}\cap \pa \Omega$ when $\vert \nabla f(z)\vert =0$ and \autoref{H1} is satisfied. On this figure, $\ft W^+(z)$ is the stable manifold of $z$ for the dynamics $\dot X=-\nabla f(X)$. }
 \label{fig:z_PC}
 \end{center}
\end{figure}

\begin{remark}\label{gc}
When~\autoref{H1} holds, it follows
 from items (a) and (b) in Corollary~\ref{co.hypoA}  that the  elements of $\bigcup_{k=1}^{\ft  N_{1}}\big(\pa  \ft  C_{1,k}\cap \pa \Omega\big)$    play geometrically the role of saddle points of $f$ in~$\overline \Omega$. Indeed,  when~$f$ is extended by~$-\infty$ outside~$\overline \Omega$ (this extension is consistent with the Dirichlet boundary conditions used to define $\Delta^{D}_{f,h}$), the points $z\in \bigcup_{k=1}^{\ft  N_{1}} \pa  \ft  C_{1,k}\cap \pa \Omega$  are  local minima of $f|_{\pa \Omega}$ and  local maxima of~$f|_{D_z}$, where $D_z$ is  the straight line passing through~$z$ and orthogonal to  $\pa \Omega$   at~$z$. Note however that when $|\nabla f(z)|\neq 0$, 
 $z$ can be a degenerate local minimum of $f|_{\pa\Omega}$ (which can even be constant around $z$).
 This extends the definition of {\em generalized saddle points} of $f$ in $\pa \Omega$ as introduced in~\cite[Definition 3.2.2]{HeNi1} to the case when $f|_{\pa \Omega}$ is not a Morse function and $f$ has critical points on~$\pa \Omega$.   
Moreover, when \autoref{H1} does not hold,   the points $z\in \bigcup_{k=1}^{\ft  N_{1}} \pa  \ft  C_{1,k}\cap \pa \Omega$ such that $\vert \nabla f(z)\vert =0$, which are thus saddle points of $f$
according to Proposition~\ref{pr.assumption}, do actually not    necessarily 
 play the role of   saddle points of~$f$ in $\overline \Omega$ in the above sense, 
 as explained in Remark~\ref{re-ex-2D} below.

 \end{remark}

\begin{remark}\label{re-ex-2D} Let  
$k\in \{1,\ldots,\ft N_1\}$ and
 $z\in \pa  \ft  C_{1,k}\cap \pa \Omega$ be  such that  $\vert\nabla f(z)\vert=0$. 
We recall that, according to  Proposition~\ref{pr.assumption},
$z$ is a saddle point of $f$, and that, by Corollary~\ref{co.hypoA},  when  $\ft n_{  \Omega}(z)$ is an eigenvector of $\Hess  f(z)$ associated with its  negative eigenvalue,
$z$ is a    local minimum of~$f|_{\pa \Omega}$ and thus a geometrical  saddle point of~$f$ in~$\overline \Omega$
in the sense of   Remark~\ref{gc}.
 We show below that the latter property fails to be true  in general when  $z\in \pa  \ft  C_{1,k}\cap \pa \Omega$
is only assumed to be a critical point, and is hence a saddle point, of $f$. To this end,  let us consider, in the canonical basis $(\ft e_x,\ft e_y)$ of $\mathbb R^2$, the Morse function
$$\psi(x,y)= y^2-x^2\,,$$
whose only critical point in $\mathbb R^2$ is $0$ and is a saddle point. Let us then introduce the two vectors 
 $$\ft u=\frac{1}{\sqrt 2}(\ft e_x-\ft e_y )  \text{ and } \ft v=\frac{1}{\sqrt 2}( \ft e_x+\ft e_y).$$
 In the orthonormal  basis $(\ft u,\ft v)$, the function $\psi$ writes $\psi(u,v)=-2uv$.
 Hence, defining the smooth curve
 $$\Gamma:=\{p=(u,u^2) \text{ in the basis } (\ft u,\ft v), u\in \mathbb R\} \text{ (see Figure~\ref{fig:example})},$$
it holds $\psi|_\Gamma: p=(u,u^2) \in \Gamma\mapsto -2u^3$ and    $0$ is then  not a local minimum of $f|_\Gamma$.
In particular, 
if, in a neighborhood of $0$ in $\mathbb R^2$,
$\pa\Omega$ coincides with $\Gamma$ 
and $\Omega$ is chosen such that 
 $\ft n_\Omega(0)=\ft v$,
and if $f=\psi$,
then, locally around $0$ in $\overline \Omega$,
$\{f<0\}\cap\{x<0\}$
is a connected component of  $\{f<0\}$
included in $\Omega$ such that
$\overline{\{f<0\}\cap\{x<0\}}\cap\pa\Omega=\{0\}$ but $0$ is not a local minimum of $f|_{\pa \Omega}$
(see Figure~\ref{fig:example}).

\end{remark}

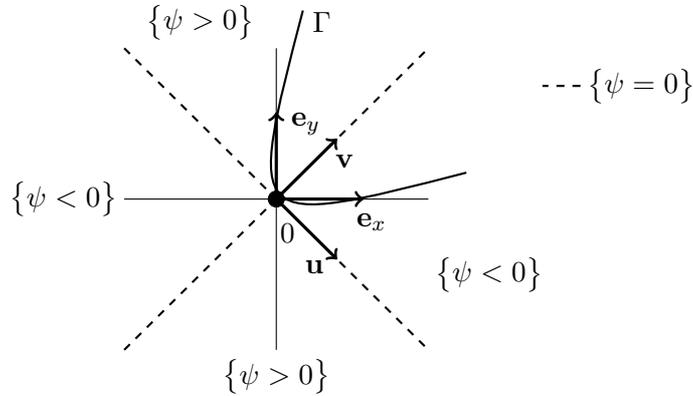
\begin{figure}[h!]
\begin{center}
\begin{tikzpicture}[scale=0.5]
\tikzstyle{vertex}=[draw,circle,fill=black,minimum size=6pt,inner sep=0pt]
\draw[-] (0,-4)--(0,4);
\draw[-] (-4,0)--(4,0);
\draw[very thick, ->] (0,0)--(2.3,0);
 \draw (2.5,-0.6) node[]{$\mbf e_x$};
  \draw (0.3,-0.9) node[]{$0$};
 \draw[very thick, ->] (0,0)--(0,2.3);
  \draw (0.77,2) node[]{$\mbf e_y$};
  \draw (0,0)  node[vertex,label=south east: {}](v){};
  \draw (-5.6,0) node[]{$\big \{\psi<0\big\}$};
    \draw (5.6,-2) node[]{$\big \{\psi<0\big\}$};
      \draw (0,-4.7) node[]{$\big \{\psi>0\big\}$};
    \draw (-2,4.7) node[]{$\big \{\psi>0\big\}$};
        \draw (1.2,4.7) node[]{$\Gamma$};
      \draw[thick]  (0.7,5) ..controls  (-0.8,-0.8)  ..  (5,0.7) ;
  \draw[thick, dashed]  (-4,4) to (4,-4) ;
    \draw[thick, dashed]  (-4,-4) to (4,4) ;  
   \draw[very thick, ->] (0,0)--(1.6,-1.6);
      \draw (1.8,1) node[]{$\mbf v$};
           \draw (1,-1.8) node[]{$\mbf u$};
   \draw[very thick, ->] (0,0)--(1.6,1.6);
   
  \draw[thick,dashed] (8,3)--(7,3);
   \draw (9.6,3) node[]{$\big \{\psi=0\big\}$};
\end{tikzpicture}
\caption{The function $\psi$ and the curve $\Gamma$ in a neighborhood of $0$ in $\mathbb R^2$. }
 \label{fig:example}
 \end{center}
\end{figure}

\noindent
 When \autoref{H1} holds, one   adapts the definition of  a separating saddle point of $f$ in~$\Omega$ given in Definition~\ref{de.SSP-omega}  to our setting by: i) only considering the relevant elements of 
$\ft U_1^{\ft{ssp}}(\Omega)$ for our study, and ii)
  taking into account the points in $\bigcup_{i=1}^{\ft  N_{1}} \pa  \ft  C_{i}\cap \pa \Omega $ which are, according to Remark~\ref{gc}, geometrical saddle points of $f$ in $\overline \Omega$.  
  Note in particular that with this definition of $\ft U_1^{\ft{ssp}}(\overline \Omega)$
  given below, it does not hold $\ft U_1^{\ft{ssp}}( \Omega)\subset \ft U_1^{\ft{ssp}}(\overline \Omega)$
  in general.

\begin{definition}
 \label{de.SSP} 
Let $f:\overline \Omega\to \mathbb R$ be a  $C^\infty$ Morse function satisfying    \autoref{H1} and 
let
$ \ft  C_{1,1},\dots, \ft  C_{1, \ft N_{1}}$ be its principal wells  defined in Definition~\ref{de.1}.
  \begin{enumerate} 
 \item[1.] A point $z\in \overline \Omega$
 is a {separating saddle point} of $f$ in $\overline \Omega$ if
$$\text{ either } \,    z\in  \bigcup_{k=1}^{ \ft N_{1}} \left(\, \overline{ \ft  C_{1,k}} \cap \ft U_1^{\ft{ssp}}(\Omega)\, \right),  \text{ or } z \in \bigcup_{k=1}^{\ft  N_{1}}\big(\pa  \ft  C_{1,k}\cap \pa \Omega\big).$$
 Notice that in the first case  $z\in \Omega$ whereas in the second case $z\in \pa \Omega$. The set of
 {separating saddle points}  of $f$ in $\overline \Omega$ is denoted by~$\ft U_1^{\ft{ssp}}(\overline \Omega)$.
 \item[2.]    For any $\sigma \in \mathbb R$, a connected component $\ft C$ of the sublevel set  
 $\{ f< \sigma  \}$ in~$\overline \Omega$ is called {a critical connected component} of $f$   if $\pa \ft C\cap \ft U_1^{\ft{ssp}}(\overline \Omega)\neq \emptyset$.
 The family of {critical connected components} is denoted by~$\mathcal C_{crit}$\label{page.ccrit}.
 \end{enumerate}
 \end{definition}
 \noindent
Equation~\eqref{eq.C-ssp} and item 1 in Definition~\ref{de.SSP} imply that the principal wells  $(\ft C_{1,\ell })_{\ell \in \{1,\ldots \ldots,\ft N_1\}}$ are critical connected components, as stated in the next corollary. This will be used in the first step of the construction of the maps $\mbf j$ and $\ft C_{\mbf j}$.
 \begin{corollary}\label{co.nonempty}
  Let $f:\overline \Omega\to \mathbb R$ be a $C^\infty$ Morse function  satisfying \autoref{H1}.   Then, it holds:  
$$
\text{ for all  $\ell \in \{1,\ldots,\ft N_1\}$, } \   \pa \ft C_{1,\ell }\cap \ft U_1^{\ft{ssp}}(\overline \Omega) \neq \emptyset.
$$
 \end{corollary}

\subsection{Construction of the maps $\mbf j$ and $\ft C_{\mbf j}$}
\label{sec.j}
Let us now construct  the maps $\mbf j$
and $\ft C_{\mbf j}$, which respectively associate 
each local minimum of
$f$ in $\Omega$ with a set of 
$\ft U_1^{\ft{ssp}}(\overline \Omega)$
 and with an element of $\mathcal C_{crit}$ (see Definition~\ref{de.SSP}).
 We closely follow the presentation
 of \cite[Section 2.4]{DLLN-saddle0} in the case when 
 $f$ does not have  any critical point on the boundary and $f|_{\pa \Omega}$ is a Morse function
 and which was inspired by \cite{herau-hitrick-sjostrand-11}
 in the case without boundary.
%
\medskip


\noindent
Let us assume that $f:\overline \Omega\to \mathbb R$ is a  $C^\infty$ Morse function  satisfying \autoref{H1}
(and thus such that $\ft U_{0}\neq \emptyset$.)
The  maps $\mbf j$ and $\ft C_{\mbf j}$ are then defined recursively as follows.\medskip

\noindent
\textbf{1. Initialization ($q=1$).} Let us consider  the  principal wells
$\ft  C_{1,1},\dots, \ft  C_{ 1,\ft N_{1}}$
 of~$f$ in~$\Omega$   (see Definition~\ref{de.1}).
\medskip

\noindent
For every $\ell\in\{1,\dots,\ft  N_1 \}$, let us choose   
$$x_{1,\ell}\in \argmin_{\ft  C_{1,\ell}}f.$$
Then, for all~$\ell \in\{1,\dots,\ft  N_1 \}$, one defines
\begin{equation}\label{jx1}
\kappa_{1,\ell}:= \max_{\overline{\ft  C_{1,\ell}  } } f,\ \  \ft C_{\mbf j}  (x_{1,\ell}):=\ft  C_{1,\ell},\   \text{and}\ \ 
\mathbf j (x_{1,\ell}):= \pa \ft  C_{1,\ell}\cap\ft   U_1^{\ft{ssp}}(\overline \Omega).
\end{equation}
From Definitions~\ref{de.principal-wells} and~\ref{de.1},  $\pa\ft C_{\mbf j}  (x_{1,\ell})\subset\{f=\kappa_{1,\ell}\}$
for all $\ell\in\{1,\dots,\ft  N_1 \}$. 
According moreover to Corollary~\ref{co.nonempty}, one has  $\mathbf j(x_{1,\ell})\neq \emptyset$  for all  $\ell\in\{1,\dots,\ft  N_1 \}$ and thus,  
$\ft C_{\mbf j}  (x_{1,\ell}) \in \mathcal C_{crit}$ (see  item 2 in Definition~\ref{de.SSP}). 
Finally,  it holds from~\eqref{eq.sspCk},
$$
\forall \ell\neq q\in \{1,\dots,\ft N_1\}^2,\ \pa \ft  C_{1,\ell} \cap \pa \ft C_{1,q} \subset \ft U_1^{\ft{ssp}}(\Omega).
$$

\noindent 
 \textbf{2. First step ($q=2$)}.
 \begin{sloppypar}
\noindent
From item 2 in Proposition~\ref{pr.p2}, for each $\ell\in\{1,\dots,\ft N_1\}$, 
$ \ft  C_{1,\ell} \cap \ft U_0\neq \{x_{1,\ell}\}$ if and only if
$\ft U_1^{\ft{ssp}}(\Omega)\cap \ft  C_{1,\ell}\neq~\emptyset$. Consequently, one has: 
 $$\ft U_1^{\ft{ssp}}(\Omega)\bigcap \Big(\cup_{\ell=1}^{\ft N_1}\ft  C_{1,\ell}\Big)\neq \emptyset \ \text{ iff }\ 
 \{x_{1,1},\dots,x_{1,\ft N_1}\}\neq \ft U_0.$$
If $\ft U_1^{\ft{ssp}}(\Omega)\bigcap \Big(\cup_{\ell=1}^{\ft N_1}\ft  C_{1,\ell}\Big)= \emptyset $ (or equivalently if  $\ft N_1=\ft m_0$), the constructions of the maps~${\mbf j}$ and $\ft C_{\mbf j}$ are finished and one goes   to item 4 below. If $\ft U_1^{\ft{ssp}}(\Omega)\bigcap \Big(\cup_{\ell=1}^{\ft N_1}\ft  C_{1,\ell}\Big)\neq \emptyset $, one defines
  $$
 \kappa_{2}:= \max_{x\in \ft U_1^{\ft{ssp}}(\Omega)  \bigcap \big(\cup_{\ell=1}^{\ft N_1}\ft  C_{1,\ell}\big) } f(x) \ \in\ \Big(\min_{\cup_{\ell=1}^{\ft N_1}\ft  C_{1,\ell} }f,  \max_{\ell\in\{1,\dots,\ft N_1\}} \kappa_{1,\ell}  \Big).
 $$ 
The  set $$  \bigcup_{\ell=1}^{\ft N_1}\Big (\ft  C_{1,\ell}\cap \{ f<  \kappa_{2} \}\Big)$$ is then the union of 
 finitely many connected components. We denote by  $\ft C_{2,1},\dots,\ft  C_{2,\ft N_2}$\label{page.ekl2} (with $\ft N_2\geq 1$\label{page.n2}) the  connected components of $\bigcup_{\ell=1}^{\ft N_1}\big (\ft  C_{1,\ell}\cap \{ f<  \kappa_{2} \}\big)$ which do not contain any of the minima~$ \{x_{1,1},\dots,x_{1,\ft N_1}\}$. 
 From item  2 in Proposition~\ref{pr.p2} (applied for each $ \ell\in \{1,\ldots,\ft N_1\}$  with  $\ft C= \ft  C_{1,\ell}\cap \{ f<  \kappa_{2} \}$ there) and item 2 in Definition~\ref{de.SSP},
$$\forall \ell \in \{1,\ldots,\ft N_2\}, \ \ft C_{2,\ell} \in \mathcal C_{crit}.~$$
Let us mention that the other connected components (i.e. those 
 containing the~points~$ \{x_{1,1},\dots,x_{1,\ft N_1}\}$) may be not   critical. 
For each
$1\leq \ell\leq\ft  N_{2}$,
one  then considers an element $x_{2,\ell}$ arbitrarily chosen in $\argmin_{\overline{\ft C_{2,\ell}}} f=\argmin_{\ft C_{2,\ell}} f$ (the equality follows from $\pa\ft  C_{2,\ell}\subset \{f=\kappa_{2}\}$)
and one  defines:
 $$\ft C_{\mbf j}(x_{2,\ell}):=\ft C_{2,\ell}
 \ \, \text{and}\ \, \mathbf j(x_{2,\ell}):=\pa \ft C_{2,\ell}\cap \ft U_1^{\ft{ssp}}(\overline \Omega) \, (\neq \emptyset) \subset  \ft U_1^{\ft{ssp}}(\Omega) \cap \{f=\kappa_{2}\}. $$

\noindent
 \textbf{3. Recurrence ($q\ge 3$)}. 
 \medskip
 
 \noindent
If all the local minima of $f$ in $\Omega$ have been labeled at the end of the previous step, i.e. if $\cup_{j=1}^{2}\{x_{j,1},\ldots,x_{j,\ft N_j}\}= \ft U_0$ (or equivalently  if $\ft N_1+\ft N_2= \ft m_0 $), the constructions of the maps $\ft C_{\mbf j}$ and ${ \mathbf j}$ are finished, all the local minima of $f$ have been labeled and one goes  to item 4 below. 
If it is not the case, from item 2 in Proposition~\ref{pr.p2}, there exists $m \in \mathbb N^*$  such that
\begin{equation}\label{eq.mm}
\text{for all $q\in \{2,\ldots,m+1\}$}, \ \ft U_1^{\ft{ssp}}(\Omega)\bigcap \bigcup_{\ell=1}^{\ft N_1}\Big (\ft  C_{1,\ell}\cap \{ f<  \kappa_{q}\}\Big) \neq\ \emptyset,
\end{equation}
 where    the decreasing sequence $( \kappa_{q})_{q=3,\ldots,m+2}$  is defined recursively by
$$
 \kappa_{q}:= \max_{x\in \ft U_1^{\ft{ssp}}(\Omega)\bigcap \bigcup_{\ell=1}^{\ft N_1}\big (\ft  C_{1,\ell}\cap \{ f<  \kappa_{q-1}\}\big)     } f(x)\ \in\ \Big(\min_{\cup_{\ell=1}^{\ft N_1}\ft  C_{1,\ell} }f, \, \kappa_{q-1}  \Big).
  $$
   Let  now $m^*\in \mathbb N^*$ be the largest $m\in \mathbb N^*$ such that~\eqref{eq.mm} holds. 
   Notice that $m^*$ is well defined since the cardinal 
 of $\ft U_1^{\ft{ssp}}(\Omega)$  
   is finite. By definition of $m^*$, one has moreover:
\begin{equation}\label{eq.m*}
 \ft U_1^{\ft{ssp}}(\Omega)\bigcap \bigcup_{\ell=1}^{\ft N_1}\Big(\ft  C_{1,\ell}\cap \{ f<  \kappa_{m^*+2}\}\Big) =\ \emptyset.
 \end{equation}
Then, one  repeats   recursively $m^*$ times  the procedure described above defining $\big (\ft C_{2,\ell},\mbf j(x_{2,\ell}), \ft C_{\mbf j}(x_{2,\ell})\big )_{ 1\le \ell\le \ft N_2  }$ : 
for  $q\in \{2,\ldots,m^*+1\}$, one defines $(\ft C_{q+1,\ell})_{\ell\in \{1,\ldots,\ft N_{q+1}\}}$ as the set of 
the
 connected components of 
$$  \bigcup_{\ell=1}^{\ft N_1}\Big (\ft  C_{1,\ell}\cap \{ f<  \kappa_{q+1} \}\Big)$$ 
which do not contain any of the local minima $\cup_{j=1}^{q}\{x_{j,1},\ldots,x_{j,\ft N_j}\}$ of~$f$ in~$\Omega$ which have been previously labeled.  
From items 1 and 2 in Proposition~\ref{pr.p2} (applied for each $ \ell\in \{1,\ldots,\ft N_{1}\}$  with  $\ft C= \ft  C_{1,\ell}\cap \{ f<  \kappa_{q+1} \}$ there),
$$\forall \ell \in \{1,\ldots,\ft N_{q+1}\}, \ \ft C_{q+1,\ell} \in \mathcal C_{crit}.$$
For~$\ell\in \{1,\ldots,\ft  N_{q+1}\}$, we then  associate with each~$\ft C_{q+1, \ell }$
 one point~$x_{q+1, \ell }$ arbitrarily chosen in~$\argmin_{\ft C_{q+1, \ell }} f$ and we define:
 $$\ft C_{\mbf j}(x_{q+1, \ell }):=\ft C_{q+1, \ell }
 \ \ \text{and}\ \ \mathbf j(x_{q+1, \ell }):=\pa \ft C_{q+1, \ell }\cap \ft U_1^{\ft{ssp}}(\Omega)  \, (\neq \emptyset) \subset \{f=\kappa_{q+1}\}. $$

\noindent
From~\eqref{eq.m*} and item 2 in Proposition~\ref{pr.p2}, $\ft U_0=\cup_{j=1}^{m^*+2}\{x_{j,1},\ldots,x_{j,\ft N_j}\}$. Thus, all the local minima of $f$ in $\Omega$ are labeled. This finishes the constructions of  the maps ${ \mathbf j}$ and $\ft C_{\mbf j}$. We refer to Figures 8 and 9 in~\cite{DLLN-saddle0} to illustrate these constructions. 
\end{sloppypar}
\medskip

\noindent \textbf{4. Properties of the  maps   ${ \mathbf j}$ and $\ft C_{\mbf j}$.}
\medskip

\noindent
 Let us now give   important features of the map $\mbf j$ which follow directly from its construction and which will be
 used in the sequel. 
Two maps have been defined 
 \begin{equation}
 \label{de.j}
 \ft C_{\mbf j}\ :\ \ft U_0\longrightarrow \mathcal C_{crit} \quad\text{and}
 \quad \mathbf j\ :\ \ft U_0\longrightarrow\mathcal P( \ft U_1^{\ft{ssp}}(\overline \Omega)) 
 \end{equation}
which are clearly injective. 
For every $x\in \ft U_{0}$,
the set $\mathbf j(x)$ is the set made of the separating saddle points of $f$ in $\overline \Omega$ on $\pa  \ft C_{\mbf j}(x)$. 
  Notice that the $\mathbf j(x)$,~$x\in \ft U_0$, are not   disjoint
in general. 
  For all $x\in\ft U_0$, the set $f(\mathbf j(x))$
 contains exactly one value, which will be denoted by~$f(\mathbf j(x))$.
 Moreover, for all $x\in\ft U_0$,  it holds
  \begin{equation}
 \label{eq.j-x1}
 f(\mathbf j(x))-f(x)>0.
 \end{equation}
  Since $\cup_{\ell=1}^{\ft N_1}\ft  C_{1,\ell}\subset \Omega$ (see the first statement in~\eqref{eq.ck-omega}), one has
  $\ft C_{\mbf j}(x)\subset \Omega$
   for all $x\in \ft U_0$. 
Moreover, only the boundaries of the principal wells can contain separating saddle points of $f$ on $\pa \Omega$, i.e.: 
  \begin{equation}
 \label{eq.j-x}
 \forall x  \in \ft U_0\setminus \{x_{1,1},\dots,x_{1,\ft N_1}\}, \ \ \mathbf j(x)\subset  \ft U^{\ft{ssp}}_1( \Omega)  \text{ (see Definition~\ref{de.SSP-omega})}.
  \end{equation}
In addition, for all $x,y\in \ft U_0$ such that $x\neq y$,  since by construction ${\mbf j}(y) \, \cap \,  \mbf j (x)=\pa  {\ft C_{\mbf j}(y)} \, \cap \, \pa   { \ft C_{\mbf j}(x)}$ (see~\eqref{eq.sspCk}), one has   two possible cases:
\begin{enumerate}
\item[(i)]  either ${\mbf j}(x) \, \cap \,  \mbf j (y)=\emptyset$, in which case either $\overline{\ft C_{\mbf j}(y)} \, \cap \, \overline{ \ft C_{\mbf j}(x)}=\emptyset$
or, up to interchanging  $x$ with $y$, $\overline{\ft C_{\mbf j}(y)} \subset  { \ft C_{\mbf j}(x)}$,
\item[(ii)] or  ${\mbf j}(x) \, \cap \,  \mbf j (y)\neq \emptyset$,  in which case $f(\mbf j(x))=f(\mbf j(y))$ and the sets ${\ft C_{\mbf j}(x)}$  and $ { \ft C_{\mbf j}(y)}$ are two different connected components of $\{f<f(\mbf j(x))\}$.

\end{enumerate}
 Finally,  for all $\ell\in \{1,\ldots,\ft N_1\}$ and all $  x\in\ft U_0 \cap \ft C_{\mbf j}(x_{1,\ell})\setminus \{x_{1,\ell}\}$,
 note that
$$
f(x)\geq f(x_{1,\ell})\,,\ f({\bf j}(x))< f(\mathbf j(x_{1,\ell}))\ \ \text{and then}\ \    f(\mathbf j(x))-f(x)  <   f(\mathbf j(x_{1,\ell}))-f(x_{1,\ell}).
$$

\noindent
Let us also mention that the maps $\mbf j$ and~$\ft C_{\mbf j}$ are not uniquely defined as soon as  there exists some $\ft C_{k,\ell}$, $k\ge 1, \, \ell\in \{1,\ldots,\ft N_k\}$, such that $f$  has more than  one global minimum  in  $\ft C_{k,\ell}$. 
However,  this non-uniqueness has no  influence on the results proven below
(in particular Theorems~\ref{th.thm2} and~\ref{th.thm3}).

\begin{remark}\label{re.denom}
The relevant wells
of the potential $f$ for our study are the sets $\ft C_{\mbf j}(x)$, $x\in \ft U_0$, and
the  elements of $\mathcal C$ 
(see  Definition~\ref{de.1}) are called the principal wells of $f$ in  $\Omega$ since, 
for any $x\in \ft U_0$, $\ft C_{\mbf j}(x)$ is either an element of 
$\mathcal C$ or a subset of an element~of~$\mathcal C$. 
  \end{remark}

\noindent
Let us end this section with the  following result which will be used in the proof of Proposition~\ref{pr.indep-psix} below. 

\begin{lemma}\label{le.debu-recurrence}
Let us assume that $f:\overline \Omega\to \mathbb R$ is  a $C^\infty$ Morse function which  satisfies~\autoref{H1}.  
Let $(\ft C_{\mbf j} (x))_{x\in \ft U_0}$ be as defined in~\eqref{de.j} and let $k\ge 1$.  Let us consider, for some $m\ge 1$,  $\big \{\ft C^1,\ldots,\ft  C^m\big \}\subset \big \{\ft C_{\mbf j} (x_{k,1}),\ldots,\ft C_{\mbf j} (x_{k,\ft N_k})\big \}$ such that  
$$
 \left\{
    \begin{array}{ll}
       \bigcup_{\ell=1}^{m}\overline{\ft C^\ell } \text{ is connected, and}  \\
        \text{ for all $\ft C\in \big \{\ft C_{\mbf j} (x_{k,1}),\ldots,\ft C_{\mbf j} (x_{k,\ft N_k})\big \}\setminus \big \{\ft C^1,\ldots,\ft  C^m\big \}$, } \, \ \overline{\ft C }\cap \bigcup_{\ell=1}^{m}\overline{\ft C^{\ell}}=\emptyset.
    \end{array}
\right.
$$
  Then, there exist  $\ell_0 \in \{1,\ldots,m\}$ and  $z\in \ft U_1^{\ft{ssp}}(\overline \Omega)$   such that 
\begin{equation}\label{eq.lemma-recu}
z\in \pa \ft C^{\ell_0}  \setminus \Big (    \cup_{\ell=1, \ell\neq \ell_0}^{m}\pa {\ft C^\ell }\Big).
\end{equation}
\end{lemma}
 \begin{proof}
 Let $\big \{\ft C^1,\ldots,\ft  C^m\big \}$ be as in Lemma~\ref{le.debu-recurrence}.\medskip
 
 \noindent
When $k=1$,
   the set $\big \{\ft C_{\mbf j} (x_{1,1}),\ldots,\ft C_{\mbf j} (x_{1,\ft N_1})\big \}$ is the set of the principal wells of $f$, i.e. the set $\mathcal C$ of Definition~\ref{de.1}, and   the proof of Lemma~\ref{eq.lemma-recu}  follows exactly 
   the same lines as
    the proof of \cite[Lemma~21]{DLLN-saddle0}. 
  \medskip
 
 \noindent
Let us now consider the case  when $k\ge 2$.   Let us first notice that according to the construction of the maps $\mbf j$ and $\ft C_{\mbf j}$, for all $\ell\in \{1,\ldots ,m\}$, $\ft C^\ell$ is a connected component of $\{f<\kappa_k\}$ which has been  labelled at the $k$-th iteration.  
Since  
  $\bigcup_{\ell=1}^{m}\overline{\ft C^\ell }$ is connected, there exists
  $q\in \{1,\ldots,\ft N_1\}$ such that  $\bigcup_{\ell=1}^{m}\overline{\ft C^\ell }\subset \ft C_{1,q}=\{f<\kappa_{1,q}\}$, where,
 since $k\geq 2$,   $\kappa_k< \kappa_{1,q}$. 
Since, from Corollary~\ref{co.nonempty}, it holds $\emptyset\neq \pa{\ft C_{1,q}}\cap \ft U_1^{\ft{ssp}}(\overline \Omega)\subset\{f=\kappa_{1,q}\}$, 
one can define
$\kappa^*\in(\kappa_k,\kappa_{1,q}]$
as the minimum of the $\lambda\in (\kappa_k,\kappa_{1,q}]$  such that the connected component   of $\{f<\lambda\}\cap \ft C_{1,q}$ containing
 $\bigcup_{\ell=1}^{m}\overline{\ft C^\ell }$
 is critical (see Definition~\ref{de.SSP}).
 We then define $\ft C^*$ as the connected component of 
 $\{f<\kappa^*\}\cap \ft C_{1,q}$ containing
 $\bigcup_{\ell=1}^{m}\overline{\ft C^\ell }$. By definition,   
 $\ft C^*$ is critical,  and, 
 from the construction of the maps $\mbf j$ and $\ft C_{\mbf j}$, it thus holds:
\begin{equation}\label{eq.contra-lemma}
\ft C^*\bigcap \cup_{j=1}^{k-1}\{x_{j,1},\ldots,x_{j,\ft N_j}\}\neq \emptyset.
\end{equation}
 Moreover, since all the $\ft C^\ell$'s are critical, and thus $\ft C^*\cap \ft U_1^{\ft{ssp}}(\Omega)\neq \emptyset$, the definitions of~$\kappa^*$ and~$\ft C^*$ together with item 2 in Proposition~\ref{pr.p2}
 applied to $\ft C=\ft C^*$ imply that
 $$  \kappa_k=\max_{y\in \ft C^*\cap \ft U_1^{\ft{ssp}}(\Omega)}f(y),$$
 where we recall that $ \kappa_k<\kappa^*$. 
 Therefore, using again item 2 in Proposition~\ref{pr.p2} with $\ft C=\ft C^*$,
 \begin{equation}\label{eq.contra-lemma2}
 \{f\le \kappa_k\}\cap \ft C^* \ \  \text{is connected \ \ and \ \ $ \ft C^*\cap \ft U_0\subset  \{f < \kappa_k\}$, }
 \end{equation}
 where the first claim follows from the fact that,
 for every $\lambda\in(\kappa_k,\kappa^*)$,  $\ft C^*\cap \{f<\lambda\}$ 
 is connected.

\noindent
To prove~\eqref{eq.lemma-recu}, one   argues by contradiction assuming that~\eqref{eq.lemma-recu} is not satisfied. 
It then follows from the local structure of the sublevel sets of a Morse function   given in Lemma~\ref{le.local-structure} 
that there exists some open set $O\subset \Omega$  such that
$O\cap \{f\leq \kappa_k\} = \bigcup_{\ell=1}^{m}\overline{\ft C^\ell }$ (see, in \cite{DLLN-saddle0}, the arguments used 
to prove Equation (50) there for more details).
In other words, the connected set $\bigcup_{\ell=1}^{m}\overline{\ft C^\ell }$ is open in $\{f\leq \kappa_k\}$ and
thus, since it is  closed and then closed in $\{f\leq \kappa_k\}$,
it is a connected component of $\{f\le \kappa_k\}$.
It thus follows 
  from~\eqref{eq.contra-lemma2}   that $\{f\le  \kappa_k\}\cap \ft C^* =\bigcup_{\ell=1}^{m}\overline {\ft C^\ell }$
 contains all the local minima of $f$ in $\ft C^*$. According to
\eqref{eq.contra-lemma}, 
   this implies,
since $\bigcup_{\ell=1}^{m}\pa{\ft C^\ell }$ does not contain any   local minimum of $f$, 
  that  at least one of the $\ft C^\ell $'s, $\ell  \in \{1,\ldots,m\}$,  does intersect $\cup_{j=1}^{k-1}\{x_{j,1},\ldots,x_{j,\ft N_j}\}$.  This leads to a contradiction  since  the $\ft C^\ell$'s ($\ell\in \{1,\ldots ,m\}$)  are  labelled at the $k$-th iteration ($k\ge 2$) and thus,   each  $\ft C^\ell$ ($\ell\in \{1,\ldots ,m\}$)  does not intersect $\cup_{j=1}^{k-1}\{x_{j,1},\ldots,x_{j,\ft N_j}\}$.  
This  
  concludes the proof of Lemma~\ref{le.debu-recurrence}. 
\end{proof}
%

\section{Quasi-modal construction}
\label{se.quasi}

The aim of this section is to construct, for every $x\in \ft U_0$,
a quasi-mode  $\psi_x$ associated with $x$, or more exactly
with $\ft C_{{\bf j}}(x)$, and whose energy in the limit $h \to 0$
will be shown to give the asymptotic behaviour of one of the $\ft m_0$ first eigenvalues of~$\Delta_{f,h}^D$ 
as exhibited in Theorems~\ref{th.thm2'} and \ref{th.thm2}.\medskip 

\noindent
More precisely, our quasi-modes $(\psi_x)_{x\in \ft U_0}$ are built as  suitable normalisations of auxiliary functions $(\phi_x)_{x\in \ft U_0}$,
which are first explicitly constructed 
in a neighborhood  of the elements of $\mbf j(x)\subset\overline \Omega$,
and then suitably extended to~$\overline \Omega$. 
This construction is partly inspired
by the construction made in \cite{DiLe}
when~$\Omega=\mathbb R^d$, see also~\cite{BEGK,landim2017dirichlet,nectoux2017sharp}. 
We also    refer to \cite{HKN,Lep,HeNi1,herau-hitrick-sjostrand-11,DLLN-saddle1,LeNi,michel2017small} for related constructions. 

\medskip

\noindent
This section is organized as follows. 
In Section~\ref{sec.coord-syst}, one introduces adapted coordinate systems
in a neighborhood of the elements of $\mbf j(x)$, where $x\in \ft U_0$, which then permit 
in Section~\ref{sec.local-construction}
to construct the auxiliary functions $\phi_{x}$ 
  in a neighborhood of $\mbf j(x)$.
 The  functions $(\phi_x)_{x\in \ft U_0}$ and $(\psi_x)_{x\in \ft U_0}$ are then defined in 
 Section~\ref{sec.global-construction}.
\medskip

\noindent
Before, let  us introduce the following  assumption which will be used throughout the rest of this work.
 \begin{manualassumption}{\bf(H2)}
\label{H2}
The function  $f:\overline \Omega\to \mathbb R$ is a $C^\infty$ Morse function
such that $\ft U_{0}\neq \emptyset$.   Moreover,    
for all $z\in \bigcup_{k=1}^{\ft N_1} \pa \ft C_{1,k} \cap \pa \Omega$ (see Definition~\ref{de.1}) such that $\vert \nabla f(z)\vert \neq 0$ (we recall that in this case, $z$ is a local minimum of $f|_{\pa \Omega}$ by item (a) in Proposition~\ref{pr.assumption}), 
\begin{equation}\label{eq.H2-non-dege}
z \text{ is a non degenerate local minimum  of $f|_{\pa \Omega}$ }. 
\end{equation}
\end{manualassumption}
\noindent
When $f$ satisfies the assumptions \autoref{H1}  and  \autoref{H2}, it holds 
\begin{equation}\label{eq.H2-cup-nbfini}
{\rm Card }\big(  \bigcup_{k=1}^{\ft N_1} \pa \ft C_{1,k} \cap \pa \Omega\big)<\infty
\ \ \ \text{and then}\ \ \ 
{\rm Card }\Big( \bigcup_{x\in \ft U_0} \mbf j(x) \Big)<\infty.
\end{equation} 
Indeed, 
${\rm Card }\big( \bigcup_{x\in \ft U_0} \mbf j(x) \cap \Omega\big)<\infty$ 
since $\bigcup_{x\in \ft U_0} \mbf j(x) \cap \Omega$  is composed of non degenerate saddle points of $f$ in $\Omega$ (see the construction of the map $\mbf j$ in Section~\ref{sec.j} and Definition~\ref{de.SSP-omega})
and,  according to item  (b) in Corollary~\ref{co.hypoA} and to \eqref{eq.H2-non-dege}, 
the elements of
\begin{align}
&\bigcup_{x\in \ft U_0} \mbf j(x) \cap \pa \Omega=\bigcup_{k=1}^{\ft N_1} \pa \ft C_{1,k} \cap \pa \Omega 
\label{eq.H2-cup-nbfini2}
\ \  \text{are non degenerate local minima  of $f|_{\pa \Omega}$}.
\end{align} 

\noindent
 In the rest of this section, one assumes that $f:\overline \Omega\to \mathbb R$  is a $C^\infty$ Morse function which satisfies the assumptions \autoref{H1} and~\autoref{H2}.





\subsection{Adapted coordinate systems} 
\label{sec.coord-syst}


 \noindent
  Let us recall that for any $x\in \ft U_{0}$, from
 the construction of the map $\mbf j$ made in Section~\ref{sec.j} and from \autoref{H1}--\autoref{H2}, $\mbf j(x)$ contains   
 saddle points of $f$ in $\overline \Omega$  (see Definition~\ref{de.SSP}) which are in finite number   and  may be  of two kinds:
 the elements $z\in\mbf j(x) \cap \pa \Omega$, such that either 
  $\vert \nabla f(z)\vert \neq 0$ or $\vert \nabla f(z)\vert =  0$, and the
  elements $z\in  \mbf j(x) \cap \Omega$, such that $\vert \nabla f(z)\vert =  0$.\medskip
%
  \noindent

\noindent
For any $x\in \ft U_0$ and $z\in {\bf j}(x)$,
we first construct a
coordinate systems  in a neighborhood of  
$z$
as follows.\medskip

\noindent{\bf 1.a) The case when $z\in\pa\Omega$ and $\vert \nabla f(z)\vert \neq 0$.}\medskip 

%


\noindent
Let us recall that, thanks to~\autoref{H2},    $z$ is in this case a non degenerate  local minimum of $f|_{\pa \Omega}$ and  that $\mu:=\pa_{\ft n_{\Omega}}f(z)>0$. 
 Then, according for example to~\cite[Section 3.4]{HeNi1}, there exists  a neighbourhood~$\ft V_z$   of $z$ in~$\overline \Omega$ and  a   coordinate system  
\begin{equation}\label{eq.cv-pa-omega-nablafnon0}
p\in \ft V_z \mapsto v=(v',v_d)=(v_1,\ldots,v_{d-1},v_{d})\in \mathbb R^{d-1}\times \mathbb R_-
\end{equation}
 such that 
\begin{equation}\label{eq.cv-pa-omega-nablafnon02}
v(z)=0\,,\ \ \{p\in \ft V_z, \, v_d(p)<0\}=  \Omega\cap  \ft V_z, \ \{p\in \ft V_z,\,  v_d(p)=0\}=\pa \Omega \cap \ft V_z,
\end{equation}
 and
%
\begin{equation}
\label{eq.cv-pa-omega-nablafnon04}
\forall i,j\in\{1,\dots,d\},\ \ \ g_z\big (\frac{\pa}{\pa v_i}(z),\frac{\pa}{\pa v_j}(z)\big )=\delta_{ij}
\quad\text{and}
\quad
\frac{\pa}{\pa v_d}(z) = \ft n_\Omega(z),
 \end{equation}
 with moreover, in the $v$ coordinates,
\begin{equation}\label{eq.cv-pa-omega-nablafnon03}
f(v',v_d)=f(0)+\mu v_d+ \frac 12 (v')^T \Hess f|_{\{v_d=0\}}(0)\,  v'. 
\end{equation} 


\noindent  For $\delta_1>0$ and $\delta_2>0$ small enough, one  then defines  the following neighborhood of~$z$ in $\pa \Omega$,
 \begin{equation}\label{eq.set1}
\ft V_{\pa \Omega}^{\delta_2}(z):=\{p\in \ft V_z, \, v_d(p)=0\text{ and }  \vert v'(p)\vert \le \delta_2\} \, \text{ (see~\eqref{eq.cv-pa-omega-nablafnon0}-\eqref{eq.cv-pa-omega-nablafnon02})}
 \end{equation}
 and the following neighbourhood of $z$  in $\overline \Omega$,
\begin{equation}\label{eq.vois-1-pc}
\ft V^{\delta_1,\delta_2}_{{\,  \overline \Omega  }}( z)=\big \{p\in \ft V_z,   \vert v'(p)\vert \le \delta_2  \text{ and } v_d(p)\in  [-2\delta_1, 0]\big  \}.
\end{equation}

\noindent{\bf 1.b) The case when $z\in\pa\Omega$ and $\vert \nabla f(z)\vert= 0$.} \medskip

\noindent
Let   $\ft V_z$ be a neighborhood   of $z$ in~$\overline \Omega$ and     let
\begin{equation}\label{eq.cv-pa-omega-nablaf=0}
p\in \ft V_z \mapsto v=(v',v_d)\in \mathbb R^{d-1}\times \mathbb R_-
\end{equation}
be a coordinate system such that  
 \begin{equation}\label{eq.cv-pa-omega-nablaf=02}
v(z)=0\,,\ \ \{p\in \ft V_z, \, v_d(p)<0\}=  \Omega\cap  \ft V_z \,,\ \ 
\{p\in \ft V_z,\,  v_d(p)=0\}=\pa \Omega \cap \ft V_z\,,
 \end{equation}
and
 \begin{equation}
\label{eq.cv-pa-omega-nablaf=03}
\forall i,j\in\{1,\dots,d\}\,,\ g_z\big (\frac{\pa}{\pa v_i}(z),\frac{\pa}{\pa v_j}(z)\big )=\delta_{ij}
\quad\text{and}
\quad
\frac{\pa}{\pa v_d}(z) = \ft n_\Omega(z)\,.
\quad
\end{equation}
Let us also recall that $z$ is a non degenerate  saddle point of $f$ in $\pa \Omega$ such that, according to  \autoref{H1},   $\ft n_\Omega(z)$ is an eigenvector associated with the negative eigenvalue~$\mu_{d}$ of $\Hess f(z)$.
Thus, denoting by $\mu_1,\dots,\mu_{d-1}$ the positive eigenvalues  of 
$\Hess f(z)$, 
%
the coordinates $v'=(v_{1},\dots,v_{d-1})$ can be chosen so that it holds, in the $v$ coordinates,
%
\begin{equation}\label{eq.cv-pa-omega-nablaf=04}
f(v)\ =\ f(0)+  \frac 12\sum_{j=1}^d \mu_j \, v_j^2+ O(\vert v\vert^3)
\ =\ 
f(0)+  \frac 12\sum_{j=1}^{d-1} |\mu_j| \, v_j^2 - \frac 12 |\mu_d| \, v_d^2+ O(\vert v\vert^3)
\,.
\end{equation}
Therefore, up to choosing   $\ft V_z$  again smaller, one can assume
that
\begin{equation}\label{eq.cv-pa-omega-nablaf=07}
\argmin_{\overline{ \ft V_z }} \big( f(v) +\vert \mu_d\vert v_d^2\big)=\{ z\}.
\end{equation} 
%

\noindent
 For $\delta_1>0$ and $\delta_2>0$ small enough, one defines  the following neighborhood of $z$ in~$\pa \Omega$,
\begin{equation}\label{eq.set11}
\ft V_{\pa \Omega}^{\delta_2}(z):=\{p\in \ft V_z, \, v_d(p)=0\text{ and }  \vert v'(p)\vert \le \delta_2\} \, \text{ (see~\eqref{eq.cv-pa-omega-nablaf=0}-\eqref{eq.cv-pa-omega-nablaf=02})},
 \end{equation}
 and  the following neighbourhood of $z$  in $\overline \Omega$,
\begin{equation}\label{eq.vois-11-pc}
\ft V^{\delta_1,\delta_2}_{{\,  \overline \Omega  }}( z)=\big \{p\in \ft V_z,   \vert v'(p)\vert \le \delta_2  \text{ and } v_d(p)\in  [-2\delta_1, 0]\big  \}.
\end{equation}

\noindent{\bf 2. The case when $z\in\Omega$.}\medskip


\noindent
Let us recall that in this case~$z$ is a non degenerate saddle point of $f$ in $\Omega$.  Let $(\ft e_1,  \ldots, \ft e_d)$ be an orthonormal  basis of eigenvectors of $\Hess  f(z)$ associated with its eigenvalues $(\mu_1,\ldots,\mu_d)$ with $\mu_d<0$ and, for all $j\in \{1,\ldots,d-1\}$, $\mu_j>0$.  
Then, since $\ft e_{d}$ is normal to $\ft W_{+}(z)$, as in the case when $z\in\pa\Omega$
and $\vert \nabla f(z)\vert= 0$ and up to replacing $\ft e_{d}$ by $-\ft e_{d}$, 
there exists  a coordinate system
\begin{equation}\label{eq.cv-omega}
p\in \ft V_z \mapsto v=(v',v_d)\in   \mathbb R^{d-1 }\times \mathbb R
\end{equation}
   such that  
\begin{equation}\label{eq.cv-omega2}
v(z)=0\,,\    \ft C_{{\bf j}}(x)\cap    \ft V_z\subset \{p\in \ft V_z, \, v_d(p)<0\} \,,\    \{p\in \ft V_z,\,  v_d(p)=0\}=\ft W_+(z)\cap    \ft V_z,  
\end{equation}
and
%
 \begin{equation}
\label{eq.cv-pa-omega-nablaf=05'}
\forall i,j\in\{1,\dots,d\}\,,\ \ \ g_z\big (\frac{\pa}{\pa v_i}(z),\frac{\pa}{\pa v_j}(z)\big )=\delta_{ij}
\quad\text{and}
\quad
\frac{\pa}{\pa v_d}(z) = \ft e_{d},
\quad
\end{equation}
with moreover, in the $v$ coordinates,
\begin{equation}\label{eq.cv-omega5}
f(v)\ =\ f(0)+  \frac 12\sum_{j=1}^d \mu_j \, v_j^2+ O(\vert v\vert^3)
\ =\ 
f(0)+  \frac 12\sum_{j=1}^{d-1} |\mu_j| \, v_j^2 - \frac 12 |\mu_d| \, v_d^2+ O(\vert v\vert^3).
\end{equation}
 
\noindent
 Then, up to choosing   $\ft V_z$   smaller, one can assume
that
\begin{equation}\label{eq.cv-pa-omega-nablaf=07oubli}
\argmin_{\overline{ \ft V_z }} \big( f(v) +\vert \mu_d\vert v_d^2\big)=\{ z\}.
\end{equation} 
 
 
\noindent
Then,  for $\delta_1>0$ and $\delta_2>0$ small enough, one defines the following neighbourhood of $z$ in $\ft W^+(z)$ (see~\eqref{eq.cv-omega} and~\eqref{eq.cv-omega2}),
  \begin{equation}\label{eq.set2}
\ft V_{\ft W^+}^{\delta_2}(z):=\{p\in \ft V_z,  \, v_d(p)=0\text{ and }  \vert v'(p)\vert \le  \delta_2\}\subset \ft W^+(z),
 \end{equation}
  and the following neighbourhood of $z$  in $ \Omega$,
\begin{equation}\label{eq.vois-1-pc-omega}
\ft V^{\delta_1,\delta_2}_{{\, \overline \Omega  }}( z)=\big \{p\in \ft V_z,   \vert v'(p)\vert \le \delta_2 \text{ and }  v_d\in [-2\delta_1, 2\delta_1]\big  \}.
\end{equation}
 Notice that one has:
\begin{equation}\label{eq.min-u}
 \argmin_{\ft V_{\ft W^+}^{\delta_2}(z)}f= \{z\}.
 \end{equation}

\ 
  
  \noindent{\bf Some properties of these coordinate systems.}\medskip

 \noindent 
The sets defined in \eqref{eq.vois-1-pc},~\eqref{eq.vois-11-pc}, and \eqref{eq.vois-1-pc-omega}    are cylinders centred at~$z$ in the respective system of coordinates.
Up to choosing $\delta_1>0$ and $\delta_2>0$ smaller, one can assume that  all these cylinders are two by two disjoint.  Schematic representations of these sets  introduced in~\eqref{eq.set1}--\eqref{eq.vois-1-pc-omega} are given in Figures~\ref{fig:cadre-sigma_k},~\ref{fig:cadre-sigma_k-pc} and~\ref{fig:cadre-sigma_k-pcomega}. 
 \medskip
 
 \noindent
Let us conclude this section by giving several  properties of the sets previously introduced   which will be needed  for upcoming computations.  Let us recall that, from~\eqref{eq.j-x1}, when $z\in \mbf j(x)$ for some $x\in \ft U_0$, it holds  $f(z)>f(x)$.  Moreover, by construction of the map $\mbf j$ in Section~\ref{sec.j}, it obviously holds $\ft U_0\cap \cup_{x\in \ft U_0} \mbf j(x)=\emptyset$. 
Therefore, up to choosing $\delta_1>0$ and $\delta_2>0$  small enough, the following properties are satisfied:
  \begin{enumerate}
 \item When   $z\in \pa \Omega \cap \mbf j(x)$  for some $x\in \ft U_0$,  it holds
  \begin{equation}\label{eq.inclur3}
  \min_{\ft V^{\delta_1,\delta_2}_{\overline \Omega}(z)}f>f(x), \  \    \ft V^{\delta_1,\delta_2}_{\overline \Omega}(z)\cap \ft U_0=\emptyset,  
  \end{equation}
  and
   \begin{equation}\label{eq.min-v}
   \argmin_{\ft V^{\delta_2}_{\pa \Omega}(z)}f= \{z\} \  \text{ (which follows from~\eqref{eq.H2-cup-nbfini2})}.
  \end{equation}
 
 \item When $z\in\Omega \cap \mbf j(x)$ for some $x\in \ft U_0$,    it holds:
 \begin{equation}\label{eq.inclur4}
  \min_{\ft V^{\delta_1,\delta_2}_{\overline \Omega}(z)}f> f(x) \,  \text{ and } \, \ft V^{\delta_1,\delta_2}_{\overline \Omega}(z)\cap \ft U_0=\emptyset.
  \end{equation} 
 \end{enumerate} 
 \noindent
  The parameter $\delta_2>0$ is  now kept fixed.
Finally, using \eqref{eq.min-u}, \eqref{eq.min-v},
and up to choosing $\delta_1>0$ smaller, there exists $r>0$  such that (see Figures~\ref{fig:cadre-sigma_k},~\ref{fig:cadre-sigma_k-pc} and~\ref{fig:cadre-sigma_k-pcomega}):
\begin{enumerate}
 \item For all  $z\in \pa \Omega \cap \mbf j(x)$  for some  $x\in \ft U_0$,  %
%
\begin{equation}\label{eq.inclur1}
 \big \{p\in \ft V_z,   \vert v'(p)\vert =\delta_2  \text{ and } v_d(p)\in  [-2\delta_1, 0]\big  \} \subset\{f\ge f(z)+r\}.
\end{equation}

%
%

  \item For all    $z\in \Omega \cap \mbf j(x)$ for some  $x\in \ft U_0$,
  \begin{equation}\label{eq.inclur2}
 \big \{p\in \ft V_z,     \vert v'(p)\vert =\delta_2 \text{ and }  v_d\in [-2\delta_1, 2\delta_1]\big  \}\subset\{f\ge f(z)+r\}.
\end{equation}
 \end{enumerate}
 The parameter $\delta_1>0$ is  now kept fixed.

%
%
\begin{figure}[h!]
\begin{center}
\begin{tikzpicture}[scale=0.7]
\tikzstyle{vertex}=[draw,circle,fill=black,minimum size=6pt,inner sep=0pt]
\draw[->] (0,-4.6)--(0,4.5);
\draw[->] (-4,0)--(3,0);
 \draw (3.4,0) node[]{$v_d$};
\draw[ultra thick]  (0,3.5)--(-2.4,3.5);
\draw[ultra thick]  (0,-3.5)--(-2.4,-3.5);
\draw (-2.4,3.5)--(-2.4,-3.5);
\draw  (0,3.5)--(0,-3.5);
\draw[<->] (-0.1,3.8)--(-2.4,3.8);
 \draw (-1,4.2) node[]{$2\delta_1$};
 \draw (0,5) node[]{$v'$};
  \draw (0,-5) node[]{$\pa \Omega$};
  \draw (-4.8,-1.7) node[]{$ \big\{f<f(\mbf j(x))\big\}$};
  \draw[thick,dashed] (2.3,1.8)--(3.3,1.8);
   \draw (5.2,1.8)node[]{$\big \{f=f(\mbf j(x))\big\}$};
      \draw[ultra thick] (2.3,2.6)--(3.3,2.6);
      \draw (7,2.6)node[]{$\{\vert v'\vert=\delta_2 \text{ and } v_d\in [-2\delta_1,0]\}$};
      \draw (0,0)  node[vertex,label= north east: {$z$}](v){};
            
             \draw[ <->]  (1.3,-0.2)--(1.3,-3.5);
              \draw (2,-1.6) node[]{  \small{$\delta_2$}};
      \draw (-8,1.1) node[]{$ \Omega$};
    \draw (-5,0) node[]{$ \ft C_{\mbf j}(x)$};  
    \draw (-4.5,4.3) node[]{$ \big\{f>f(\mbf j(x))\big\}$};
      \draw (-4.5,-4) node[]{$\big \{f>f(\mbf j(x))\big\}$};
  \draw[thick, dashed]   (-7,3) ..controls  (2.3,2)  and (2.3,-2)  ..  (-7,-3) ;
   \draw (-1.2,2.7 )node[]{\footnotesize{$\ft V^{\delta_1,\delta_2}_{\overline \Omega}( z)$}};
\end{tikzpicture}

\caption{Schematic representation of the cylinder $\ft V^{\delta_1,\delta_2}_{{\,  \overline \Omega  }}( z)$,
in the $v$-coordinates, 
when $z\in \mbf j(x) \cap \pa \Omega$ (for some  $x\in \ft U_0$) is such that  $\vert \nabla f(z)\vert \neq 0$.  One recalls that $ \mbf j(x)\subset \pa \ft C_{\mbf j}(x)$ and that,
in this case, $z$ is a non degenerate   local minimum of $f|_{\pa \Omega}$ and $\pa_{{\ft n}_{\Omega}} f(z)>0$. }
 \label{fig:cadre-sigma_k}
 \end{center}
\end{figure}
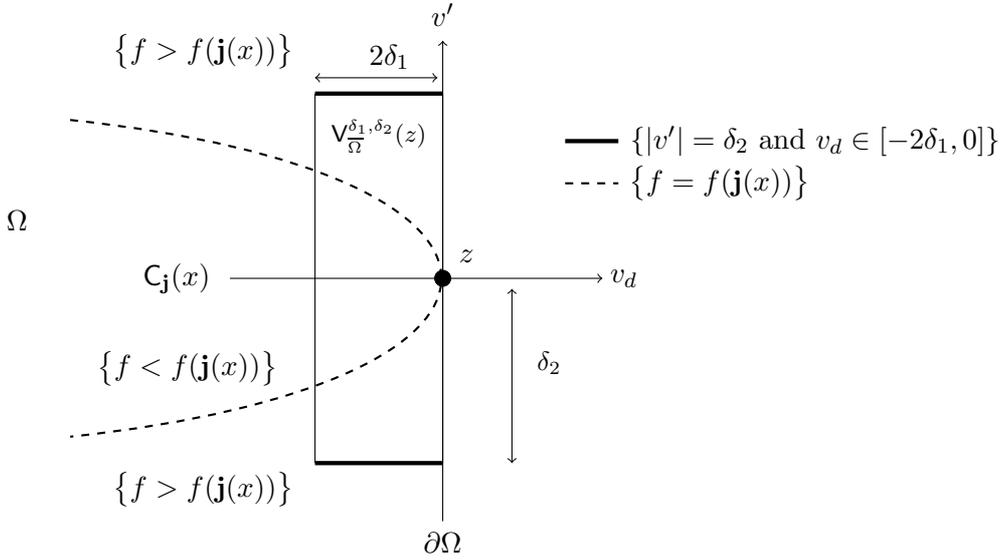

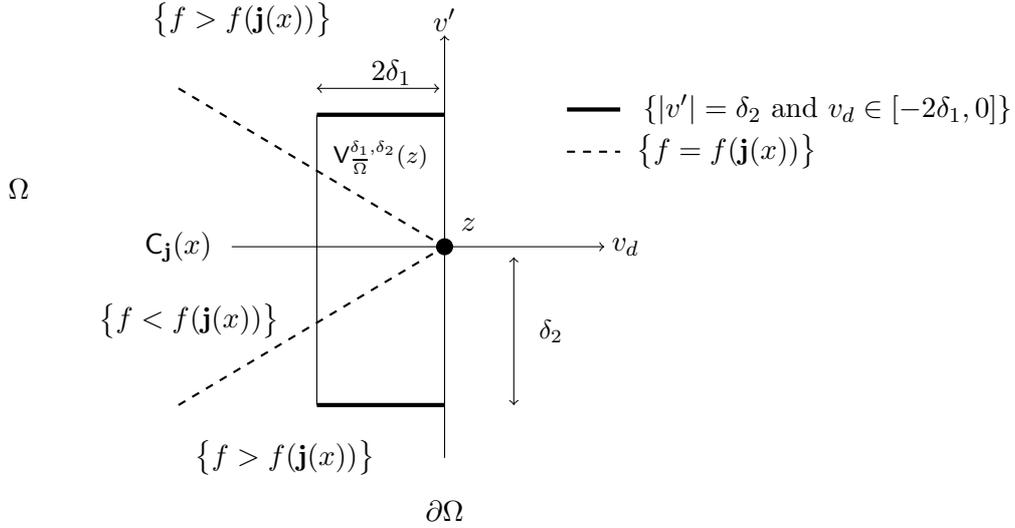
\begin{figure}[h!]
\begin{center}
\begin{tikzpicture}[scale=0.7]
\tikzstyle{vertex}=[draw,circle,fill=black,minimum size=6pt,inner sep=0pt]
\draw[->] (0,-4)--(0,4);
\draw[->] (-4,0)--(3,0);
 \draw (3.4,0) node[]{$v_d$};
\draw[ultra thick]  (0,2.5)--(-2.4,2.5);
\draw[ultra thick]  (0,-3)--(-2.4,-3);
\draw (-2.4,2.5)--(-2.4,-3);
\draw  (0,2.5)--(0,-3);
\draw[<->] (-0.1,3)--(-2.4,3);
 \draw (-1,3.3) node[]{$2\delta_1$};
 
 \draw (0,4.3) node[]{$v'$};
  \draw (0,-5) node[]{$\pa \Omega$};
      \draw (-1.2,1.7 )node[]{\footnotesize{$\ft V^{\delta_1,\delta_2}_{{\,  \overline \Omega  }}( z)$}};
  \draw (0,0)  node[vertex,label= north east: {$z$}](v){};
  \draw (-4.8,-1.4) node[]{$ \big\{f<f(\mbf j(x))\big\}$};
  \draw[thick,dashed] (2.3,1.8)--(3.3,1.8);
   \draw (5.3,1.8)node[]{$\big \{f=f(\mbf j(x))\big\}$};
      \draw[ultra thick] (2.3,2.6)--(3.3,2.6);
       \draw (7.2,2.6)node[]{$\{\vert v'\vert=\delta_2 \text{ and } v_d\in [-2\delta_1,0]\}$};
                 
   \draw[ <->]  (1.3,-0.2)--(1.3,-3);
              \draw (2,-1.6) node[]{  \small{$\delta_2$}};
      \draw (-8,1.1) node[]{$ \Omega$};
    \draw (-5,0) node[]{$ \ft C_{\mbf j}(x)$};
  
    \draw (-3.8,4.3) node[]{$ \big\{f>f(\mbf j(x))\big\}$};
      \draw (-3,-4) node[]{$\big \{f>f(\mbf j(x))\big\}$};
        \draw[thick, dashed]  (-5,3) to (0,0) ;
   \draw[thick, dashed]  (-5,-3) to (0,0) ;

\end{tikzpicture}

\caption{Schematic representation of the cylinder $\ft V^{\delta_1,\delta_2}_{{\,  \overline \Omega  }}( z)$,
in the $v$-coordinates,  
when $z\in \mbf j(x) \cap \pa \Omega$ (for some  $x\in \ft U_0$) is such that  $\vert \nabla f(z)\vert=0$. One recalls that $ \mbf j(x)\subset \pa \ft C_{\mbf j}(x)$ and that,
in this case,  $z$ is a non degenerate saddle point of $f$ and a non degenerate local minimum of $f|_{\pa \Omega}$. 
  }
 \label{fig:cadre-sigma_k-pc}
 \end{center}
\end{figure}

 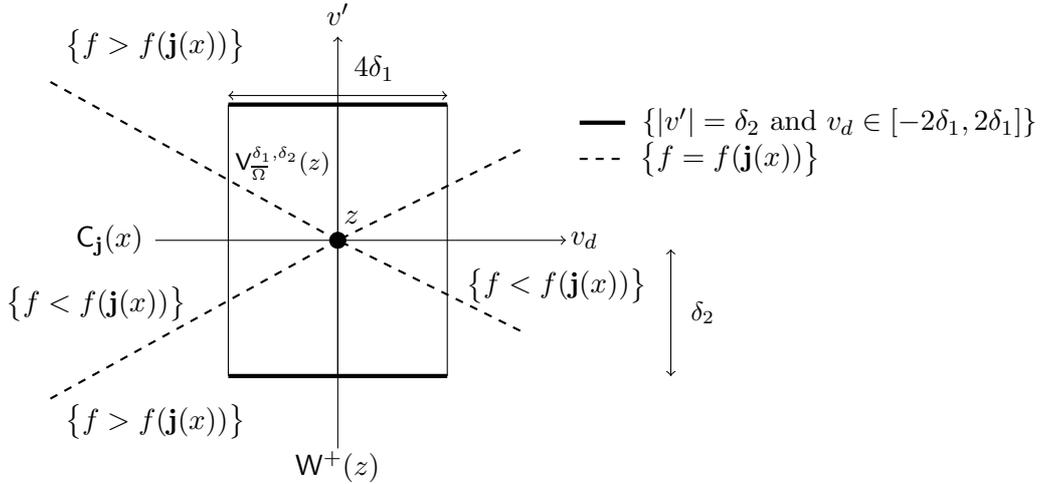
\begin{figure}[h!]
\begin{center}
\begin{tikzpicture}[scale=0.6]
\tikzstyle{vertex}=[draw,circle,fill=black,minimum size=6pt,inner sep=0pt]
\draw[->] (0,-4.6)--(0,4.5);
\draw[->] (-4,0)--(5,0);
 \draw (5.4,0) node[]{$v_d$};
\draw[ultra thick]  (2.4,3)--(-2.4,3);
\draw[ultra thick]  (2.4,-3)--(-2.4,-3);
\draw (-2.4,3)--(-2.4,-3);
\draw (2.4,3)--(2.4,-3);
\draw  (0,3)--(0,-3);
\draw[<->] (-2.4,3.2)--(2.4,3.2);
 \draw (0.8,3.8) node[]{$4\delta_1$};
  \draw (0.3,0.5) node[]{$z$};

 \draw (0,5) node[]{$v'$};
  \draw (0,-5) node[]{$\ft W^+(z)$};
      \draw (-1.2,1.7 )node[]{\footnotesize{$\ft V^{\delta_1,\delta_2}_{\overline \Omega}( z)$}};
  \draw (0,0)  node[vertex,label= north east: {$$}](v){};
  \draw (-5.3,-1.4) node[]{$ \big\{f<f(\mbf j(x))\big\}$};
    \draw (4.8,-1) node[]{$ \big\{f<f(\mbf j(x))\big\}$};
  \draw[thick,dashed] (5.3,1.8)--(6.3,1.8);
   \draw (8.6,1.8)node[]{$\big \{f=f(\mbf j(x))\big\}$};
      \draw[ultra thick] (5.3,2.6)--(6.3,2.6);
      \draw (11,2.6)node[]{$\{\vert v'\vert=\delta_2 \text{ and } v_d\in [-2\delta_1,2\delta_1]\}$};

    \draw (-5,0) node[]{$  \ft C_{\mbf j}(x)$};
  
    \draw (-4,4.3) node[]{$ \big\{f>f(\mbf j(x))\big\}$};
      \draw (-4,-4) node[]{$\big \{f>f(\mbf j(x))\big\}$};
        \draw[thick, dashed]  (-6.3,3.5) to (0,0) ;
   \draw[thick, dashed]  (-6.3,-3.5) to (0,0) ;
   \draw[thick, dashed]  (4,2) to (0,0) ;
   \draw[thick, dashed]  (4,-2) to (0,0) ;
   
      \draw[ <->]  (7.3,-0.2)--(7.3,-3);
              \draw (8,-1.6) node[]{  \small{$\delta_2$}};
   
\end{tikzpicture}

\caption{Schematic representation of the cylinder $\ft V^{\delta_1,\delta_2}_{{\,  \overline \Omega  }}(z)$,
in the $v$-coordinates,  
when $z\in \mbf j(x)\cap \Omega$ for some  $x\in \ft U_0$. One recalls that $ \mbf j(x)\subset \pa \ft C_{\mbf j}(x)$ and that,
in this case,  $z$ is a  separating  saddle point of $f$ in~$\Omega$ (see Definition~\ref{de.SSP-omega}).  }
 \label{fig:cadre-sigma_k-pcomega}
 \end{center}
\end{figure}
%
%


\subsection{Quasi-modal construction 
near the elements of $ \cup_{x\in \ft U_0}\mbf j(x)$}
\label{sec.local-construction}

Let us introduce  an even cut-off  function $\chi\in C^\infty(\mathbb R,[0,1])$   such that 
\begin{equation}\label{eq.chi-cut}
\text{supp } \chi\subset  [- \delta_1 ,  \delta_1] \, \text{ and } \,   \chi=1 \text{ on } \Big [-\frac{\delta_1}{2}, \frac{\delta_1}{2} \Big].
  \end{equation}
Let  $z\in \bigcup_{x\in \ft U_0} \mbf j(x)$. Then,  the  function $\varphi_z$ associated with $z$ and $x$ is defined as follows: 
\begin{enumerate}
 
\item  Let   us assume that $z\in \pa \Omega$.

\begin{enumerate}
\item When  $\vert \nabla f(z)\vert \neq 0$, one defines (see~\eqref{eq.cv-pa-omega-nablafnon0}, \eqref{eq.cv-pa-omega-nablafnon02}, and \eqref{eq.vois-1-pc}):
\begin{equation}\label{eq.qm-local1}
\forall v=(v',v_d)  \in v\big (\ft V^{\delta_1,\delta_2}_{\,\overline \Omega}(z)\big), \  \  \varphi_{z} (v',v_d):= \frac{  \int_{v_d}^0\chi(t)e^{\frac 2h\mu \, t} dt}{\int_{-2\delta_1}^0\chi(t)e^{\frac 2h\mu \, t} dt},
  \end{equation}
where we recall that  $\mu=\pa_{{\ft n}_{\Omega}}f(z)>0$. 
Note that
the function~$ \varphi_{z}$ only depends on the variable $v_{d}$. Moreover, it holds (see~\eqref{eq.chi-cut}),
      \begin{equation}\label{eq.qm-local1-property1}
       \left\{
    \begin{array}{ll}
      \varphi_{z}\in   C^\infty\big (v\big (\ft V^{\delta_1,\delta_2}_{\,\overline \Omega}(z)\big),[0,1]\big ) \text{ and }  \\
      \forall (v',v_d)\in v\big (\ft V^{\delta_1,\delta_2}_{\,\overline \Omega}(z)\big), \, \varphi_{z}(v',v_d)=1 \, \text{ if }    v_d \in [-2\delta_1,-\delta_1].
    \end{array}
\right.
\end{equation}
   
   \item When  $\vert \nabla f(z)\vert =0$,   
  one defines (see~\eqref{eq.cv-pa-omega-nablaf=0}, \eqref{eq.cv-pa-omega-nablaf=02}, and \eqref{eq.vois-11-pc}):
   \begin{equation}\label{eq.qm-local12}
\forall v=(v',v_d)  \in  v\big (\ft V^{\delta_1,\delta_2}_{\,\overline \Omega}(z)\big), \  \  \varphi_{z} (v',v_d):=\frac{  \int_{v_d}^0\chi(t)e^{-\frac 1h\vert \mu_d\vert  \, t^2}  dt}{\int_{-2\delta_1}^0\chi(t)\, e^{-\frac 1h\vert \mu_d\vert  \, t^2}  dt},
  \end{equation}
where we recall that   $\mu_d<0$ is the negative eigenvalue of  $\text{Hess} f (z)$.
The function $\varphi_{z}$ thus  only depends on the variable $v_d$
and it holds
      \begin{equation}\label{eq.qm-local1-property12}
  \left\{
    \begin{array}{ll}
    \varphi_{z}\in   C^\infty\big (v\big (\ft V^{\delta_1,\delta_2}_{\,\overline \Omega}(z)\big),[0,1]\big ) \text{ and } \\
\forall (v',v_d)\in v\big (\ft V^{\delta_1,\delta_2}_{\,\overline \Omega}(z)\big),\, \varphi_{z}(v',v_d)=1 \text{  if  } v_d \in [-2\delta_1,-\delta_1].
 \end{array}
\right.
\end{equation}
  \end{enumerate}
 \item  Let   us assume that $z\in  \Omega$. We recall that in this case, $z$ is a  separating  saddle point of $f$ in~$\Omega$ (by construction of the map $\mbf j$, see also Definition~\ref{de.SSP-omega}). Then, one defines the function (see \eqref{eq.cv-omega}, \eqref{eq.cv-omega2}, and \eqref{eq.vois-1-pc-omega}):
 \begin{equation}\label{eq.qm-local2}
\forall v=(v',v_d)  \in  v\big (\ft V^{\delta_1,\delta_2}_{\,\overline \Omega}(z)\big), \  \varphi_{z} (v',v_d):=\frac{ \int_{v_d}^{2\delta_1}\chi(t)   \, e^{-\frac 1h\vert \mu_d\vert  \, t^2} dt  }{ \int_{-2\delta_1}^{2\delta_1}\chi(t)\, e^{-\frac 1h\vert \mu_d\vert  \, t^2} dt  },
  \end{equation}
  where $\mu_d$ is the negative eigenvalue of $\Hess  f(z)$. 
  Again, $\varphi_{z}$
  only depends on the variable $v_d$ and    
 it holds:
    \begin{equation}\label{eq.qm-local1-property3}
\varphi_{z}\in   C^\infty\Big (v\big (\ft V^{\delta_1,\delta_2}_{\,\overline \Omega}(z)\big),[0,1]\Big )
\end{equation}
and  for all $(v',v_d)\in v\big (\ft V^{\delta_1,\delta_2}_{\,\overline \Omega}(z)\big)$, 
   \begin{equation}\label{eq.qm-local1-property4}
\varphi_{z}(v',v_d)=1 \text{ if  } v_d \in [-2\delta_1,-\delta_1]  \, \text{ and } \varphi_z(v',v_d)=0 \, \text{ if } v_d \in [\delta_1,2\delta_1].
   \end{equation}
   \end{enumerate}

\subsection{Construction of $\ft m_0$ quasi-modes for $\Delta_{f,h}^D$}

\label{sec.global-construction}

\noindent
In the following, one considers some arbitrary
$$x\in \ft U_0.$$

\noindent
Let us recall the geometry of $f$ near the boundary of the   critical component  $\pa \ft C_{\mbf j}(x)$. 
Let us consider a point~$p\in \pa  \ft C_{\mbf j}(x)\setminus\,  \mbf j(x)$. Since  $\mbf j(x)= \pa\ft C_{\mbf j}(x)\cap    \ft U_1^{\ft{ssp}}(\overline \Omega)$  and $\pa\ft C_{\mbf j}(x)\cap \pa \Omega\subset \mbf j(x)$, $p\in \Omega\setminus   \ft U_1^{\ft{ssp}}(\Omega)$. Thus, there are two possible cases:
\begin{itemize}
\item Either $p$ is a saddle point of $f$ in $\Omega$. From Lemma~\ref{le.local-structure}, 
$\{f<f(\mbf j(x))\}\cap B(p,r)$ has then, for~$r>0$ small enough, two connected components
which are included in~$ \ft C_{\mbf j}(x)$,  since    $p$  is not separating (see Figure~\ref{fig:S-b}).
\item Or  $p$ is not a saddle point of $f$ in $\Omega$. According to Lemma~\ref{le.local-structure},  
$\{f<f(\mbf j(x))\}\cap B(p,r)$ is then connected for~$r>0$ small enough and is thus included in $\ft C_{\mbf j}(x)$. 
\end{itemize}
\begin{sloppypar}
\noindent
In conclusion,  when  $p\in \pa  \ft C_{\mbf j}(x)\setminus\, \mbf j(x)$,  $\{f<f(\mbf j(x))\}\cap B(p,r)$ is included in $\ft C_{\mbf j}(x)\cap \Omega$
for~$r>0$ small enough.
Moreover,  one  constructed in~\eqref{eq.vois-1-pc},~\eqref{eq.vois-11-pc}, and~\eqref{eq.vois-1-pc-omega}, disjoint   cylinders in  neighborhoods of each  $z\in \cup_{y\in \ft U_0}\mbf j(y)$ which satisfy
\eqref{eq.inclur3} and
\eqref{eq.inclur4}--\eqref{eq.inclur2}. 
This makes possible the construction used in the definition below.
\end{sloppypar}
\begin{definition}\label{de.omega1}
Let $f:\overline \Omega\to \mathbb R$  be a $C^\infty$ Morse function which satisfies \autoref{H1} and~\autoref{H2}. 
Then  $\ft U_{0}\neq \emptyset$ and, for each $x\in \ft U_0$, there exist two $C^\infty$ connected open sets~$\ft \Omega_1(x)$ and~$\ft \Omega_2(x)$ of $\Omega$  satisfying the following properties:
\begin{enumerate}
\item For all $x\in \ft U_0$, it holds
$$\overline{\ft C_{\mbf j}(x)} \subset \ft \Omega_1(x)\cup \pa\Omega\  \text{ and }\ \argmin_{ \overline{\ft \Omega_{1}(x)}}f =\argmin_{ \overline{ \ft C_{\mbf j}(x)} } f.$$ 
\item For all $x\in \ft U_0$, 
 $\overline{\ft \Omega_2(x)}\subset \ft \Omega_1(x)$ and 
 the strip $\overline{\ft \Omega_{1}(x)}\setminus \ft\Omega_{2}(x)$ equal:
\begin{equation}\label{stripS}
 \overline{\ft \Omega_{1}(x)}\setminus \ft\Omega_{2}(x)= 
 \bigcup \limits_{z\in \mbf j(x)}  \ft V^{\delta_1,\delta_2}_{\,\overline \Omega}(z)   \  \bigcup \   \ft O_1(x),
\end{equation}
where  there exists $c>0$ such that:
\begin{equation}\label{eq.SS}
\forall q\in  \ft O_1(x), \,\  f(q)\ge f(\mbf j(x)) + c.
\end{equation}
Notice that item 1,~\eqref{stripS}, \eqref{eq.SS}, and the first statements in~\eqref{eq.inclur3} and in \eqref{eq.inclur4}  imply that  $\argmin \limits_{ \overline{\ft \Omega_1(x)}}f=\argmin\limits_{ \overline{\ft \Omega_{2}(x)}}f=\argmin\limits_{ \ft C_{\mbf j}(x)}f$.

\item For all $x,y\in \ft U_0$ such that $x\neq y$, it holds (depending on the two possible cases described in items 4.(i) and 4.(ii) in   Section~\ref{sec.j}):
\begin{enumerate}
\item[(i)] If ${\mbf j}(y) \, \cap \,  \mbf j (x)=\emptyset$:
$$
\left\{
    \begin{array}{ll}
    
       \text{either}\  \overline{\ft C_{\mbf j}(y)} \, \cap \, \overline{ \ft C_{\mbf j}(x)}=\emptyset \, \text{and }  \,  \overline{\ft \Omega_{1}(x)}\cap  \overline{\ft \Omega_{1}(y)} =\emptyset,\\   
       \text{or, up to switching $x$ and $y$, }   \overline{\ft C_{\mbf j}(y)}\subset  \ft C_{\mbf j}(x)\, \text{and }  \,  \overline{\ft \Omega_{1}(y) }\subset  \ft \Omega_{2}(x).
    \end{array}
\right.
$$
\item[(ii)]  If  ${\mbf j}(y) \, \cap \,  \mbf j (x)\neq \emptyset$ (in this case, one recalls that $f({\mbf j}(y))=f({\mbf j}(x))$ and thus,~$\ft C_{\mbf j}(y)$ and $\ft C_{\mbf j}(x)$ are two connected components of   $\{f<f(\mbf j(x))\}$), then: 
$$
\overline{\ft \Omega_{1}(x)}\cap  \overline{\ft \Omega_{1}(y)}=  \bigcup_{z\in \mbf j(y) \, \cap \,  \mbf j(x)} \!\!\!\! \ft V^{\delta_1,\delta_2}_{\, \overline \Omega}(z)\  \bigcup \ \ft O_2(x,y) , 
$$
 where  $\ft O_2(x,y)\subset \ft O_{1}(x)$ and $\ft O_2(x,y) \cap   \ft V^{\delta_1,\delta_2}_{\, \overline \Omega}(z)=\emptyset$ for all $z\in \mbf j(y) \, \cup \,  \mbf j(x)$.

\end{enumerate}

\end{enumerate}

\end{definition}

\noindent
For $x\in \ft U_0$, schematic representations  of $\ft \Omega_{1}(x)$, $\ft \Omega_{2}(x)$, and $\ft O_1(x)$  are  given in Figures~\ref{fig:S} and~\ref{fig:S-b}. 
With the help of the sets $\ft \Omega_1(x)$ and $ \ft \Omega_2(x)$ introduced in Definition~\ref{de.omega1}, one defines  a smooth  function $\phi_x:\overline \Omega \to [0,1]$ associated with each $x\in \ft U_0$  as follows.

\begin{definition}\label{de.psix}
Let $f:\overline \Omega\to \mathbb R$  be a $C^\infty$ Morse function which satisfies \autoref{H1} and~\autoref{H2}.  For each  $x\in \ft U_0$, a function  
$\phi_x :\overline \Omega\to \mathbb [0,1]$ is constructed as follows:
\begin{enumerate}
\item For every $z\in \mbf j(x)$, $\phi_x$  is defined on the cylinder $\ft V^{\delta_1,\delta_2}_{\,\overline \Omega}(z)$ 
(see \eqref{eq.vois-1-pc}, \eqref{eq.vois-11-pc}, and \eqref{eq.vois-1-pc-omega})
by
 \begin{equation}\label{eq.qm-local1x}
\forall p \in   \ft V^{\delta_1,\delta_2}_{\,\overline \Omega}(z), \  \phi_x(p):= \varphi_{z} (v(p)),  \  
\text{ see \eqref{eq.qm-local1},~\eqref{eq.qm-local12},   and  \eqref{eq.qm-local2}}.
  \end{equation}

\item From~\eqref{eq.qm-local1-property1},~\eqref{eq.qm-local1-property12},~\eqref{eq.qm-local1-property3},~\eqref{eq.qm-local1-property4}, and  the facts that $ \overline{\ft \Omega_2(x)}\subset  \ft \Omega_1(x)$  (see Definition~\ref{de.omega1}) and~\eqref{stripS} holds,    
$\phi_x$   can be extended to $\overline \Omega$  such that 
\begin{equation}\label{eq.psix=10}
 \phi_x=0 \text{ on } \overline \Omega\setminus\ft \Omega_1(x), \ \ \ \phi_x=1 \text{ on }  {\ft \Omega_2(x)}, \text{ and } \phi_x\in C^\infty(\overline \Omega,[0,1]). 
\end{equation}
Notice that~\eqref{eq.psix=10} implies that:  
\begin{equation}\label{eq.psix=nabla}
\supp d   \phi_x \subset \overline{\ft \Omega_1(x)}\setminus \ft \Omega_2(x).
\end{equation}
Finally, in view of~\eqref{eq.qm-local1},~\eqref{eq.qm-local12},~\eqref{eq.qm-local2}, and~\eqref{stripS},
 $\phi_x$ can be chosen on~$\ft O_1$ such that for some $C>0$ and   for  every $h$ small enough
 (see indeed \eqref{eq.chiYY2},~\eqref{eq.chiUU2}, and~\eqref{eq.chiVV2} below):
\begin{equation} 
\label{eq.nablapsix=0}
\forall \alpha \in \mathbb N^d, \, \vert \alpha \vert \in \{1,2\}, \, \big  \Vert \pa^\alpha  \phi_x   \big \Vert_{L^\infty(\ft O_1(x))} \le \frac{C}{h^2}. 
\end{equation}
\end{enumerate}

\end{definition}

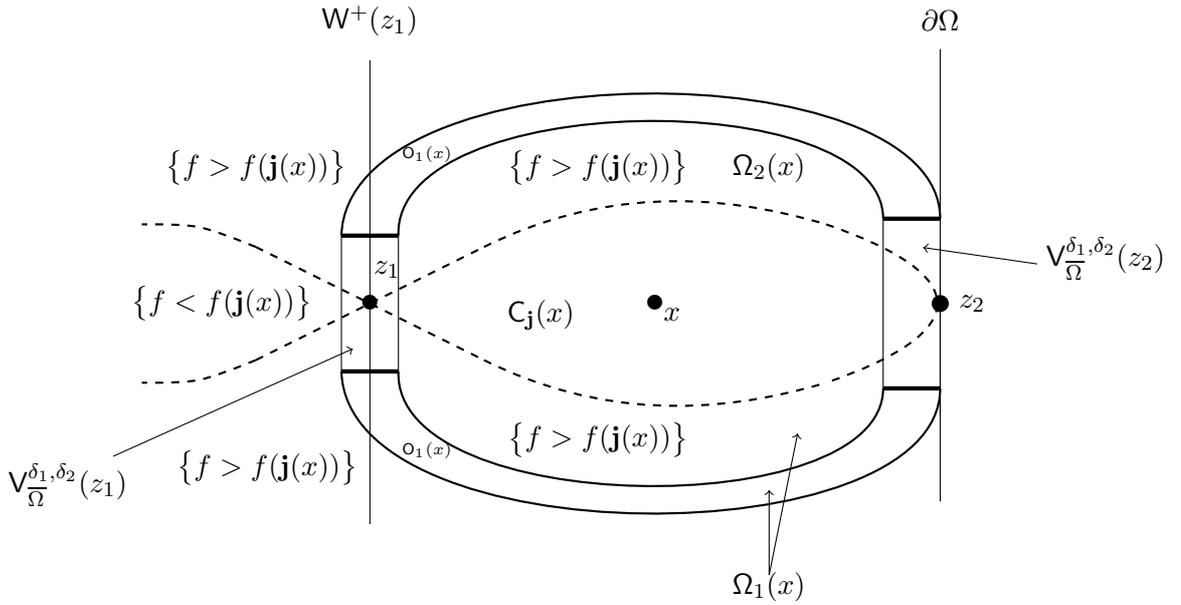
\begin{figure}[h!]
\begin{center}
\begin{tikzpicture}[scale=0.75]
\tikzstyle{vertex}=[draw,circle,fill=black,minimum size=6pt,inner sep=0pt]
\draw (0,-3.9)--(0,4.3);
\draw[ultra thick]  (0.5,1.2)--(-0.5,1.2);
\draw[ultra thick]  (0.5,-1.2)--(-0.5,-1.2);
\draw (-0.5,1.2)--(-0.5,-1.2);
\draw  (0.5,1.2)--(0.5,-1.2);
 \draw (0,5) node[]{$\ft W^+(z_1)$};
    \draw (0,0)   node {\Large{$\bullet$}} ;
  \draw (-2.6,0) node[]{$ \big\{f<f(\mbf j(x)) \big\}$};
   \draw (5,0)   node {\Large{$\bullet$}} ;
  \draw (5.3,-0.2) node[]{$ x$};
    \draw (3,-0.2) node[]{$\ft C_{\mbf j}(x)$};
    \draw (4,2.4) node[]{$ \big\{f>f(\mbf j(x)) \big\}$};
    
    \draw (7,-5) node[]{$\ft \Omega_{1}(x)$};
    \draw [->] (7,-4.8) to (7,-3.3) ;
    \draw [->] (7,-4.8) to (7.5,-2.3) ;

     \draw (7,2.4) node[]{$\ft \Omega_{2}(x)$};
    
      \draw (4,-2.4) node[]{$\big \{f>f(\mbf j(x)) \big\}$};
       \draw (-2,2.4) node[]{$ \big\{f>f(\mbf j(x)) \big\}$};
      \draw (-1.8,-2.9) node[]{$\big \{f>f(\mbf j(x)) \big\}$};
  \draw[thick, dashed]  (-2,1) to (2,-1) ;
   \draw[thick, dashed]  (-2,-1) to (2,1) ;
   \draw (10,-3.5)--(10,4.5);
\draw[ultra thick]  (10,1.5)--(9,1.5);
\draw[ultra thick]  (10,-1.5)--(9,-1.5);
\draw (9,1.5)--(9,-1.5);
\draw  (10,1)--(10,-2);
 \draw (10,5) node[]{$\pa \Omega$};
  \draw (10,0)  node[vertex,label=east: {$z_2$}](v){};
  \draw[thick, dashed]  (8,1.4) ..controls (10.6,0.5) and   (10.6,-0.5)  .. (8,-1.4) ;
     \draw[thick, dashed]  (-2,1) ..controls (-3,1.4)  .. (-4,1.4) ;
       \draw[thick, dashed]  (-2,-1) ..controls (-3,-1.4)  .. (-4,-1.4) ;        
   \draw[thick, dashed]  (2,1) ..controls (4,2) and   (6,2)  .. (8,1.4) ;
\draw[thick, dashed]  (2,-1) ..controls (4,-2) and   (6,-2)  .. (8,-1.4) ;
\draw[thick]  (0.5,1.2) ..controls (0.4,3.8) and   (9,3.9)  .. (9,1.5);
\draw[thick]  (-0.5,1.2) ..controls (-0.4,4.5) and   (10,4.5)  .. (10,1.5);
\draw[thick]  (0.5,-1.2) ..controls (0.4,-3.8) and   (9,-3.9)  .. (9,-1.5);
\draw[thick]  (-0.5,-1.2) ..controls (-0.4,-4.5) and   (10,-4.5)  .. (10,-1.5);
\draw (1,-2.6) node[]{\tiny{$\ft O_1(x)$}};
\draw (1,2.64) node[]{\tiny{$\ft O_1(x)$}};
\draw (12.9,0.8) node[]{$\ft V^{\delta_1,\delta_2}_{\overline \Omega}(z_2)$};
\draw [->] (11.7,0.7) to (9.7,0.99) ;
\draw (-5.3,-3.2) node[]{$\ft V^{\delta_1,\delta_2}_{\overline \Omega}(z_1) $};
\draw [->] (-4.5,-2.7) to (-0.3,-0.8) ;
 \draw (0.3,0.6) node[]{${\small z_1}$};
\end{tikzpicture}
\caption{Schematic representation of $\ft \Omega_{2}(x)$, $\ft \Omega_{1}(x)$, and $\ft O_1(x)$ (see Definition~\ref{de.omega1}). On the figure, $\mbf j(x)=\{z_1,z_2\}$ with $z_1\in   \Omega$ and $z_2\in \pa \Omega$ ($\vert \nabla f(z_2)\vert=0$). }
 \label{fig:S}
 \end{center}
\end{figure}
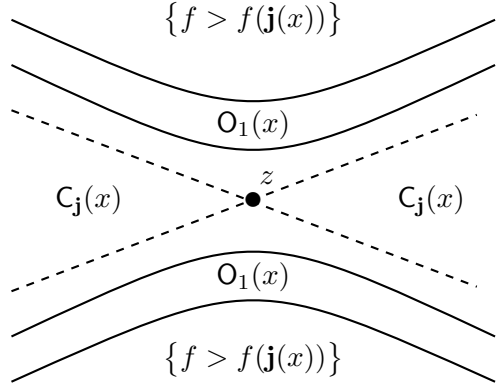
\begin{figure}[h!]
\begin{center}
\begin{tikzpicture}[scale=0.6]
\tikzstyle{vertex}=[draw,circle,fill=black,minimum size=6pt,inner sep=0pt]
\draw[thick, dashed]  (-5.3,2) to (4.9,-1.9) ;
   \draw[thick, dashed]  (-5.3,-2) to (4.9,1.9) ;
\draw (0,0) node[]{\Large{$\bullet$}}; 
  \draw (-3.6,0) node[]{$ \ft C_{\mbf j}(x)$};
  \draw (3.9,0) node[]{$\ft C_{\mbf j}(x)$};
    \draw (0,4) node[]{$ \big\{f>f(\mbf j(x))\big\}$};
      \draw (0,-3.5) node[]{$\big \{f>f(\mbf j(x))\big\}$};
       \draw[thick]  (-5.3,3) ..controls  (0,0.5)  .. (5.3,3) ;
       \draw[thick]  (-5.3,4) ..controls  (0,1.6)  .. (5.3,4) ;
       \draw[thick]  (-5.3,-3) ..controls  (0,-0.5)  .. (5.3,-3) ;
              \draw[thick]  (-5.3,-4) ..controls  (0,-1.6)  .. (5.3,-4) ;
              \draw (0,1.69) node[]{$ {\ft O_1(x)}$};
              \draw (0.3,0.5) node[]{$z$};
              \draw (0,-1.69) node[]{$ {\ft O_1(x)}$};
\end{tikzpicture}
\caption{Schematic representation of $\ft O_1(x)$  (see~\eqref{eq.SS}) in a neighborhood of  a  non separating saddle point $z$ of $f$ on $\pa \ft C_{{\bf j}}(x)$. }
 \label{fig:S-b}
 \end{center}
\end{figure}

\medskip

\noindent
Let us now define, for each $x\in \ft U_0$,  the quasi-mode     $\psi_x: \overline \Omega \to \mathbb R^+$ of $\Delta_{f,h}^D$ as follows.

\begin{definition}\label{de.QM}
Let $f:\overline \Omega\to \mathbb R$  be a $C^\infty$ Morse function which satisfies \autoref{H1} and~\autoref{H2}.  For every 
$x\in \ft U_0$,  one defines 
$$\psi_x:= \frac{\phi_x \, e^{-\frac fh}}{ Z_x } \ \  \text{and}\ \  Z_x:=\big \Vert \phi_x \, e^{-\frac fh}   \big  \Vert_{L^2(\Omega)}, $$
where $\phi_x$ is the function introduced in  Definition~\ref{de.psix}.
\end{definition}
\noindent
 By construction of $\phi_x$ in Definition~\ref{de.psix}, $\psi_x \in C^\infty(\overline \Omega,\mathbb R^{+})$ and $\psi_x=0$ on $\pa \Omega$ (see indeed \eqref{eq.psix=10} together with the fact that $\ft \Omega_1(x)\subset \Omega$, see Definition~\ref{de.omega1}). In particular:
 \begin{equation} 
\label{eq.nablapsix-domaine}
\psi_x\in D(\Delta_{f,h}^D)=H^2(\Omega)\cap H^1_0(\Omega).
\end{equation}



\section{Asymptotic equivalents of the small eigenvalues of $\Delta_{f,h}^{D}$}
\label{sec.eigenvalues}

%
%

\subsection{First quasi-modal estimates}
\label{sec.quasi-modal estimates}

\medskip

\noindent
Let us start with the following result which gives asymptotic estimates  on the $L^2$-norms of  $d_{f,h}(\psi_x)$ and  of $\Delta_{f,h}(\psi_x)$ around  the points $z\in \mbf j(x)$ in the limit $h\to 0$. 

\begin{proposition}\label{pr.qm1}
Let $f:\overline \Omega\to \mathbb R$  be a $C^\infty$ Morse function which satisfies \autoref{H1} and~\autoref{H2}. Let $x\in \ft U_0 $, $\psi_x$ be as introduced in Definition~\ref{de.QM}, and $z\in  \mbf j(x)$. 
\begin{enumerate}
\item Let us assume that  $z\in \pa \Omega$.  
\begin{enumerate}
\item  When $\vert \nabla f(z)\vert \neq 0$  (recall that in this case  $z$ is a non degenerate local minimum of $f|_{\pa \Omega}$ and $\pa_{{\ft n}_{\Omega}}f(z)>0$, see  item (a) in Corollary~\ref{co.hypoA} and~\eqref{eq.H2-non-dege}), it holds in the limit $h\to 0$:
\begin{align} 
   \label{eq.cxz1}
\left\{
    \begin{array}{ll}
    \displaystyle \int_{\ft V^{\delta_1,\delta_2}_{\,\overline \Omega}(z) } \big  \vert d_{f,h} \, \psi_x  \big \vert ^2 = \ft c_{x,z} \, \sqrt h \,  e^{-\frac 2h (f(\mbf j(x))-f(x))}     \big(1+O(h)\big),\\
    \text{where }  \ft c_{x,z} :=  \displaystyle{\frac{2\,\pa_{{\ft n}_{\Omega}}f(z)}{\sqrt \pi} \frac{  \big(  \det\Hess f|_{\pa \Omega}(z)\big)^{-\frac 12}   }{ \sum \limits_{q\in \argmin_{ \ft C_{\mbf j}(x)}f }     \big(  \det\Hess f(q)\big)^{-\frac 12}      }}.
    \end{array}
    \right. 
\end{align}
Furthermore, one has when $h\to 0$: 
\begin{align*}
     \int_{\ft V^{\delta_1,\delta_2}_{\,\overline \Omega}(z) } \big  \vert \Delta_{f,h} \, \psi_x  \big \vert ^2& = O(h^2 ) \, \int_{\ft V^{\delta_1,\delta_2}_{\,\overline \Omega}(z) } \big  \vert d_{f,h} \, \psi_x  \big \vert ^2.
\end{align*}

\item  When $\vert \nabla f(z)\vert =0$ (recall that in this case  $z$ is a saddle point of $f$, see item (b) in Corollary~\ref{co.hypoA}), it holds in the limit $h\to 0$:
\begin{align}
  \label{eq.cxz2}
\left\{
    \begin{array}{ll}
\displaystyle \int_{\ft V^{\delta_1,\delta_2}_{\,\overline \Omega}(z) } \big  \vert d_{f,h}  \,  \psi_x \big \vert ^2   =    \ft c_{x,z} \, h \,  e^{-\frac 2h (f(\mbf j(x))-f(x))}    \big(1+ O(\sqrt h)\big),\\
 \text{where } \ft c_{x,z}:=\displaystyle{ \frac{2\,\vert \mu_d\vert}{\pi}  \frac{ \big \vert   \det\Hess f(z) \big \vert   ^{-\frac 12}  }{   \sum \limits_{q\in \argmin_{ \ft C_{\mbf j}(x)}f }     \big(  \det\Hess f(q)\big)^{-\frac 12}      }   },
   \end{array}
    \right. 
\end{align}
where we recall that $\mu_d<0$ is the negative eigenvalue of $\Hess  f (z)$. 
Moreover, when $h\to 0$, one has:
\begin{align*}
 \int_{\ft V^{\delta_1,\delta_2}_{\,\overline \Omega}(z) } \big  \vert \Delta_{f,h} \, \psi_x  \big \vert ^2 = O(h^2) \int_{\ft V^{\delta_1,\delta_2}_{\,\overline \Omega}(z) } \big  \vert d_{f,h} \, \psi_x  \big \vert ^2 . 
\end{align*}
\end{enumerate}

\item  Let us assume that  $z\in  \Omega$ (recall that in this case  $z$ is a saddle point of $f$ in $\Omega$). Then, it holds in the limit $h\to 0$:
\begin{align}
  \label{eq.cxz3}
\left\{
    \begin{array}{ll}
\displaystyle\int_{\ft V^{\delta_1,\delta_2}_{\,\overline \Omega}(z)}  \big \vert d_{f,h}   \,  \psi_x \big \vert ^2  = \ft c_{x,z} \, h \,e^{-\frac 2h (f(\mbf j(x))-f(x))}   \big(1+O(h)\big),\\
 \text{where } \ft c_{x,z}:=\displaystyle{ \frac{\vert \mu_d\vert}{\pi}  \frac{ \big \vert  \det\Hess f(z) \big \vert   ^{-\frac 12}  }{     \sum \limits_{q\in \argmin_{ \ft C_{\mbf j}(x)}f }     \big(    \det\Hess f(q)\big)^{-\frac 12}      } },
   \end{array}
    \right. 
\end{align}
where we recall that $\mu_d<0$ is the negative eigenvalue of $\Hess  f (z)$. 
 Finally, when $h\to 0$, one has:
\begin{align*} \int_{\ft V^{\delta_1,\delta_2}_{\,\overline \Omega}(z) } \big  \vert \Delta_{f,h} \, \psi_x  \big \vert ^2 = O(h^2)\,  \int_{\ft V^{\delta_1,\delta_2}_{\,\overline \Omega}(z) } \big  \vert d_{f,h} \, \psi_x  \big \vert ^2 . 
\end{align*}

\end{enumerate}
\end{proposition}
\begin{remark}\label{re.lap}
The remainder term $O(\sqrt h)$ in~\eqref{eq.cxz2} follows from  the Laplace method  applied to $\int_{\mathbb R^d_-\cap B(0,r)} \varphi_2 \, e^{-\frac 1h \varphi_1} $  when    $\vert \nabla \varphi_1 (0)\vert =0$, $\Hess \varphi_1(0)>0$, and  $0$ is the unique minimum of $\varphi_1$  on $\overline{B(0,r)}$, see~\eqref{eq.chiUU3} and the lines below (when $d=1$,  this is also known   as  Watson's lemma). On the other hand,  the~$O(h)$ in~\eqref{eq.cxz3} arises from the standard Laplace method, i.e. when considering $\int_{B(0,r)} \varphi_2 \, e^{-\frac 1h \varphi_1} $.
In particular,    these remainder terms  are optimal. 

\end{remark}
\begin{proof}
Let $x\in \ft U_0$. Then, according to  Definitions~\ref{de.QM} and~\ref{de.psix}, one has
$$Z_x^2=\int_\Omega \phi_x^2  \,e^{-\frac 2h}=\int_{\ft \Omega_1(x)} \phi_x^2  \,e^{-\frac 2hf} = \int_{\ft \Omega_{2}(x)} \phi_x^2  \,e^{-\frac 2hf}+\int_{\ft \Omega_{1}(x)\setminus\ft \Omega_2(x)} \phi_x^2  \,e^{-\frac 2h f}.$$
Let us recall that by construction $0\le  \phi_x  \le 1$ on $\Omega$. Moreover,  from the first statements in~\eqref{eq.inclur3} and \eqref{eq.inclur4} together with  \eqref{stripS} and~\eqref{eq.SS}, there exists $c>0$ such that $f\ge f(x)+c$ on $\overline{\ft \Omega_1(x)}\setminus\ft \Omega_2(x)$. Thus,  it  holds, for some $C>0$ independent of $h$:
$$\int_{\ft \Omega_{1}(x)\setminus\ft \Omega_2(x)} \phi_x^2  \,e^{-\frac 2h f}\le  Ce^{-\frac 2h( f(x)+c)}.$$
In addition, since  $\phi_x = 1$ on $\ft \Omega_{2}(x)$ (see~\eqref{eq.psix=10}) and $$\argmin \limits_{ \overline{\ft \Omega_{1}(x)}}f=\argmin\limits_{ \overline{\ft \Omega_{2}(x)}}f=\argmin\limits_{ \ft C_{\mbf j}(x)}f  \ \, \text{  (see item 2 in Definition~\ref{de.omega1})}$$  consists in a finite number of non degenerate local minima  $q$ of $f$ in $\Omega$ such that $f(q)=f(x)$ (since by construction  of $\ft C_{\mbf j}$, $x\in \argmin_{\ft C_{\mbf j}(x)}f$), one has when $h\to 0$, using the Laplace method,
$$\int_{ \ft \Omega_2(x)} \phi_x^2  \,e^{-\frac 2h f} =\sum_{q\in \argmin\limits_{ \ft C_{\mbf j}(x)}f }  \frac{(\pi\, h)^{\frac d2}}{ \sqrt{   \det \Hess f(q)}}\, e^{-\frac 2h f(x)} \big(1+O(h)\big).$$
Therefore, when $h\to 0$, 
\begin{equation}\label{Zx}
Z_x=(\pi\, h)^{\frac d4} \Big(\sum_{q\in \argmin_{ \ft C_{\mbf j}(x)}f }  \big(  \det\Hess  f(q)\big )^{-\frac 12}      \Big)^{\frac12}   \ e^{-\frac 1h f(x)} \big(1+O(h)\big).
\end{equation}
Let now $z$ belong to $\mbf j (x)$. The rest of
the proof of Proposition~\ref{pr.qm1} is  divided into two steps, whether $z\in \pa\Omega$ or $z\in \Omega$. \medskip

\noindent
\textbf{Step 1.a) The case when $z\in \pa \Omega$ and $\vert \nabla f(z)\vert \neq 0$.}
\medskip

\noindent
In this case, from Definition~\ref{de.QM}, one has  
\begin{equation}\label{eq.chiYY0}
 \int_{\ft V^{\delta_1,\delta_2}_{\,\overline \Omega}(z) } \big  \vert d_{f,h} \, \psi_x  \big \vert ^2=h^2\frac{  \int_{\ft V^{\delta_1,\delta_2}_{\,\overline \Omega}(z) }  \big \vert  d \phi_x\big \vert ^2\, e^{-\frac 2h f}   }{  Z^{2}_x}.
 \end{equation}
Moreover, according to~\eqref{eq.qm-local1x} and to \eqref{eq.qm-local1},  it holds:
\begin{equation}\label{eq.chiYY1}
\int_{\ft V^{\delta_1,\delta_2}_{\,\overline \Omega}(z) }  \big \vert  d \phi_x\big \vert ^2\, e^{-\frac 2h f}  =\frac{ \int_{\vert v'\vert \le \delta_2}\int_{v_{d}=-2\delta_1}^0 |dv_{d}|^{2}  \, \chi^2(v_{d})  \, e^{-\frac 2h (f -2\mu v_d)}\ d_{g}v   }{ \Big(\int_{-2\delta_1}^0\chi(t)e^{\frac 2h\mu\, t} dt\Big)^2  }\,,
 \end{equation}
   where
 we recall  that $\mu=\pa_{{\ft n}_{\Omega}}f(z)>0$,
 and $d_{g}v=\sqrt{\det g}\,dv$ denotes the Riemannian volume form.
A straightforward computation  (see~\eqref{eq.chi-cut})  implies  that there exists $c>0$ such that in
 the limit $h\to 0$,
\begin{equation}\label{eq.chiYY2}
N_z:= \int_{-2\delta_1}^0\chi(t)e^{\frac 2h\mu\, t} dt=\frac{h}{2\mu} \big(1+O(e^{-\frac ch})\big).
 \end{equation}
Moreover, 
from the Laplace method     together with,~\eqref{eq.chi-cut},  \eqref{eq.cv-pa-omega-nablafnon03}, and \eqref{eq.cv-pa-omega-nablafnon04}, one has when $h\to 0$:
\begin{equation}\label{eq.chiYY3}
\int_{\vert v'\vert \le \delta_2}\int_{-2\delta_1}^0 |dv_{d}|^{2}  \, \chi^2(v_{d})  \, e^{-\frac 2h (f -2\mu v_d)}\,d_{g}v = \frac{h}{2\mu} \frac{ (\pi h)^{\frac{d-1}{2} }   e^{-\frac 2h f(0)}}{\big(\det \Hess f|_{\{v_d=0\}}(0)\big)^{\frac12}}  \big(1+O(h)\big),
  \end{equation}
  where we recall that with our notation, $f(0)=f(z)=f(\mbf j(x))$ since $z\in \mbf j(x)$ (see item 4 in Section~\ref{sec.j}). 
The relations \eqref{eq.chiYY0}--\eqref{eq.chiYY3} and~\eqref{Zx} lead to the first statement of item 1.(a) in Proposition~\ref{pr.qm1}. Let us now prove the second  statement of  item 1.(a). 
Since $
\Delta_{f,h}=2h e^{-\frac fh} \big (\,  \frac h2 \Delta_H+\nabla f\cdot \nabla \,  \big)e^{\frac fh}$, 
 one has 
\begin{equation}\label{equiv2}
\Delta_{f,h} \psi_x\ =\ \frac{2h e^{-\frac fh} }{Z_{x}}\Big (\,  \frac h2 \Delta_H+\nabla f\cdot \nabla \,  \Big)\phi_x
\ =\ \frac{2h e^{-\frac fh} }{Z_{x}} \Big (\,  \frac h2 d^{*}d \phi_x+df(\nabla\phi_x)\,  \Big)\,.
\end{equation}
 Thus, according to \eqref{eq.qm-local1} and to \eqref {eq.cv-pa-omega-nablafnon04}, \eqref{eq.cv-pa-omega-nablafnon03},   it holds on $\ft V^{\delta_1,\delta_2}_{\,\overline \Omega}(z)$, 
\begin{align}
\nonumber
\Delta_{f,h} \psi_x&=  \frac{2he^{-\frac f h}}{Z_x}\big(\, \frac{h}{2}
d^{*}d\varphi_z \, + \, df(\nabla\varphi_z) \,\big)
\\
\nonumber
&=
 \frac{2he^{-\frac f h}}{Z_xN_z}\Big(\, \frac{h}{2}
d^{*}\big(-\chi(v_{d})e^{\frac2h \mu v_{d}}dv_{d}\big) \, - \,\mu\,dv_{d} \big(\chi(v_{d})e^{\frac2h \mu v_{d}}\nabla v_{d}\big) + O(|v|^{2})\,e^{\frac2h \mu v_{d}}\Big)\\
\nonumber
&=
 \frac{2he^{-\frac f h}\,e^{\frac2h \mu v_{d}}}{Z_xN_z}\Big(\, O(h) +\frac h2 \,\chi(v_{d})\,dv_{d}\big(\frac2h \,\mu\,\nabla v_{d}\big)
 \, - \,\mu\,dv_{d} \big(\chi(v_{d})\nabla v_{d}\big)+ O(|v|^{2}) \,\Big)
\\
\label{eq.g-delta}
&=   \frac{  h\, e^{-\frac 1h(f -2\mu v_d)}}{Z_xN_z} \big( O(h)+ O(|v|^{2})\big)\,,  
\end{align} 
where $N_z$ is defined by~\eqref{eq.chiYY2}. 
It then follows from \eqref{Zx} that for every $h$ small enough, it holds
\begin{align*}
\int_{\ft V^{\delta_1,\delta_2}_{\,\overline \Omega}(z) } \big  \vert \Delta_{f,h} \, \psi_x  \big \vert ^2 
\ =\   O(h^2 ) \sqrt h \, e^{-\frac 2h (f(z)-f(x))}
\ =\ O(h^2)\int_{\ft V^{\delta_1,\delta_2}_{\,\overline \Omega}(z) } \big  \vert d_{f,h} \, \psi_x  \big \vert ^2, 
\end{align*}
which concludes the proof of   item 1.(a) in Proposition~\ref{pr.qm1}. 

\medskip

\noindent
\textbf{Step 1.b) The case when  $z\in \pa \Omega$ and  $\vert \nabla f(z)\vert=0$.}\medskip

\noindent
From Definition~\ref{de.QM}, it holds  
\begin{equation}\label{eq.chiUU0}
 \int_{\ft V^{\delta_1,\delta_2}_{\,\overline \Omega}(z) } \big  \vert d_{f,h} \, \psi_x  \big \vert ^2=h^2\frac{  \int_{\ft V^{\delta_1,\delta_2}_{\,\overline \Omega}(z) }  \big \vert  d \phi_x\big \vert ^2\, e^{-\frac 2h f}   }{  Z_x}\,,
\end{equation}
where, according to~\eqref{eq.qm-local1x} and~\eqref{eq.qm-local12}, 
\begin{equation}\label{eq.chiUU1}
\int_{\ft V^{\delta_1,\delta_2}_{\,\overline \Omega}(z) }  \big \vert  d \phi_x\big \vert ^2\, e^{-\frac 2h f}  =\frac{ \int_{\vert v'\vert \le \delta_2}\int_{v_{d}=-2\delta_1}^0 |dv_{d}|^{2}  \, \chi^2(v_{d})  \, e^{-\frac 2h (f +\vert \mu_d\vert  v_d^2)}  d_{g}v   }{ \Big(\int_{-2\delta_1}^0\chi(t)\, e^{-\frac 1h\vert \mu_d\vert  \, t^2}  dt\Big)^2  },
\end{equation}
where we recall that $\mu_d$ is the negative eigenvalue of $\Hess f(z)$.
 A  straightforward computation (see~\eqref{eq.chi-cut}) implies that there exists $c>0$ such that in
 the limit $h\to 0$,
\begin{align}\label{eq.chiUU2}
N_z:=\int_{-2\delta_1}^0\chi(t)\, e^{-\frac 1h\vert \mu_d\vert  \, t^2}  dt=\frac{\sqrt{\pi h} }{2\sqrt{\vert \mu_d\vert}} \big(1+O(e^{-\frac ch})\big).
\end{align}
Furthermore, from,~\eqref{eq.chi-cut}, \eqref{eq.cv-pa-omega-nablaf=03}, \eqref{eq.cv-pa-omega-nablaf=04}, and \eqref{eq.cv-pa-omega-nablaf=07} together with the Laplace method, one has in the limit $h\to 0$: 
\begin{equation}\label{eq.chiUU3}
 \int_{\vert u'\vert \le \delta_2}\int_{-2\delta_1}^0 |dv_{d}|^{2} \,    \chi^2(v_{d}) \, e^{-\frac 2h (f +\vert \mu_d\vert  v_d^2)} d_{g}v  = \frac{ (\pi h)^{\frac{d}{2} }  e^{-\frac 2h f(0)}}{\sqrt{ \mu_1 \cdots \mu_{d-1} \vert \mu_d\vert}}
 \big(\frac12+ O ( \sqrt h) \big),
 \end{equation}
 where   $\mu_1,\ldots,\mu_{d-1}$ are the positive eigenvalues of $\Hess f(z)$. Let us point out that the integral in~\eqref{eq.chiUU3} has the form~$\int_{\mathbb R^d_-\cap B(0,r)} \varphi_2(v) \, e^{-\frac 1h \varphi_1(v)}dv$. Hence, the terms of the type $\int_{\mathbb R^d_-\cap B(0,r)} v^\alpha \, e^{-\frac 1h  \, ^t v \,\Hess \varphi_1(0) \, v} dv$ which appear when performing the Laplace method do not cancel (up to an exponentially small error term)  when $\vert \alpha\vert$ is odd, contrary to the  terms   
 $\int_{B(0,r)} v^\alpha \, e^{-\frac 1h \, ^t  v\, \Hess \varphi_1(0) \, v} dv$  appearing in the standard Laplace method
 (by a parity argument) as used to get~\eqref{eq.chiVV3}. This justifies the optimality of the $O(\sqrt h)$ in~\eqref{eq.chiUU3} (see~Remark~\ref{re.lap} above). \\
Equations~\eqref{eq.chiUU0}--\eqref{eq.chiUU3} and~\eqref{Zx} lead to the first statement in  item 1.(b) of Proposition~\ref{pr.qm1}. 
Let us now prove the second statement in item 1.(b).
Doing the same computations as to obtain 
\eqref{eq.g-delta}, one deduces from~\eqref{eq.qm-local12},~\eqref{eq.cv-pa-omega-nablaf=03}, 
and~\eqref{eq.cv-pa-omega-nablaf=04} that on $\ft V^{\delta_1,\delta_2}_{\,\overline \Omega}(z)$, 
\begin{align*}
\Delta_{f,h} \psi_x&=\frac{2he^{-\frac f h}}{Z_xN_z}\Big(\, \frac{h}{2}
d^{*}\big(-\chi(v_{d})e^{-\frac1h |\mu_{d}| v^{2}_{d}}dv_{d}\big) \, + \,|\mu_{d}|\,v_{d}\,dv_{d} \big(\chi(v_{d})e^{-\frac1h |\mu_{d}| v^{2}_{d}}\nabla v_{d}\big) + O(|v|^{2})\,  e^{-\frac1h |\mu_{d}| v^{2}_{d}}\,\Big)\\
&=  \frac{2h\, e^{-\frac 1h(f + \vert \mu_d\vert  v_d^2)}}{Z_xN_z}
\Big(\, O(h) -\frac h2 \,\chi(v_{d})\,dv_{d}\big(\frac1h \,|\mu_{d}|\,\nabla v^{2}_{d})
 \, + \,|\mu_{d}|\,v_{d}\,dv_{d} \big(\chi(v_{d})\nabla v_{d}\big)+ O(|v|^{2}) \,\Big)\\
&=  \frac{2h\, e^{-\frac 1h(f + \vert \mu_d\vert  v_d^2)}}{Z_xN_z}\Big(
O(h) + O(|v|^{2})
  \Big)\,,
\end{align*}
where $N_z$ is defined by~\eqref{eq.chiUU2}.
It then follows from
\eqref{eq.chiUU2}, \eqref{Zx},  and~\eqref{eq.cv-pa-omega-nablaf=07}
  that in the limit $h\to0$,
%
\begin{align*}
\int_{\ft V^{\delta_1,\delta_2}_{\,\overline \Omega}(z) } \big  \vert \Delta_{f,h} \, \psi_x  \big \vert ^2 
&= O(h^3)   \, e^{-\frac 2h (f(z)-f(x))}=O(h^2)\int_{\ft V^{\delta_1,\delta_2}_{\,\overline \Omega}(z) } \big  \vert d_{f,h} \, \psi_x  \big \vert ^2. 
\end{align*}
 This proves the second  statement of item 1.(b) in Proposition~\ref{pr.qm1}. 

 \medskip

\noindent
\textbf{Step 2.  The case when $z\in  \Omega$.}\medskip

\noindent   
According to Definition~\ref{de.QM}, one has
\begin{equation}\label{eq.chiVV0}
 \int_{\ft V^{\delta_1,\delta_2}_{\,\overline \Omega}(z) } \big  \vert d_{f,h} \, \psi_x  \big \vert ^2=h^2\frac{  \int_{\ft V^{\delta_1,\delta_2}_{\,\overline \Omega}(z) }  \big \vert  d \phi_x\big \vert ^2\, e^{-\frac 2h f}   }{  Z_x},
\end{equation}
where,
from~\eqref{eq.qm-local1x} and~\eqref{eq.qm-local2}, one has:
\begin{equation}\label{eq.chiVV1}
\int_{\ft V^{\delta_1,\delta_2}_{\,\overline \Omega}(z) }  \big \vert  d \phi_x\big \vert ^2\, e^{-\frac 2h f}  =\frac{ \int_{\vert v'\vert \le \delta_2}\int_{v_{d}=-2\delta_1}^{2\delta_1} |dv_{d}|^{2} \,    \chi^2(v_{d}) \, e^{-\frac 2h (f +\vert \mu_d\vert  v_d^2)}  d_gv    }{ \Big(\int_{-2\delta_1}^{2\delta_1}\chi(t)\, e^{-\frac 1h\vert \mu_d\vert  \, t^2}  dt\Big)^2  },
\end{equation}
 where  $\mu_d$ is the negative eigenvalue of $\Hess f(z)$.
A  straightforward computation   (see~\eqref{eq.chi-cut})  implies the existence of $c>0$ such that
in  the limit $h\to 0$, 
\begin{align}\label{eq.chiVV2}
N_z:=\int_{-2\delta_1}^{2\delta_1}\chi(t)\, e^{-\frac 1h\vert \mu_d\vert  \, t^2}  dt=\frac{\sqrt{\pi h} }{\sqrt{\vert \mu_d\vert}} \big(1+O(e^{-\frac ch})\big).
\end{align}
Moreover, from,~\eqref{eq.chi-cut},~\eqref{eq.cv-pa-omega-nablaf=05'},~\eqref{eq.cv-omega5}, \eqref{eq.cv-pa-omega-nablaf=07oubli}  and  the Laplace method, one has in the limit $h\to 0$:
\begin{equation}\label{eq.chiVV3}
 \int_{\vert v'\vert \le \delta_2}\int_{v_{d}=-2\delta_1}^{2\delta_1} |dv_{d}|^{2} \,    \chi^2(v_{d}) \, e^{-\frac 2h (f +\vert \mu_d\vert  v_d^2)}d_{g}v  =   \frac{ (\pi h)^{\frac{d}{2} }  e^{-\frac 2h f(0)}}{\sqrt{ \mu_1\cdots\mu_{d-1} \vert \mu_d\vert}} \big(1+O(h)\big),
 \end{equation}
 where   $\mu_1,\ldots,\mu_{d-1}$ are the positive eigenvalues of $\Hess f(z)$. 
The relations \eqref{eq.chiVV0}--\eqref{eq.chiVV3} and~\eqref{Zx} imply the first statement of item 2 in Proposition~\ref{pr.qm1}. 
Let us lastly prove the second statement in item 2.
From~\eqref{eq.qm-local2},~\eqref{eq.cv-pa-omega-nablaf=05'}, and~\eqref{eq.cv-omega5},  the same computations as those used to obtain~\eqref{eq.g-delta} imply that on $\ft V^{\delta_1,\delta_2}_{\,\overline \Omega}(z)$, 
\begin{align*}
\Delta_{f,h} \psi_x&=  \frac{2h\, \chi\, e^{-\frac 1h(f + \vert \mu_d\vert  v_d^2)}}{Z_xN_z}\Big(   O(h) + O(|v|^{2})  \Big)\end{align*}
and the relations \eqref{eq.cv-pa-omega-nablaf=07oubli}, \eqref{eq.chiVV2}, and \eqref{Zx}
then lead to
%
%
%
\begin{align*}
\int_{\ft V^{\delta_1,\delta_2}_{\,\overline \Omega}(z) } \big  \vert \Delta_{f,h} \, \psi_x  \big \vert ^2 &=  O(h^3)   \, e^{-\frac 2h (f(z)-f(x))}=O(h^2)\int_{\ft V^{\delta_1,\delta_2}_{\,\overline \Omega}(z) } \big  \vert d_{f,h} \, \psi_x  \big \vert ^2. 
\end{align*}
  \noindent
 This concludes the proof of  Proposition~\ref{pr.qm1}. 
\end{proof}

\noindent
For every $x\in \ft U_0$, one defines the following constants:
\begin{equation}
\label{eq.cstCx}
  \ft K_{1,x}:= \sum_{\substack{z\in \mbf j(x)\\  \vert \nabla f(z)\vert \neq 0}} \ft c_{x,z}\  \text{ and }  \ \ft K_{2,x}:= \sum_{\substack{z\in \mbf j(x)\\  \vert \nabla f(z)\vert = 0}} \ft c_{x,z},
\end{equation}
 where the constants $\ft c_{x,z}$ are defined in~\eqref{eq.cxz1},~\eqref{eq.cxz2}, and~\eqref{eq.cxz3},
 with the convention   $\sum_{\emptyset}=0$.    Let us recall that $\{z\in \mbf j(x), \vert \nabla f(z)\vert \neq 0\}\subset \pa \Omega$. 
Finally, for $y\neq x\in \ft U_0$, one defines:
\begin{equation}
\label{eq.cstCxy}
  \ft K_{x,y}:=   \sum_{z\in \mbf j(x) \cap \mbf j(y)} \sqrt{\ft c_{x,z}}\sqrt{ c_{y,z}}, \ \  \text{ see~\eqref{eq.cxz1}--\eqref{eq.cxz3}}. 
\end{equation}
 Let us mention that  since   for all $x\in \ft U_0$, one has $\mbf j(x)\neq \emptyset$, it holds $(\ft K_{1,x} ,\ft K_{2,x})\neq (0,0)$.\\
Proposition~\ref{pr.qm1} has the following consequence.

\begin{proposition}\label{pr.qm2}
Let $f:\overline \Omega\to \mathbb R$  be a $C^\infty$ Morse function which satisfies \autoref{H1} and~\autoref{H2}.  Let $x\in \ft U_0 $ and  $\psi_x$   be as introduced in Definition~\ref{de.QM}.  
\begin{enumerate}
\item In the limit $h\to 0$, one has:
$$
 \big \Vert   d_{f,h}  \psi_x  \big \Vert_{\Lambda^1 L^2(\Omega)}^2 =\Big(\sqrt h \,   \ft K_{1,x}(1+{O  (  h\,   )})+ h\, \ft K_{2,x}(1+{O  ( \sqrt h\,  )})\Big)e^{-\frac 2h (f(\mbf j(x))-f(x))}  ,
$$
where the constants $\ft K_{1,x}$ and  $\ft K_{2,x}$  are defined in~\eqref{eq.cstCx}. When $\mbf j(x)\cap \pa \Omega$ does not contain any critical point of $f$,  the term $O(\sqrt h)$  is actually of order $O(h)$.  Moreover, it holds in the limit $h\to 0$:
$$
\big \Vert   \Delta_{f,h}  \psi_x  \big \Vert_{L^2(\Omega)}^2=O(h^2) \big \Vert   d_{f,h}  \psi_x  \big \Vert_{\Lambda^1 L^2(\Omega)}^2.
$$

\item  Let $y\in \ft U_0 $ be such that $y\neq x$. 
 Then, for each of the two possible cases described in  items 4.(i) and 4.(ii)    in  Section~\ref{sec.j}, it holds
 in the limit $h\to 0$:
 \begin{enumerate}
\item[(i)]  When  $ {\mbf j}(x)\cap {\mbf j}(y)=\emptyset $, 
$
\big \langle   d_{f,h}\psi_x,  d_{f,h}  {\psi}_y  \big \rangle_{\Lambda^1 L^2(\Omega)} =0$. 
 
\item[(ii)]   When  $ {\mbf j}(x)\cap {\mbf j}(y)\neq \emptyset $, \begin{align*}
\big \langle   d_{f,h} \psi_x,  d_{f,h}  {\psi}_y  \big \rangle_{\Lambda^1 L^2(\Omega)} &= -h\, \ft K_{x,y}\,   e^{-\frac 1h (2f(\mbf j(x))-f(x)-f(y))}   \big(1+O(h)\big),
\end{align*}
where  $\ft K_{x,y}$ is defined in~\eqref{eq.cstCxy}.  
\end{enumerate}
\end{enumerate}

\end{proposition}

\begin{proof}
Let $x\in \ft U_0 $.
\medskip

\noindent
 Let us first prove item 1 in Proposition~\ref{pr.qm2}. 
From  Definition~\ref{de.QM} and~\eqref{eq.psix=nabla},  
\begin{equation}
\label{eq.psix--}
d_{f,h} \psi_x= Z_{x}^{-1} \, h e^{-\frac fh} d\phi_x \text{ is supported in $\overline{\ft \Omega_1(x)}\setminus \ft \Omega_2(x)$}.
\end{equation}
Moreover,  from~\eqref{eq.SS},~\eqref{eq.nablapsix=0}, and~\eqref{Zx},    there exists $c>0$ such that for $h$ small enough, $h^2Z_{x}^{-2}\int_{\ft O_1(x)}  \vert d   \,  \phi_x\vert^2e^{-\frac 2h f}=   O\big (  e^{-\frac 2h (f(\mbf j(x))-f(x) +c)}\big )$.  Thus, using in addition~\eqref{eq.psix--} and~\eqref{stripS}, there exists $c>0$ such that for $h$ small enough, 
\begin{align*}
\int_{\Omega}    \vert d_{f,h}   \,  \psi_x\vert^2&=
\sum \limits_{\substack{z\in \mbf j(x)}} \,  \,  \int_{ \ft V_{\,\overline \Omega}^{\delta_1,\delta_2}(z)} \vert d_{f,h}   \,  \psi_x\vert^2 +   O\big (  e^{-\frac 2h (f(\mbf j(x))-f(x) +c)}\big ).
\end{align*}
The first statement in item 1 in Proposition~\ref{pr.qm2} is then a direct consequence of Proposition~\ref{pr.qm1}. Let us now prove the second statement in Proposition~\ref{pr.qm2}. To this end, 
note first that according to \eqref{eq.psix--},
$$
\Delta_{f,h}\psi_x= d^{*}_{f,h} d_{f,h}\psi_{x}
\text{ is supported in $\overline{\ft \Omega_1(x)}\setminus \ft \Omega_2(x)$}. 
$$ 
Thus, from~\eqref{stripS},~\eqref{eq.SS},~\eqref{eq.nablapsix=0} together with~\eqref{Zx}, it holds for some $c>0$ and every  $h$ small enough, 
 \begin{align*}
\int_{\Omega}    \vert \Delta_{f,h}   \,  \psi_x\vert^2&=
\sum \limits_{\substack{z\in \mbf j(x)}} \,  \,  \int_{ \ft V_{\,\overline \Omega}^{\delta_1,\delta_2}(z)} \vert \Delta_{f,h}   \,  \psi_x\vert^2 +O\big (  e^{-\frac 2h (f(\mbf j(x))-f(x) +c)}\big ).
\end{align*}
Together with  Proposition~\ref{pr.qm1}, this proves item 1 in Proposition~\ref{pr.qm2}.\\  
 Let us now prove item 2 in Proposition~\ref{pr.qm2}. Let us consider $y\in \ft U_0 $ such that $y\neq x$. 
 According to \eqref{eq.psix--} and~\eqref{eq.psix=nabla},
$$
    d_{f,h}   \,  \psi_x \cdot d_{f,h}   \,  {\psi}_y  = \frac{h^2  \, e^{-\frac 2h f} d     \, {\phi}_x \cdot d   \,{\phi}_y }{Z_xZ_y} \text{  is supported in $\overline{\ft \Omega_1(x)}\setminus \ft \Omega_2(x)\bigcap \overline{\ft \Omega_1(y)}\setminus \ft \Omega_2(y)$}.
$$
Thus, using item 3 in Definition~\ref{de.omega1}, it holds:
 \begin{enumerate}
\item[(i)]  When  $ {\mbf j}(x)\cap {\mbf j}(y)=\emptyset$, then, either 
$  \overline{\ft \Omega_1(x)}\,  \cap \,   \overline{\ft \Omega_1(y)}=\emptyset$ or, up to switching~$x$ and~$y$,  $\overline{\ft \Omega_{1}(y) }\subset  \ft \Omega_{2}(x)$. In any case, this implies  $\int_{\Omega}    d_{f,h}   \,  \psi_x \cdot d_{f,h}   \,  {\psi}_y =0$. 
  
\item[(ii)]  When  $\ {\mbf j}(x)\cap {\mbf j}(y)\neq \emptyset$, one has,
\begin{align}
\nonumber
\int_{\Omega}    d_{f,h}   \,  \psi_x \cdot d_{f,h}   \,  {\psi}_y   &=  \frac{h^2}{Z_xZ_y} \sum_{z\in \mbf j(y) \, \cap \,  \mbf j(x)}  \int_{\ft V^{\delta_1,\delta_2}_{\, \overline \Omega}(z)} d     \,{\phi}_x \cdot d   \, {\phi}_y \ e^{-\frac 2h f}\\
\label{eq.prqm21}
&\quad + \frac{h^2}{Z_xZ_y}  \int_{\ft O_2(x,y)} d     \, {\phi}_x \cdot d   \,{\phi}_y \ e^{-\frac 2h f}. 
\end{align}
Since $\ft O_2(x,y)\subset \ft O_1(x)$, from~\eqref{eq.SS},~\eqref{eq.nablapsix=0}, and~\eqref{Zx},  there exists $c>0$ such that  for $h$ small enough:
\begin{equation}\label{eq.zxzy}
\frac{h^{2}}{Z_xZ_y}  \int_{\ft O_2(x,y)} d     \, {\phi}_x \cdot d   \, {\phi}_y \ e^{-\frac 2h f}=O\big ( e^{-\frac 1h (2f(\mbf j(x))-f(x)-f(y) +c)}\big ),
\end{equation}
where we used  $f(\mbf j(y))=f(\mbf j(x))$. 
Moreover, using item 1 in Definition~\ref{de.psix}, for all $z\in \mbf j(x)\cap \mbf j(y)$ (recall that  $\mbf j(x)\cap \mbf j(y)\subset \Omega$), $d{\phi}_x=-d{\phi}_y$ on $\ft V^{\delta_1,\delta_2}_{\, \overline \Omega}(z)$.  Thus, from~\eqref{eq.psix--},  for all $z\in \mbf j(x)\cap \mbf j(y)$, it holds:
$$
   \frac{h^2}{Z_xZ_y} \int_{\ft V^{\delta_1,\delta_2}_{\, \overline \Omega}(z)} \!\!\!\! d     \, {\phi}_x \cdot d   \, {\phi}_y \ e^{-\frac 2h f}= - \frac{Z_x}{Z_y} \int_{\ft V^{\delta_1,\delta_2}_{\, \overline \Omega}(z)} \vert d_{f,h}   \,  \psi_x\vert^2.$$
 Then, item 2.(ii) in Proposition~\ref{pr.qm2} is a consequence of~\eqref{eq.prqm21} and~\eqref{eq.zxzy} together with~\eqref{Zx} and item 2 in Proposition~\ref{pr.qm1}.  
\end{enumerate} 
This concludes the proof of    Proposition~\ref{pr.qm2}.
\end{proof}

\subsection{Linear independence of the quasi-modes
}
Let us recall that according to Theorem~\ref{thm.count},  there exists  $c_0>0$  such that   for every $h$ small enough:
$$\dim \Ran \, \pi_{[0,c_0h]}\big (\Delta_{f,h}^{D}\big )=\ft  m_0.$$
In the following, for ease of notation, one denotes
\begin{equation}\label{eq.pih}
 \pi_h:=\pi_{[0,c_0h]}\big (\Delta_{f,h}^{D}\big ).
 \end{equation}
In this section,  one proves that   for every $h$ small enough,    $\big(\pi_h\psi_x\big)_{x\in \ft U_0}$ is
linearly independent, and hence 
  a basis of $\Ran\,   \pi_h$, and that
  $\big(d_{f,h}\pi_h\psi_x\big)_{x\in \ft U_0}$ is linearly independent in $\Lambda^1L^2(\Omega)$.   
Let us start with the  following result. 

\begin{proposition}\label{pr.qm3}
Let $f:\overline \Omega\to \mathbb R$  be a $C^\infty$ Morse function which satisfies \autoref{H1} and~\autoref{H2}.  Let $x\in \ft U_0 $ and   $\psi_x$     be as introduced in Definition~\ref{de.QM}.  Then, there exists $C>0$ such that  for every $h$ small enough:
 $$\big \Vert  (1-\pi_h) \psi_x\big \Vert_{L^2(\Omega)}\le C\,\big \Vert d_{f,h}\psi_x \big  \Vert_{L^2(\Omega)}$$
and
$$\big \Vert d_{f,h} (\pi_h\psi_x)   \big \Vert_{\Lambda^1L^2(\Omega)}=\big \Vert d_{f,h} \psi_x  \big \Vert_{\Lambda^1L^2(\Omega)} \big(1+O(h)\big).$$
\end{proposition}

\begin{proof}
Let $c_0>0$ be  the constant  used to define $\pi_h$ in \eqref{eq.pih}.  According  to Theorem~\ref{thm.count}, for every $h$ small enough,~$\Delta^D_{f,h}$ has $\ft m_0$ eigenvalues smaller than $c_0h$ which are moreover exponentially small. 
 Let  $C(\frac{c_0}{2}h)\subset \mathbb C$  be the circle centered at $0$  of radius $\frac{c_0}{2}h$. Then,
 there exists $c>0$ such that
  for every $h$ small enough, all the points in $C(\frac{c_0}{2}h)$ are at a distance larger than $ch$ of the spectrum of~$\Delta^D_{f,h}$. Thus, by the spectral theorem, it holds: 
\begin{equation}\label{eq.distc}
\sup_{z\in C(\frac{c_0}{2}h) } \big \Vert  (z-\Delta^D_{f,h})^{-1}\big \Vert_{L^2(\Omega)\to L^2(\Omega)} \le \frac{1}{ch}.  
\end{equation}
Moreover, since $\psi_x\in D(\Delta_{f,h}^D)$ for all $x\in \ft U_0 $ (see~\eqref{eq.nablapsix-domaine}), it holds
\begin{align}
\nonumber
 (1-\pi_h) \psi_x&=\frac{1}{2\pi \mathrm{i}}\int_{ C(\frac{c_0}{2}h) }\!\! \big (z^{-1}-  (z-\Delta^D_{f,h})^{-1}\big)\psi_x \, dz \\
\nonumber
 &=-\frac{1}{2\pi \mathrm{i}}\int_{ C(\frac{c_0}{2}h) }\!\!\!\!z^{-1}\big    (z-\Delta^D_{f,h})^{-1} \Delta^D_{f,h}\psi_x\, dz.
\end{align}
Thus, using~\eqref{eq.distc} and  the second estimate in  item 1 in Proposition~\ref{pr.qm2},  one obtains that  
$$\big \Vert  (1-\pi_h) \psi_x\big \Vert_{L^2(\Omega)}\le C \Vert d_{f,h} \psi_x  \big \Vert_{\Lambda^1L^2(\Omega)},$$ for some $C>0$ independent of $h$. 
Let us now prove the second asymptotic estimate of Proposition~\ref{pr.qm3}. Since the orthogonal projector $\pi_h$ and $\Delta_{f,h}^D$ commute  on $D(\Delta_{f,h})$ and $\psi_x\in D(\Delta_{f,h}^D)$,  one has
 \begin{align*}
 \big \Vert d_{f,h}( \pi_h\psi_x )  \big \Vert_{\Lambda^1L^2(\Omega)}^2&= \langle    \pi_h \psi_x    ,   \Delta_{f,h}  \psi_x   \big \rangle_{ L^2(\Omega)}\\
 &=  \langle  \psi_x    ,   \Delta_{f,h}  \psi_x   \big \rangle_{ L^2(\Omega)}- \langle    (1-\pi_h) \psi_x    ,   \Delta_{f,h}  \psi_x   \big \rangle_{ L^2(\Omega)}\\
 &=\big \Vert d_{f,h}  \psi_x   \big \Vert_{\Lambda^1L^2(\Omega)}^2+ O\big ( \big \Vert  (1-\pi_h) \psi_x\big \Vert_{L^2(\Omega)} \big \Vert \Delta_{f,h}   \psi_x\big \Vert_{L^2(\Omega)}\big)
\\
&=\big \Vert d_{f,h}  \psi_x   \big \Vert_{\Lambda^1L^2(\Omega)}^2 + O(h)\big \Vert d_{f,h}  \psi_x   \big \Vert_{\Lambda^1L^2(\Omega)}^2,
\end{align*}
where one used at the last line the second asymptotic estimate in  item 1 in Proposition~\ref{pr.qm2} and the first asymptotic estimate in  Proposition~\ref{pr.qm3}. This concludes the proof of Proposition~\ref{pr.qm3}. 
\end{proof}

\begin{remark}
Note here that using the estimate \eqref{quadra} to obtain an upper bound on $\big \Vert  (1-\pi_h) \psi_x\big \Vert_{L^2(\Omega)}$, one would obtain $\big \Vert  (1-\pi_h) \psi_x\big \Vert_{L^2(\Omega)}\le \frac{1}{\sqrt{c_0 h}} \Vert d_{f,h} \psi_x  \big \Vert_{\Lambda^1L^2(\Omega)}$. This  would finally lead to a remainder term of order $O(\sqrt h)$  instead of the $O(h)$  appearing in~\eqref{eq.eq-lambdathm} in Theorem~\ref{th.thm3} below.
\end{remark}


\begin{definition}\label{de.proj-qm}
Let $f:\overline \Omega\to \mathbb R$  be a $C^\infty$ Morse function which satisfies \autoref{H1} and~\autoref{H2}.  Let $x\in \ft U_0 $ and   $\psi_x$     be as introduced in Definition~\ref{de.QM}.  Then, one defines the $1$-form:
$${\Theta}_x:= \frac{d_{f,h} \psi_x}{\Vert d_{f,h} \psi_x   \Vert_{\Lambda^1L^2(\Omega)}   },$$
which is   $C^\infty$  on $\overline \Omega$ and supported in $\overline {\ft \Omega_1(x)}\setminus \ft \Omega_2(x)$ (see~\eqref{eq.psix--}). Notice that from item~1 in Proposition~\ref{pr.qm2}, $\Vert d_{f,h} \psi_x   \Vert_{\Lambda^1L^2(\Omega)} \neq 0$ for $h$ small enough. Moreover,   for every $h$ small enough, one defines: 
$${\psi}^\pi_x:= \frac{ \pi_h \, \psi_x}{\Vert \pi_h \, \psi_x   \Vert_{L^2(\Omega)}   } \ \text{ and } \ {\Theta}^\pi _x:= \frac{d_{f,h} (\pi_h \psi_x)}{\Vert d_{f,h} (\pi_h  \psi_x )  \Vert_{\Lambda^1L^2(\Omega)}   },$$
which are well defined for every $h$ small enough (see indeed Proposition~\ref{pr.qm3}) and where we recall
that the orthogonal projector $\pi_h$ on $L^2(\Omega)$  is defined by~\eqref{eq.pih}. 

\end{definition}
\noindent
A  consequence of Proposition~\ref{pr.qm3}  on the families   $\big (\psi^\pi _x\big)_{x\in \ft U_0}$ and    
$ \left ( \Theta^\pi _x\right)_{x\in \ft U_0}$ introduced in Definition~\ref{de.proj-qm} is the following. 

\begin{proposition}\label{pr.qm4}
Let $f:\overline \Omega\to \mathbb R$  be a $C^\infty$ Morse function which satisfies \autoref{H1} and~\autoref{H2}.  Let $x,y\in \ft U_0 $.  Then, there exists $c>0$ such that   for every $h$ small enough:
 $$\langle  \psi^\pi _x, \psi^\pi _y\big \rangle_{L^2(\Omega)}=\langle  \psi _x, \psi_y\big \rangle_{L^2(\Omega)}+  O\big (e^{-\frac ch}\big ),$$
and  
 $$ \langle  \Theta^\pi _x, \Theta^\pi _y\big \rangle_{\Lambda^1 L^2(\Omega)} =\langle  \Theta _x, \Theta _y\big \rangle_{\Lambda^1 L^2(\Omega)} + O(h).$$ 
\end{proposition}
\begin{proof}
Let us recall that  the orthogonal projector $\pi_h$ and $\Delta_{f,h}^D$ commute  on $D(\Delta_{f,h})$
and that $\psi_x\in D(\Delta_{f,h}^D)$. 
Then, for every $x,y\in \ft U_0 $, it holds
$$\langle    \pi_h \psi_x    ,    \pi_h \psi_y   \big \rangle_{ L^2(\Omega)}=\langle  \psi_x    ,     \psi_y   \big \rangle_{L^2(\Omega)}- \langle    (1-\pi_h) \psi_x    ,    \psi_y   \big \rangle_{ L^2(\Omega)},$$
and
$$\langle   d_{f,h} (\pi_h \psi_x )   ,  d_{f,h} (\pi_h \psi_y  ) \big \rangle_{\Lambda^1 L^2(\Omega)}=\langle d_{f,h} \psi_x    ,   d_{f,h}  \psi_y   \big \rangle_{\Lambda^1 L^2(\Omega)}- \langle    (1-\pi_h) \psi_x    ,   \Delta_{f,h}  \psi_y   \big \rangle_{ L^2(\Omega)}.$$
 Proposition~\ref{pr.qm4} is then a direct consequence of these identities together with Propositions~\ref{pr.qm2} and~\ref{pr.qm3} (see also~\eqref{eq.j-x1}).\end{proof}
  
\noindent
The Gram matrices of the families $\big (\psi  _x\big)_{x\in \ft U_0}$ and    
$ \left ( \Theta _x\right)_{x\in \ft U_0}$ are not necessarily quasi-unitary,  i.e. of the form ${\rm Id}+o(1)$ when $h\to 0$. 
 For the family $\big(\psi  _x\big)_{x\in \ft U_0}$, this follows from     the fact that  a global minimum  of $f$ in $\supp\psi _x$ can also be a global minimum  of $f$  in $\supp\psi_y$  
 (this can only occur in the situation described in
  item 4.(i) in Section~\ref{sec.j} when 
 ${\mbf j}(x) \cap  \mbf j (y)=\emptyset$ and, up to interchanging  $x$ with $y$, $\overline{\ft C_{\mbf j}(y)} \subset  { \ft C_{\mbf j}(x)}$).
 For the family $ \left ( \Theta _x\right)_{x\in \ft U_0}$, this follows from  the fact that 
 $\langle   d_{f,h} \psi_x    ,  d_{f,h}  \psi_y   \big \rangle_{\Lambda^1 L^2(\Omega)}$
  can be of the same order  as
 both $ \Vert   d_{f,h}  \psi_x  \Vert_{\Lambda^1 L^2(\Omega)}^2$ and $\Vert   d_{f,h}  \psi_y \Vert_{\Lambda^1 L^2(\Omega)}^2 $
 (see item 2.(ii) in Proposition~\ref{pr.qm2}). 
However, according to Proposition~\ref{pr.indep-psix}  below, these families are, in the limit $h\to 0$, uniformly linearly independent
 in the sense of  the following  definition (see~\cite{herau-hitrick-sjostrand-11}).

 
 \begin{definition}\label{li}
Let    $\mc H$ be a   Hilbert space,  $n\ge 1$ be an integer smaller than $\dim\mc H$, and $\mathcal B'$
be a family of $n$ elements of $\mc H$  depending on a parameter $h>0$.  The family $\mathcal B'$ is said to be uniformly linearly independent in the limit $h\to 0$   if
there exists $C>0$ and $h_0>0$ such that for all $h \in (0,h_0)$,
the family $\mathcal B'$ is linearly independent and
for some (and thus for any) orthonormal     family    
 $\mathcal B$ 
 of  $\sspan\big( \mathcal B' \big)$ and for some (and thus for any)  matrix norm $\Vert \cdot\Vert $  on $\mathbb R^{n \times n}$, 
 it holds
$$
\Big \Vert {\rm Mat}_{\mathcal B,\mathcal B'  }({\rm Id})\Big  \Vert \le C \quad
 \text{and}
 \quad
 \Big  \Vert  {\rm Mat}_{\mathcal B',  \mathcal B }({\rm Id})\Big \Vert \le C.
$$
 \end{definition}

\begin{remark}
\label{re.li}
Since the Gram matrix $G^{\mathcal B'}$ of $\mathcal B'$ writes $G^{\mathcal B'}=\,^t {\rm Mat}_{\mathcal B',\mathcal B  }({\rm Id})\,  {\rm Mat}_{\mathcal B',\mathcal B  }({\rm Id})$, 
 the family $\mathcal B'$ is uniformly linearly independent
in the limit $h\to 0$ if and only if there exists a  constant $C>0$ independent of $h$
such that, for every $h$ small enough,   $\frac 1C \leq G^{\mathcal B'}\leq C$ in the sense of quadratic forms.
\end{remark}
 

\begin{proposition}\label{pr.indep-psix} 
Let $f:\overline \Omega\to \mathbb R$  be a $C^\infty$ Morse function which satisfies \autoref{H1} and~\autoref{H2}.  Then,   the family of functions    $\big (\psi^\pi _x\big)_{x\in \ft U_0}$ (resp. the family of $1$-forms   
$ \left ( \Theta^\pi _x\right)_{x\in \ft U_0}$)    introduced in Definition~\ref{de.proj-qm} is uniformly  linearly independent in $L^2(\Omega)$ (resp. in $\Lambda^1L^2(\Omega)$) in the limit $h\to 0$ (see Definition~\ref{li}).
In particular, $\big (\psi^\pi _x\big)_{x\in \ft U_0}$ is a basis of $\Ran\, \pi_{h}$
for every $h$ small enough.
\end{proposition}

\noindent
The following lemma, which is a direct consequence
of 
Proposition~\ref{pr.qm1}, item 1 in Proposition~\ref{pr.qm2}, and  
Definition~\ref{de.proj-qm},
will be used in the proof of Proposition~\ref{pr.indep-psix}.
\begin{lemma}\label{le.debut1}
Let $f:\overline \Omega\to \mathbb R$  be a $C^\infty$ Morse function which satisfies \autoref{H1} and~\autoref{H2}, and $x\in \ft U_0$.
\begin{enumerate}
\item When there exists $z\in  \mbf j(x)$ such that $\vert \nabla f(z)\vert \neq 0$ (in this case $z\in  \pa \Omega$), one has in the limit $h\to0$, for every $z\in  \mbf j(x)$ such that $\vert \nabla f(z)\vert \neq 0$,
\begin{align*}
\Vert \Theta_x\Vert_{\Lambda^1L^2\big (\ft V_{\, \overline \Omega}^{\delta_1,\delta_2}(z)\big)}^2&=  \frac{ \ft c_{x,z} } { \sum_{p\in {\mbf j}(x), \vert \nabla f(p)\vert \neq 0 }  c_{x,p} } \big(1+O(\sqrt h)\big),
\end{align*}
where  the constant $\ft c_{x,z}$ is defined  in~\eqref{eq.cxz1} and $\ft V_{\, \overline \Omega}^{\delta_1,\delta_2}(z)$ is defined in~\eqref{eq.vois-1-pc}.
\item  When 
$\vert \nabla f(z)\vert = 0$ for every $z\in \mbf j(x)$,
one has  in the limit $h\to 0$, for every $z\in \mbf j(x)$, 
\begin{align*}
\Vert \Theta_x\Vert_{\Lambda^1L^2(\ft V_{\, \overline \Omega}^{\delta_1,\delta_2}(z)))}^2&=  \frac{ \ft c_{x,z} } { \sum_{p\in  \mbf j(x)    }  c_{x,p} } \big(1+O(\sqrt h)\big),
\end{align*}
where the constants $\ft c_{x,z}$ are defined in~\eqref{eq.cxz2} and~\eqref{eq.cxz3}
and $\ft V_{\, \overline \Omega}^{\delta_1,\delta_2}(z)$ is defined in~\eqref{eq.vois-11-pc} and~\eqref{eq.vois-1-pc-omega}.
\end{enumerate} 
\end{lemma}


\begin{proof}[Proof of Proposition~\ref{pr.indep-psix}]\begin{sloppypar}
In view of Proposition \ref{pr.qm4} and of Remark~\ref{re.li}, Proposition~\ref{pr.indep-psix}  is equivalent to the fact that    
 the family   $\big (\psi _x\big)_{x\in \ft U_0}$ (resp.    
$ \left ( \Theta _x\right)_{x\in \ft U_0}$)     is  uniformly linearly independent in $L^2(\Omega)$ (resp. in $\Lambda^1L^2(\Omega)$), in the limit $h\to 0$. 
Moreover,
the proof of this property for 
 $\big (\psi _x\big)_{x\in \ft U_0}$      is   exactly the same as the one made in~\cite[Section~4.2]{herau-hitrick-sjostrand-11}.
Let us now prove that $ \left ( \Theta _x\right)_{x\in \ft U_0}$      is  uniformly linearly independent in  $\Lambda^1L^2(\Omega)$ in the limit $h\to 0$. 
 The following proof is inspired by the analysis done in \cite[Section~4.2]{herau-hitrick-sjostrand-11}. 
Let us recall that according to the construction of $\ft C_{\mbf j}$ made in  Section~\ref{sec.j}, one has:
$$(\ft C_{\mbf j} (x))_{x\in \ft U_0}= \bigcup_{k\ge 1} \big \{\ft C_{\mbf j} (x_{k,1}),\ldots,\ft C_{\mbf j} (x_{k,\ft N_k})\big \},$$
where the union over $k$ is actually finite.  
For all $k\ge 1$, let us divide $\big \{\ft C_{\mbf j} (x_{k,1}),\ldots,\ft C_{\mbf j} (x_{k,\ft N_k})\big \}$ into $n_k$  groups  ($n_k \le \ft N_k$):
$$\big \{\ft C_{\mbf j} (x_{k,1}),\ldots,\ft C_{\mbf j} (x_{k,\ft N_k})\big \} =\bigcup_{\ell=1}^{n_k}  \{\ft C^1_{k,\ell},\ldots,\ft C^{m_\ell}_{k,\ell}\} $$
which are such that for all $\ell\in \{1,\ldots,n_k\}$,   
\begin{equation}\label{eq.debutdelarecurrence}
  \left\{
    \begin{array}{ll}
      \text{the set } \, \bigcup_{j=1}^{m_\ell} \overline{\ft C^j_{k,\ell}} \ \text{ is connected, and} \\
      \forall    \ft C  \in \big \{\ft C_{\mbf j} (x_{k,1}),\ldots,\ft C_{\mbf j} (x_{k,\ft N_k})\big \} \setminus  \{\ft C^1_{k,\ell},\ldots,\ft C^{m_\ell}_{k,\ell}\}, \, \overline{\ft C}  \cap \bigcup_{j=1}^{m_\ell} \overline{\ft C^j_{k,\ell}}  =\emptyset.
    \end{array}
\right.
\end{equation}
Let $x,y\in \ft U_0$. Let $k$, $k'$, $\ell$, and $\ell'$ be such that $\ft C_{\mbf j}(x)\in  \{\ft C^1_{k,\ell},\ldots,\ft C^{m_\ell}_{k,\ell}\}$ and $\ft C_{\mbf j}(y)\in  \{\ft C^1_{k',\ell'},\ldots,\ft C^{m_{\ell'}}_{k',\ell'}\}$.  
 Let us recall that  $\mbf j(x)\cap \mbf j(y)\neq \emptyset$ is equivalent to  $f(\mbf j(x))=f(\mbf j(y))$ (which implies $ k=k'$) and $\overline{\ft C_{\mbf j}(y)} \, \cap \, \overline{ \ft C_{\mbf j}(x)}\neq \emptyset $ (which implies  $\ell=\ell'$). Therefore, when  $\ft C_{\mbf j}(x)$ and $\ft C_{\mbf j}(y)$ belong to different groups, i.e. when   $(k',\ell')\neq (k,\ell)$, it holds  $\mbf j(x)\cap \mbf j(y)= \emptyset$.  
Thus, 
according to
item 2.(i) in Proposition~\ref{pr.qm2}   
and to Definition~\ref{de.proj-qm}, it holds $\langle  \Theta _x, \Theta _y\big \rangle_{\Lambda^1 L^2(\Omega)} =0$.
This implies that in $\Lambda^1L^2(\Omega)$, it holds:
\begin{equation}\label{eq.ortho}
\sspan\big(\left ( \Theta _x\right)_{x\in \ft U_0}\big)=\bigoplus^\perp_{k\ge 1}  \Big[\bigoplus^\perp_{\ell=1,\ldots,n_k}   \sspan\big(  \Theta _x, \ x \text{ s.t.  } \ft C_{\mbf j}(x)\in \{\ft C^1_{k,\ell},\ldots,\ft C^{m_\ell}_{k,\ell}\}   \big) \Big].
\end{equation}
According to Definition~\ref{li}, in order to prove that $ \left ( \Theta _x\right)_{x\in \ft U_0}$ 
is uniformly linearly independent in the limit $h\to 0$, 
it then suffices to prove that
for all $k\ge 1$ and $\ell\in \{1,\ldots,n_k\}$, 
the family $\big(  \Theta _x, \ x \text{ s.t.  } \ft C_{\mbf j}(x)\in \{\ft C^1_{k,\ell},\ldots,\ft C^{m_\ell}_{k,\ell}\}   \big)$
is uniformly linearly independent in the limit $h\to 0$.
  To this end, let $k\ge 1$ and $\ell\in \{1,\ldots,n_k\}$. For ease of notation, 
we denote  $m_{\ell}$ by $m$, $\{\ft C^1_{k,\ell},\ldots,\ft C^{m_\ell}_{k,\ell}\}$ by $\{\ft C^1,\ldots,\ft C^m\}$, and 
$\big(  \Theta _x, \ x \text{ s.t.  } \ft C_{\mbf j}(x)\in \{\ft C^1_{k,\ell},\ldots,\ft C^{m_\ell}_{k,\ell}\}   \big)$ by $(\Theta_1,\ldots,\Theta_m)$.
 For $h$ small enough, let us then consider some $\varphi=\varphi(h)\in \sspan \{\Theta_1,\ldots,\Theta_m\}$: 
 \begin{equation}\label{eq.spann}
 \varphi =\sum_{i=1   }^m a_i(h)  \, \Theta_i, \text{ where for all $i\in \{1,\ldots,m\}$, $a_i(h) \in \mathbb R$}.
 \end{equation}
From~\eqref{eq.debutdelarecurrence} and using Lemma~\ref{le.debu-recurrence}, up to reordering $\{\ft C^1,\ldots,\ft C^m\}$, there exists $z_{1}\in \ft U_1^{\ft{ssp}}(\overline \Omega)$    such that 
$z_{1}\in \pa \ft C^{1}  \setminus \big (    \cup_{i=2}^{m}\pa {\ft C^i }\big)$. 
Let us now choose such a point $z_{1}$ 
as follows:
\begin{itemize}
\item[--] When $\{p\in \pa \ft C^{1}\cap \ft U_1^{\ft{ssp}}(\overline \Omega) \text{ s.t.} \vert \nabla f(p)\vert \neq 0\}=\emptyset$, 
one chooses any $z_1$  in
$\ft U_1^{\ft{ssp}}(\overline \Omega)\cap \pa \ft C^{1}  \setminus \big (    \cup_{i=2}^{m}\pa {\ft C^i }\big)$ (and it holds    $\vert \nabla f(z_1)\vert =0$). 
\item[--] When $\{p\in \pa \ft C^{1}\cap \ft U_1^{\ft{ssp}}(\overline \Omega) \text{ s.t. } \vert \nabla f(p)\vert \neq 0\}\neq \emptyset$,
then $ \ft C^{1}$ is a principal well of $f$ (see~\eqref{eq.j-x})  and  thus $\ft C^1,\ldots,\ft C^{m}$ are principal wells of $f$. 
In this case, one chooses $z_1\in \ft U_1^{\ft{ssp}}(\overline \Omega)\cap\{p\in \pa \ft C^{1} \text{ s.t.} \vert \nabla f(p)\vert \neq 0\}\subset \pa \Omega$ and from~\eqref{eq.sspCk},  it  holds $z_1\notin      \cup_{i=2}^{m}\pa {\ft C^i }$. 
\end{itemize}
In   both cases, according to Lemma~\ref{le.debut1}, 
one has   when $h\to 0$,   
$$\Vert \Theta_1\Vert_{\Lambda^1L^2(\ft V_{\, \overline \Omega}^{\delta_1,\delta_2}(z_1))}=c_1(1+o(1))\,,$$  where  $c_1\in  (0,1]$ is independent of $h$. 
Since  $z_1\in \pa \ft C^{1}  \setminus \big (    \cup_{i=2}^{m}\pa {\ft C^i }\big)$ and since all  the cylinders defined by~\eqref{eq.vois-1-pc},~\eqref{eq.vois-11-pc}, and \eqref{eq.vois-1-pc-omega} are two by two disjoint, the cylinder $\ft V_{\, \overline \Omega}^{\delta_1,\delta_2}(z_1)$ does not meet any of the cylinders associated with the  $z\in \ft U_1^{\ft{ssp}}(\overline \Omega)\cap \cup_{i=2}^{m}\pa {\ft C^i }$.  
 Therefore, by definition of $ \Theta_i$ (see Definition~\ref{de.proj-qm}) and  item 3 in Definition~\ref{de.omega1}, it holds  $  \Theta_i\equiv 0$ on $\ft V_{\, \overline \Omega}^{\delta_1,\delta_2}(z_1)$
 for all $i\in \{2,\ldots,m\}$. 
Taking the $\Lambda^1L^2$-norm of~\eqref{eq.spann} in $\ft V_{\, \overline \Omega}^{\delta_1,\delta_2}(z_1)$, one has for $h$ small enough, $\Vert \varphi \Vert_{\Lambda^1L^2(\Omega)}\ge \Vert \varphi\Vert_{\Lambda^1L^2(\ft V_{\, \overline \Omega}^{\delta_1,\delta_2}(z_1))}\ge \frac{c_1}{2} \vert a_1(h)\vert$. 
Thus, for $h$ small enough, it holds:
\begin{equation}\label{eq.step1-theta}
 \vert a_1(h)\vert   \le \frac{2}{c_1}\Vert \varphi \Vert_{\Lambda^1L^2(\Omega)}.
 \end{equation}
 Let us now get a similar upper bound on $ \vert a_2(h)\vert $. 
Since $ \cup_{i=1}^{m}\overline {\ft C^i }$ is connected (see~\eqref{eq.debutdelarecurrence}), up to reordering  $\{\ft C^2,\ldots,\ft C^m\}$, it holds $\overline{\ft C^1}\cap \overline{\ft C^2}\neq \emptyset$,
and one chooses $z_2  \in \ft U_1^{\ft{ssp}}(\overline \Omega)\cap \pa \ft C_2$ 
as follows:
\begin{itemize}
\item[--] When $\{p\in \pa \ft C^{2}\cap \ft U_1^{\ft{ssp}}(\overline \Omega) \text{ s.t.} \vert \nabla f(p)\vert \neq 0\}=\emptyset$, one chooses  any $z_2\in \pa \ft C^{2}\cap \pa \ft C^{1}$.
\item[--] When $\{p\in \pa \ft C^{2}\cap \ft U_1^{\ft{ssp}}(\overline \Omega) \text{ s.t.} \vert \nabla f(p)\vert \neq 0\}\neq\emptyset$,
 one chooses  $z_2\in \{p\in \pa \ft C^{2}\cap \ft U_1^{\ft{ssp}}(\overline \Omega) \text{ s.t.} \vert \nabla f(p)\vert \neq 0\}$.
\end{itemize}
In both cases, $z_2\in  \ft U_1^{\ft{ssp}}(\overline \Omega)\cap\pa \ft C^{2}  \setminus \big (    \cup_{i=3}^{m}\pa {\ft C^i }\big)$. Therefore, it holds  $ \Theta_i\equiv 0$ on $\ft V_{\, \overline \Omega}^{\delta_1,\delta_2}(z_2)$ for all $i\in \{3,\ldots,m\}$  while, from Lemma~\ref{le.debut1},  $\Vert \Theta_2\Vert_{\Lambda^1L^2(\ft V_{\, \overline \Omega}^{\delta_1,\delta_2}(z_2))}=c_2(1+o(1))$ in the limit $h\to 0$ and for some $c_2 \in (0,1]$ independent of $h$. Taking the $\Lambda^1L^2$-norm of~\eqref{eq.spann}
in $\ft V_{\, \overline \Omega}^{\delta_1,\delta_2}(z_2)$
 and using the fact that $\Vert \Theta_1\Vert_{\Lambda^1L^2(\ft V_{\, \overline \Omega}^{\delta_1,\delta_2}(z_2))}\le 1$
 lead to
 $$\Vert \varphi \Vert_{\Lambda^1L^2(\Omega)}\ge \Vert \varphi\Vert_{\Lambda^1L^2(\ft V_{\, \overline \Omega}^{\delta_1,\delta_2}(z_2))}\ge  -\vert a_1(h)\vert    + \frac{c_2}{2} \vert a_2(h)\vert$$  for every $h$ small enough. 
Using in addition~\eqref{eq.step1-theta}, one obtains  
$$
 \vert a_2(h)\vert   \le \frac{2}{c_2}\Big(1+\frac{2}{c_1}\Big)\Vert \varphi \Vert_{\Lambda^1L^2(\Omega)}.
$$
\end{sloppypar}  
\noindent
Repeating this last procedure $m-2$ times   leads to the existence of some $C>0$ independent of $h$ such that for every $h$ small enough, it holds $\sum_{i=1   }^m\vert  a_i(h)\vert  \le C\, \Vert \varphi \Vert_{\Lambda^1L^2(\Omega)}$. 
Using \eqref{eq.spann}, it follows that the family $(\Theta_1,\ldots,\Theta_m)$
is uniformly linearly independent  
in the limit $h\to 0$,
which concludes the proof of Proposition~\ref{pr.indep-psix}.   
\end{proof}




\subsection{An accurate interaction matrix}
\label{sec.finite-dim}

\begin{sloppypar}
Let $f:\overline \Omega\to \mathbb R$  be a $C^\infty$ Morse function which satisfies \autoref{H1} and~\autoref{H2}.
 In the rest of this section,
  one chooses for ease of notation   an arbitrary labeling of $\ft U_0=\{x_1,\ldots,x_{\ft m_0}\}$ and one assumes that  $ \left ( \psi_x\right)_{x\in \ft U_0}=\left (\psi_{1},\ldots,\psi_{{\ft m_0}}\right)$ and $\left ( \Theta_x\right)_{x\in \ft U_0} =\left (\Theta_{1},\ldots,\Theta_{{\ft m_0}}\right)$ (see Definitions~\ref{de.QM} and~\ref{de.proj-qm}) are ordered according
to this labeling. 
 
\medskip
\noindent
Let us recall from Proposition~\ref{pr.indep-psix} that    for every $h$ small enough,  $\big (\psi^\pi  _j\big)_{j\in \{1,\ldots,\ft m_0\}}$ and $ \left ( \Theta _i^\pi\right)_{i\in \{1,\ldots,\ft m_0\}}$ are uniformly linearly independent 
  (see Definitions~\ref{de.proj-qm} and~\ref{li}),
  which implies in  particular, according to Theorem~\ref{thm.count}, that
$$\sspan\big (\psi^\pi  _j\big)_{j\in \{1,\ldots,\ft m_0\}}=\Ran\,\pi_h.$$
Let us now consider an orthonormal   basis $\mathcal B_0$ of $\Ran(\pi_h)$
 in $L^2(\Omega)$ and   an orthonormal basis   $\mathcal B_1$   of $\sspan\left ( \Theta _i^\pi\right)_{i\in \{1,\ldots,\ft m_0\}}$ in $\Lambda^1L^2(\Omega)$. 
The eigenvalues of $\Delta_{f,h}^{D}$ which are smaller than~$c_0h$ for $h$ small enough are
then the eigenvalues of the matrix $M^{\mathcal B_0}$ of $\Delta_{f,h}^{D}$ in the basis $\mathcal B_0$, and hence
 the squares of the singular values of the  matrix $S^{\mathcal B_0,\mathcal B_1}$  
defined by
\begin{equation}\label{eq.sing.matrix}
S^{\mathcal B_0,\mathcal B_1}:=\text{Mat}_{\mathcal B_0,\mathcal B_1}(d_{f,h}),
\end{equation}
which  follows from the relation $M^{\mathcal B_0}=\, ^t S^{\mathcal B_0,\mathcal B_1}  S^{\mathcal B_0,\mathcal B_1}$. 
This reduces the analysis of  the asymptotic behaviour of the $\ft m_0$ smallest eigenvalues of $\Delta_{f,h}^D$ in the limit $h\to 0$ to the  study of   the asymptotic behaviour of  the singular values of the matrix $S^{\mathcal B_0,\mathcal B_1}$.\medskip

\noindent
Note moreover that according to Definition~\ref{de.proj-qm}, the   matrix $S^{\mathcal B_0,\mathcal B_1}$     defined by~\eqref{eq.sing.matrix} has the form
\begin{equation}
\label{eq.equa-Spi}
S^{\mathcal B_0,\mathcal B_1}=   \, ^tC_1^\pi \, S^\pi \,  C_0^\pi ,
\end{equation}
where 
\begin{equation}
\label{eq.Spi}
C_1^\pi :=\text{Mat}_{\mathcal B_1 \, ,    \,     \left ( \Theta _i^\pi\right)_{i\in \{1,\ldots,\ft m_0\}} } ({\rm Id}) , \ \  \  C_0^\pi :=\text{Mat}_{\mathcal B_0\, ,    \,   \big (\psi^\pi  _j\big)_{j\in \{1,\ldots,\ft m_0\}}  }  ({\rm Id}),
\end{equation}
and
\begin{equation}
\label{eq.Spi2}
\begin{aligned}
\text{for all $i,j\in \{1,\ldots,\ft m_0\}$},\  \  S_{i,j}^\pi &=\frac{  \langle  d_{f,h} \psi^\pi  _j  , d_{f,h} \psi^\pi  _i \big \rangle_{\Lambda^1 L^2(\Omega)}}{ \big \Vert d_{f,h}  \psi^\pi  _i \big \Vert_{\Lambda^1L^2(\Omega)} }\\
&=
\big \Vert d_{f,h}  \psi^\pi  _j \big \Vert_{\Lambda^1L^2(\Omega)}\langle  \Theta _j^\pi , \Theta _i^\pi \big \rangle_{\Lambda^1 L^2(\Omega)}. 
\end{aligned}
\end{equation}
\end{sloppypar} 
\noindent
In order to give asymptotic estimates on the entries of the  matrix $S^\pi$ in the limit $h\to 0$, let us  introduce the   square matrix $S$ defined by:   
\begin{equation}
\label{eq.Sij}
\begin{aligned}
\text{for all $i,j\in \{1,\ldots,\ft m_0\}$},\  \ S_{i,j}:&=\frac{  \langle  d_{f,h} \psi_{j}  , d_{f,h} \psi_{i} \big \rangle_{\Lambda^1 L^2(\Omega)}}{ \big \Vert d_{f,h} \psi_{i} \big \Vert_{\Lambda^1L^2(\Omega)} }\\
&=
\big \Vert d_{f,h}  \psi  _j \big \Vert_{\Lambda^1L^2(\Omega)}\langle  \Theta _j , \Theta _i \big \rangle_{\Lambda^1 L^2(\Omega)}. 
\end{aligned}
\end{equation}
From       Propositions~\ref{pr.qm2}, \ref{pr.qm3}, and \ref{pr.qm4}, one has the following asymptotic result on the entries of the matrices $S$ and $S^\pi$.

\begin{proposition}\label{pr.interS}
Let $f:\overline \Omega\to \mathbb R$  be a $C^\infty$ Morse function which satisfies \autoref{H1} and~\autoref{H2},   and $i,j\in \{1,\ldots,\ft m_0\}$. We
then have the following estimates when $h\to 0$:
\begin{enumerate}
\item When $\mbf j(x_i)\cap \mbf j(x_j)=\emptyset$,  $S_{i,j}=0$. 
\item When  $\mbf j(x_i)\cap \mbf j(x_j)\neq \emptyset$ and $i=j$, 
$$S_{j,j}= h^{\frac 14}\, \Big(    \,   \ft K_{1,x_j} (1+ {O\big (   h\, \big )}) +  h^{\frac 12} \,  \ft K_{2,x_j}   (1+O(\sqrt  h))\Big)^{\frac 12} e^{-\frac 1h (f(\mbf j(x_j))-f(x_j))}$$
and, when  $\mbf j(x_i)\cap \mbf j(x_j)\neq \emptyset$ and $i\neq j$, 
$$S_{i,j}=  - \frac{    h^{\frac 34}   {\ft K_{x_i,x_j} } }  {  \big(   \ft K_{1,x_i}  (1+ {O\big (   h\, \big )})  +   h^{\frac 12}  \ft K_{2,x_i} (1+O(\sqrt h)) \big)^{\frac 12} } e^{-\frac 1h ( f(\mbf j(x_j))-  f(x_j)  )},$$
where the constants $\ft K_{1,x_j}$, $\ft K_{2,x_j}$, and $\ft K_{x_i,x_j} $ are defined in~\eqref{eq.cstCx} and~\eqref{eq.cstCxy}. 
\item Finally, it holds in any case
  $$S_{i,j}^\pi=S_{i,j} +O(h)S_{j,j}.$$
  \end{enumerate}
\end{proposition}



\noindent
In order to suitably factorize the matrix $S^\pi$, let us first write~$S= T D$,
where~$D$ and~$T$
are  the following  $\ft m_0\times \ft m_0$ matrices (defined for every $h$  small enough):
\begin{itemize}
\item the matrix $D$ is the diagonal matrix such that for all $j\in \{1,\ldots,\ft m_0\}$,
\begin{equation}\label{eq.DiagD}
 D_{j,j}:= h^{p_j} \, e^{-\frac 1h (f(\mbf j(x_j))-f(x_j))},
\end{equation}
where 
\begin{equation}\label{eq.pj}
p_j:= {\frac 14} \text{ when } \ft K_{1,x_j} \neq 0\  \text{ and }\  p_j:=\frac 12  \text{ when } \ft K_{1,x_j} = 0,
\end{equation}
\item the  matrix $T$ is the matrix $S D^{-1}$, i.e.
\begin{equation}\label{eq.T}
\text{for all } i,j\in \{1,\ldots,\ft m_0\}, \ \ \ T_{i,j}:= \frac{ S_{i,j}}{D_{j,j}}.
\end{equation}
\end{itemize}
It then follows from \eqref{eq.DiagD}--\eqref{eq.T} and Proposition~\ref{pr.interS}
that in the limit $h\to 0$,
$$
S^\pi\ =\ (T + R )D\ \ \ \text{with}\ \ \ R=  S^\pi D^{-1}-T=(S^\pi -S) D^{-1}=O(h)
$$
and $T=O(1)$. Moreover, according to Lemma~\ref{le.invT}  below,
$T$ is invertible and its inverse satisfies $T^{-1}=O(1)$.
Thus,  the matrix $S^\pi$ factorizes as follows:
\begin{equation}
\label{eq.SpiS}
S^\pi=(T + O(h) )D = (I_{\ft m_0} + O(h)T^{-1} )TD=(I_{\ft m_0} + O(h) )TD=
(I_{\ft m_0} + O(h) )S.
\end{equation}


\noindent
We conclude this section by stating and proving Lemma~\ref{le.invT}
which led to  \eqref{eq.SpiS}.

\begin{lemma}\label{le.invT} 
Let $f:\overline \Omega\to \mathbb R$  be a $C^\infty$ Morse function which satisfies \autoref{H1} and~\autoref{H2}. Let    $\Vert \cdot\Vert $ be a matrix norm on $\mathbb R^{\ft m_0 \times \ft m_0}$. Then, for every $h$ small enough, the matrix~$T$ defined by~\eqref{eq.T} is invertible and  there exists $C>0$ independent of $h$ such that
$$\Vert T\Vert \le C \text{ and }  \Vert T^{-1}\Vert \le C.$$ 
\end{lemma}

\begin{proof}  
We already noticed the relation $\Vert T\Vert =O(1)$ in the limit $h\to 0$.  
To prove the relation $\Vert T^{-1}\Vert =O(1)$,  let us
first
 notice  that from \eqref{eq.Sij}, \eqref{eq.DiagD}, \eqref{eq.T}, and Definition~\ref{de.proj-qm}, it holds 
$$T= SD^{-1}=G^{\Theta} U  D^{-1},$$
where 
$$U= \text{Diag } \big(\big \Vert d_{f,h} \psi_{1} \big \Vert_{\Lambda^1L^2(\Omega)},\ldots,\big \Vert d_{f,h} \psi_{{\ft m_0}} \big \Vert_{\Lambda^1L^2(\Omega)}      \big)=
\text{Diag } \big(S_{1,1},\dots,S_{\ft m_{0},\ft m_{0}}\big)$$ and 
  $G^{\Theta}$ is the Gram matrix of the family $(\Theta_{1},\ldots,\Theta_{{\ft m_0}})$ in $\Lambda^1L^2(\Omega)$.
  Moreover, according to \eqref{eq.DiagD}, \eqref{eq.pj}, and   Proposition~\ref{pr.interS}, 
  there exist positive constants $c_{1},\dots,c_{\ft m_{0}}$ such that
  $\lim_{h\to 0} U D^{-1}=  \text{Diag } (c_1,\ldots,c_{\ft m_0}     )$
  and thus  $ DU^{-1} = O(1)$.
  Lastly, let us recall from Proposition~\ref{pr.indep-psix} that the family 
  $(\Theta_{1},\ldots,\Theta_{{\ft m_0}})$ is uniformly linearly independent 
  in the limit $h\to 0$ and then, according to Remark~\ref{re.li},
 $(G^{\Theta})^{-1}=O(1)$. It follows that $T^{-1}=DU^{-1}  (G^{\Theta})^{-1}= O(1)$,
 which concludes the proof of Lemma~\ref{le.invT}. 
\end{proof}

\subsection{Asymptotic behaviour of the small eigenvalues of $\Delta_{f,h}^{D}$}
\label{sub.sharp}

In this section, one  
states and proves the main results of this work, Theorems~\ref{th.thm2}
and~\ref{th.thm3} below, on the precise asymptotic behaviour of the small eigenvalues of $\Delta_{f,h}^{D}$ in the limit $h\to 0$.
\medskip

\noindent
The proofs of these results make both  use of a weak form of the Fan inequalities 
stated in the following lemma (see for instance \cite[Theorem~1.6]{simon1979trace}). 

\begin{lemma}
\label{le.fan}
Let  $A$, $B$, and $C$ be three $\ft m_0\times \ft m_0$ matrices. It then holds: 
$$
\forall j\in\{1,\ldots,\ft m_0\},\  \eta_{j}(A\,B\,C)\leq\big \|A\big\|\,\big\|C\big\|\,\eta_{j}(B),
$$
where, for any matrix $U\in  \mathbb R^{\ft m_0 \times \ft m_0}$, $\eta_{1}(U)\geq\cdots\geq \eta_{\ft m_{0}}(U) $ denote the singular values of $U$  and  $ \|U\|:=\sqrt{ \max \sigma( \, ^{t}UU)}=\eta_{1}(U)$ is the spectral norm of  $U$.
\end{lemma}

\noindent
Notice that the singular values are labeled in decreasing order whereas  the eigenvalues are labeled in increasing order. 
In Theorem~\ref{th.thm2},
one gives a precise lower and  upper bound on every small
eigenvalue of $\Delta_{f,h}^{D}$ in the limit $h\to 0$
under the sole assumptions \autoref{H1} and~\autoref{H2}.

\begin{theorem}\label{th.thm2}
Let $f:\overline \Omega\to \mathbb R$  be a $C^\infty$ Morse function
 which satisfies \autoref{H1} and~\autoref{H2}, and thus such that
 $\ft U_{0}\neq \emptyset$.
Let us  order the  set $\ft U_0=\{x_1,\ldots,x_{\ft m_0}\}$    such that 
\begin{enumerate}
\item[--] the sequence 
$  \big (f(\mathbf j(x_{j}))-f(x_{j}) \big )_{j\in\{1,\dots,\ft m_0\}}$ is decreasing,
\item[--] and, on any $\mathcal J\subset \{1,\dots,\ft m_{0}  \} $ such that   $ \big (f(\mathbf j(x_{j}))-f(x_{j}) \big )_{j\in\mathcal J}$ is constant, the sequence $(p_{j})_{j\in\mathcal J}$ is decreasing (see~\eqref{eq.pj}). 
\end{enumerate}
Finally, for $j\in \mathbb N^{*}$,  let us denote by $\lambda_{j,h}$  the $j$-th eigenvalue of $\Delta_{f,h}^{D}$ counted with multiplicity. Then, there  exist $C>0$ and $h_0>0$ such that   for every $j\in \big \{1,\dots,\ft m_{0}  \big \}$ and every  $h\in (0,h_0)$, it holds
$$
\frac 1C \,h^{2p_{j}}\, e^{-\frac{2}{h}(f(\mathbf j(x_{j}))-f(x_{j}))} \,\leq \,\lambda_{j,h}\, \leq\, C \,h^{2p_{j}}\,e^{-\frac{2}{h} (f(\mathbf j(x_{j}))-f(x_{j}))}\,.
$$ 
\end{theorem}


\begin{proof} 
For any matrix $U\in  \mathbb R^{\ft m_0 \times \ft m_0}$, we will denote by 
$\|U\|$  the spectral norm of  $U$ and by
  $\|U\|=\eta_{1}(U)\geq\cdots\geq \eta_{\ft m_0 }(U) $  the singular values of $U$. 
Let us recall  from~Section~\ref{sec.finite-dim}
 that   the $\ft m_0$ smallest eigenvalues of $\Delta_{f,h}^D$ are squares of  the singular values of the matrix  
 $S^{\mathcal B_0,\mathcal B_1}= \,^tC_1^\pi S^\pi C_0^{\pi} \in  \mathbb R^{\ft m_0 \times \ft m_0}$,
 where $C_0^\pi$, $C_1^{\pi}$, and $S^\pi$
 are defined in \eqref{eq.Spi} and in \eqref{eq.Spi2}.
Moreover, using Proposition~\ref{pr.indep-psix}, 
there exists $c>0$ such that for every~$h$ small enough, it holds
\begin{equation}
\label{eq.borne-CC}
\max\Big (\big\|C_0^\pi\big\|,\big\|(C_0^\pi)^{-1}\big\|,\big\|C_1^\pi\big\|,\big\|(C_1^\pi)^{-1}\big\|\Big )\le c.
 \end{equation}
Thus, using Lemma~\ref{le.fan}, 
there exists $c>0$ such that for every $h$ small enough, it holds
\begin{equation}
\label{eq.sing-D1}
 \forall j\in\{1,\dots,{\ft m_{0}}\},\ \frac1c \,   \eta_{ j}(S^{\pi})      \le  \eta_{ j}(S^{\mathcal B_0,\mathcal B_1})\le c \,   \eta_{ j}(S^{\pi}).
 \end{equation}
Moreover, let us recall that $S^\pi= (I_{\ft m_0}+O(h) )T D$
according to 
\eqref{eq.SpiS} and then, using Lemmata~\ref{le.invT} and~\ref{le.fan}, there exists $c>0$ such that
 for every $h$ small enough, 
\begin{equation}
\label{eq.sing-D2}
 \forall j\in\{1,\dots,{\ft m_{0}}\},\ \frac1c\,   \eta_{ j}(D)      \le  \eta_{ j}(S^\pi )\le c \,   \eta_{ j}(D).
  \end{equation}
Finally, according to the ordering of the elements of $\ft U_{0}$ considered
 in the statement of   
Theorem~\ref{th.thm2}, the singular values of $D$ satisfy (see indeed \eqref{eq.DiagD}),
\begin{equation*}
\label{eq.sing-D3}
\forall j\in\{1,\dots,{\ft m_{0}}\},\  \eta_{{\ft m_{0} }+1-j}(D)= D_{j,j}=h^{p_{j}}e^{-\frac{ 1}{h}(f(\mathbf j(x_{j}))-f(x_{j}))}.
\end{equation*}
Together with \eqref{eq.sing-D1} and \eqref{eq.sing-D2}, this implies the statement of Theorem~\ref{th.thm2}.
\end{proof}

\noindent
Lastly, in the main result of this work stated below,
one gives asymptotic equivalents of the smallest eigenvalues of $\Delta_{f,h}^{D}$
under additional  assumptions on the maps~$\mbf j$ and~$\ft C_{\mbf j}$ built in Section~\ref{sec.j}
which ensure that the wells $\ft C_{\mbf j}(x)$, $x\in \ft U_{0}$, are adequately separated.

\begin{theorem}\label{th.thm3}
Let $f:\overline \Omega\to \mathbb R$  be a $C^\infty$ Morse function which satisfies \autoref{H1} and~\autoref{H2},
and thus such that $\ft U_{0}\neq \emptyset$. Let us assume that there exists 
$\ft m^*\in \{1,\ldots, \ft m_0\}$ and 
a labeling of $\ft U_0=\{x_1,\ldots,x_{\ft m_0}\}$  such that (see Section~\ref{sec.j} for the constructions of the maps $\mbf j$ and~$\ft C_{\mbf j}$):
\begin{enumerate}
\item  It holds
$$f(\mathbf j(x_{1}))-f(x_{1})\geq \ldots \geq f(\mathbf j(x_{\ft m^*}))-f(x_{\ft m^*})> \max\limits_{ i=\ft m^*+1,\ldots,\ft m_0}  f(\mathbf j(x_{i}))-f(x_{i}),$$ with the convention   $\max\limits_{ i= \ft m_0+1,\ft m_0}f(\mathbf j(x_{i}))-f(x_{i})=0$.
\item   For all $j \in  \{1,\ldots, \ft m^*\}$, \  $\mbf j(x_j)\cap \bigcup\limits_{i\in  \{1,\ldots, \ft m_0\}, i\neq j} \mbf j(x_i) =\emptyset$ (i.e.  $\pa \ft C_{\mbf j}(x_j)$ does not contain any separating saddle point which belongs to another   $\pa \ft C_{\mbf j}(x_i)$, $ i\neq j$). 
\item 
For all $k,\ell \in  \{1,\ldots, \ft m^*\}$ such that $k\neq \ell$ and    $\ft C_{\mbf j}(x_\ell)\subset \ft C_{\mbf j}(x_k) $ (notice that this implies  $f(x_\ell)\ge f(x_k)$ by construction of $\ft C_{\mbf j}$),  it holds $f(x_\ell)>f(x_k)$. 
\end{enumerate}
For $j\in \mathbb N^{*}$,  let us denote by $\lambda_{j,h}$  the $j$-th eigenvalue of $\Delta_{f,h}^{D}$ counted with multiplicity.
Then,
there exists $c>0$ such that in the limit $h\to 0$, it holds
$$
\lambda_{\ft m^*+1,h}=O(e^{-\frac ch} )\lambda_{\ft m^*,h}.
$$
Moreover,
there exists $h_{0}>0$ such that for every $h\in(0,h_{0})$,
 there exists a bijection 
$$\Lambda_h:\{x_{1},\dots,x_{\ft m^{*}}\}\longrightarrow \sigma(\Delta^{D}_{f,h})\cap[0,\lambda_{\ft m^{*},h}],$$
where the spectrum is counted with multiplicity, 
such that, for every $j\in\{1,\dots,\ft m^{*}\}$, it holds when $h\to 0$:
\begin{align}
\label{eq.eq-lambdathm}
\Lambda_h(x_{j})&=  \Big(   \sqrt h  \,   \ft K_{1,x_j} \big (1+ {O (  h  )}\big) +  h  \,  \ft K_{2,x_j}  \big(1+O(\sqrt  h)\big)\Big)  e^{-\frac 2h (f(\mbf j(x_j))-f(x_j))} \\
\nonumber
&= \Big(\frac{A_{j,1}+\sqrt h \, A_{j,2}}{B_{j}}+ O(h)\Big)
\,\sqrt {\frac h\pi} \ e^{-\frac 2h (f(\mbf j(x_j))-f(x_j))},
\end{align}
where $\ft K_{1,x_j}$ and $\ft K_{2,x_j}$
are  defined in~\eqref{eq.cstCx}, $
B_{j}:=
\sum \limits_{q\in \argmin_{ \ft C_{\mbf j}(x_{j})}f }     \left(  \det\Hess f(q)\right)^{-\frac 12}$,   
\begin{equation*}
\label{eq.Aj}
A_{j,1}={\scriptsize \sum_{\substack{z\in \mbf j(x_{j})\\  \vert \nabla f(z)\vert \neq 0}}
\frac{2 \pa_{{\ft n}_{\Omega}}f(z)} {\left(  \det\Hess f|_{\pa \Omega}(z)\right)^{\frac 12}} }
\  \text{ and } \ A_{j,2}=  \frac {1}{\sqrt \pi}  {\scriptsize\sum_{\substack{z\in \mbf j(x_{j})\\  \vert \nabla f(z)\vert = 0}} \frac{(1+{\bf 1}_{\pa\Omega}(z))\,\vert \mu_d\vert }  {\left\vert   \det\Hess f(z) \right\vert   ^{\frac 12} } },
\end{equation*}
where ${\bf 1}_{\pa\Omega}(z)=1$ if $z\in \pa\Omega$ and ${\bf 1}_{\pa\Omega}(z)=0$ if not,
and $\mu_{d}$ denotes the negative eigenvalue
of $\Hess f(z)$
 when $z\in \mbf j(x)$ and  $\vert \nabla f(z)\vert = 0$.\\
Finally, when $\mbf j(x_j)\cap \pa \Omega$ does not contain any critical point of $f$,  the above error term $O(\sqrt h)$  is actually of order $O(h)$ in~\eqref{eq.eq-lambdathm}. 
\end{theorem}

\begin{remark}  The first statement of Theorem~\ref{th.thm3} is a simple consequence of its first item together with Theorem~\ref{th.thm2} (or even of Theorem~\ref{thm.count} when $\ft m^*=\ft m_0$). 
Moreover, when in addition $f(\mathbf j(x_{1}))-f(x_{1})> \ldots > f(\mathbf j(x_{\ft m^*}))-f(x_{\ft m^*})$,
the eigenvalues $\lambda_{1,h},\dots,\lambda_{\ft m^{*},h}$
are respectively $\Lambda_h(x_{1}),\dots,\Lambda_h(x_{\ft m^{*}})$. 
They
are then simple and, for every $\ell\in\{1,\dots,\ft m^{*}-1\}$,  
there exists $c>0$ such that it holds
$\lambda_{ \ell+1,h}=O(e^{-\frac ch} )\lambda_{\ell,h}$ in the limit $h\to 0$.
In general, the situation is slightly more involved and, when for example 
Theorem~\ref{th.thm3} applies with $\ft m^{*}=2$ and
$f(\mathbf j(x_{1}))-f(x_{1})=f(\mathbf j(x_{2}))-f(x_{2})$,
Theorem~\ref{th.thm3} permits to discriminate
which eigenvalue among~$\Lambda_h(x_{1})$ and~$\Lambda_h(x_{2})$
is~$\lambda_{1,h}$ if and only if~$(A_{1,1}/B_{1},A_{1,2}/B_{1} ) \neq (A_{2,1}/B_{2},A_{2,2}/B_{2} )$, even though~$\lambda_{1,h}$ is simple (see~\cite{DoNe} in this connection when $f$ is a     double-well potential). 
\end{remark}

\begin{remark} 
\label{re.optimality}
 The term $O(\sqrt h)$  in~\eqref{eq.eq-lambdathm} 
 is  in general optimal, see  Remark~\ref{re.lap} and~item~1 in Proposition~\ref{pr.qm2}.
\end{remark}

\begin{remark}
\label{re.cor3}
Note that under the hypotheses made in Corollary~\ref{cor.thm3'},
the set of principal wells of $f$
consists in the unique element $\ft C(x_{1})=\ft C_{\bf j}(x_{1})=\{f<\min_{\pa\Omega}f\}$,
where
$x_{1}\in \argmin_{\overline\Omega}  f$
(see Definition~\ref{de.1} and Section~\ref{sec.j}).
It holds moreover ${\bf j}(x_{1})=\{z_{1},\dots,z_{N}\}$
and  the hypotheses of Theorem~\ref{th.thm3} are satisfied for
$\ft m^{*}=1$. The statement of  Corollary~\ref{cor.thm3'} follows easily.
\end{remark}


 \begin{proof}
 Let us work with the labeling of  $\ft U_0=\{x_1,\ldots,x_{\ft m_0}\}$ considered in the statement of Theorem~\ref{th.thm3}. Note
 in passing that  the labeling of $\{x_{\ft m^*+1},\ldots,x_{\ft m_0}\}$ is actually arbitrary.  
 Let us moreover order   $ \left ( \psi_x\right)_{x\in \ft U_0}=(\psi_{1},\ldots,\psi_{{\ft m_0}})$ and $\left ( \Theta_x\right)_{x\in \ft U_0} =(\Theta_{1},\ldots,\Theta_{{\ft m_0}})$ according to this labeling of $\ft U_0$. 
The proof of Theorem~\ref{th.thm3} is divided into several steps and  is partly inspired by the analysis led in \cite[Section~7.4]{herau-hitrick-sjostrand-11} which generalizes the procedure made in~\cite{HKN,HeNi1} (see also~\cite[Section C.3.1.2]{BN2017}).
\medskip

\noindent
\textbf{Step 1.} 
Let us first choose an   orthonormal basis
$\mathcal B_0$ of $\Ran(\pi_h)$
 in $L^2(\Omega)$ and   an adapted orthonormal basis   $\mathcal B_1$   of $\sspan\left ( \Theta _i^\pi\right)_{i\in \{1,\ldots,\ft m_0\}}$ in $\Lambda^1L^2(\Omega)$.
\medskip

\noindent
\textbf{Step 1.a) Choice of the basis $\mathcal  B_0$.}\\
Let us first prove that items 2 and 3 in Theorem~\ref{th.thm3} imply the existence of  $c>0$ such that for every $h$ small enough, 
\begin{equation}\label{eq.psij-ort}
\forall i,j \in  \{1,\ldots, \ft m^*\}, \, \, \,\,\,\,\langle \psi_i,\psi_j\rangle_{L^2(\Omega)}=\delta_{i,j}+O\big (e^{-\frac ch}\big )  .
\end{equation}
To this end, let us recall that from~\eqref{eq.psix=10} and Definition~\ref{de.QM}, one has
\begin{equation}\label{eq.psij-supp}
\forall i\in \{1,\ldots, \ft m_0\}\,,\ \ \  \text{supp } \psi_i\subset  \overline{\ft \Omega_1(x_i)}
\end{equation}
and let us consider  $i,j \in  \{1,\ldots, \ft m^*\}$.  
According to item~2  in Theorem~\ref{th.thm3}, 
it thus holds   $\mbf j(x_i)\cap \mbf j(x_j)=\emptyset$ and, according to item~4.(i)  in~Section~\ref{sec.j}, there are two possible cases which finally lead to \eqref{eq.psij-ort}:
\begin{itemize}
\item[--]
 either $\overline{\ft C_{\mbf j}(x_j)}\cap\overline{\ft C_{\mbf j}(x_i)}=\emptyset$, in which case,  according to item 3.(i) in Definition~\ref{de.omega1} and to \eqref{eq.psij-supp}, the supports of $\psi_i$ and $\psi_j$ are disjoint and thus $\langle \psi_i,\psi_j\rangle_{ L^2(\Omega)}=0$,
 \item[--]  
 or, up to switching $i$ and $j$, $\overline{\ft C_{\mbf j}(x_j)}\subset {\ft C_{\mbf j}(x_i)}$, in which case,  
 according to  item~3.(i) in Definition~\ref{de.omega1}, 
$\overline{\ft \Omega_1(x_j)}\subset {\ft \Omega_2(x_i)}\subset  {\ft \Omega_1(x_i)}$. 
In this case, it then follows from  Definition~\ref{de.QM}, \eqref{eq.psix=10}, and~\eqref{Zx},  that
$$\langle \psi_i,\psi_j\rangle_{L^2(\Omega)}=\frac{\int_{\ft \Omega_1(x_j)} \phi_i\phi_{j} \, e^{-\frac 2h f}}{ Z_{x_j}Z_{x_i}}\le C h^{-\frac d2}\, e^{-\frac1h (    2f(x_j) -f(x_i)-f(x_j)   )},$$
where we  also used the relation $\min_{ \overline{\ft \Omega_1(x_j) }}f = \min_{ \overline{\ft C_{\mbf j}(x_j)}}f=f(x_j)$ arising
 from the construction of the map $\ft C_{\mbf j}$ and  item 1 in Definition~\ref{de.omega1}.  Moreover, using
  item~3 in Theorem~\ref{th.thm3}, it holds $f(x_j)>f(x_i)$, and thus, there exists $c>0$ such that when $h\to 0$:
$$\langle \psi_i,\psi_j\rangle_{L^2(\Omega)}=O\big (e^{-\frac ch}\big ).$$  
\end{itemize} 
\noindent
Then, according to \eqref{eq.psij-ort} and to Proposition~\ref{pr.qm4}, there exists $c>0$ such that  for  all $i,j \in  \{1,\ldots, \ft m^*\}$,  it holds in the limit $h\to 0$:
\begin{equation}\label{eq.deltaijpsi}
\langle \psi^\pi_i,\psi^\pi_j\rangle_{\Lambda^1L^2(\Omega)}= \delta_{i,j}+O\big (e^{-\frac ch}\big ).
\end{equation}
Let us now consider the    Gram-Schmidt orthonormalization  $\mathcal  B_0:=(e_1,\ldots, e_{\ft m_{0}})$
 of the
family $(\psi^\pi_1,\dots,\psi^\pi_{\ft m_{0}})$ in $L^2(\Omega)$. According to \eqref{eq.deltaijpsi}, it thus holds, for all    $k\in \{1,\ldots,\ft m^*\}$, 
$$e_k=\big(1+O\big (e^{-\frac ch}\big )\big)\psi_k^\pi  \, +\,  \sum_{q=1}^{k-1} O\big (e^{-\frac ch}\big ) \,  \psi_q^\pi .$$ 
Thus, the matrix $ C_0^\pi$ defined by \eqref{eq.Spi} has   the  block structure
\begin{equation}\label{eq.C0pi-block}
 C_0^\pi = \begin{bmatrix}
I_{\ft m^*} +O\big (e^{-\frac ch}\big )&  [ C_0^\pi ]_2\\ 0 &   [ C_0^\pi ]_4  \end{bmatrix},
\end{equation}
where $I_{\ft m^*} $ is the identity matrix of $\mathbb R^{\ft m^*\times \ft m^*}$, $  [ C_0^\pi ]_4 \in \mathbb R^{(\ft m_0-\ft m^*)\times (\ft m_0-\ft m^*)}$ is an invertible matrix (since, according to Proposition~\ref{pr.indep-psix}, $ C_0^\pi$ is invertible), and $  [ C_0^\pi ]_2 \in \mathbb R^{\ft m^*\times (\ft m_0-\ft m^*)}$. 
One then defines the  $\ft m_0\times \ft m_0$ matrix $C_0$ by
\begin{equation}\label{eq.C0-block}
 C_0:= \begin{bmatrix}
I_{\ft m^*}   & [ C_0^\pi ]_2\\ 0 &  [ C_0^\pi ]_4  \end{bmatrix},
\end{equation}
so that, according to Proposition~\ref{pr.indep-psix}, $C_0$ is invertible and
\begin{equation}\label{eq.c0fan}
\!\!\!C_{0}=O(1), \ 
 C_0^{-1}= \begin{bmatrix}
I_{\ft m^*}   &\!\!\! -[ C_0^\pi ]_2   [ C_0^\pi ]_4^{-1}\\ 0 & \!\! [ C_0^\pi ]_4^{-1}  \end{bmatrix}=O(1),   \text{ and }   C_0 ^{-1} C_0^\pi  =I_{\ft m_0}  +O\big (e^{-\frac ch}\big ).
\end{equation}

\noindent
\textbf{Step 1.b) Choice of the basis  $\mathcal  B_1$. }\\
According to Definition~\ref{de.proj-qm}, item 2 in Theorem~\ref{th.thm3}, and to
item 2.(i) in Proposition~\ref{pr.qm2}, it holds, for every $h$ small enough:
\begin{equation}\label{eq.deltaij-tt}
\forall j\in \{1,\ldots, \ft m^*\}\,,\ \forall i\in  \{1,\ldots, \ft m_0\}\,, \ \ \  \langle \Theta_i,\Theta _j\rangle_{\Lambda^1L^2(\Omega)}= \delta_{i,j}.
\end{equation} 
 Thus, using in addition Proposition~\ref{pr.qm4}, it holds, for every $h$ small enough:
\begin{equation}\label{eq.deltaijTheta}
\forall j\in \{1,\ldots, \ft m^*\}\,,\ \forall i\in  \{1,\ldots, \ft m_0\}\,, \ \ \  \langle \Theta^\pi_i,\Theta^\pi_j\rangle_{\Lambda^1L^2(\Omega)}= \delta_{i,j}+O(h).
\end{equation}
Let us now consider the   Gram-Schmidt orthonormalization  
$\mathcal  B_1:=(\Upsilon_1,\ldots, \Upsilon_{\ft m_{0}})$ of the
family $(\Theta^\pi_1,\dots, \Theta^\pi_{\ft m_{0}}\}$ in $\Lambda^1L^2(\Omega)$. 
It thus holds in the limit $h\to 0$,
\begin{equation*}\label{eq.C1gram}
\forall k\in \{1,\ldots,\ft m^{*}\}\,,\ \ \ \Upsilon_k=\big(1+O(h)\big)\Theta_k^\pi  \, +\,   \sum_{q=1}^{k-1} O(h) \,  \Theta_q^\pi
\end{equation*}
and, for some real numbers $a_{k,q}(h)$, where $k\in\{\ft m^{*}+1,\ldots,\ft m_0\}$
and $q\in\{\ft m^{*}+1,\ldots,k\}$,
\begin{equation*}\label{eq.C1gram'}
\forall k\in \{\ft m^{*}+1,\ldots,\ft m_0\}\,,\ \ \ \Upsilon_k=  \sum_{q=1}^{\ft m^{*}} O(h) \,  \Theta_q^\pi
+
\sum_{q=\ft m^{*}+1}^{k} a_{k,q}(h) \,  \Theta_q^\pi.
\end{equation*}
Hence, with this  choice of $\mathcal  B_1$,  the matrix $ C_1^\pi$ defined by~\eqref{eq.Spi} has the block structure
\begin{equation}\label{eq.C1pi-block}
 C_1^\pi = \begin{bmatrix}
I_{\ft m^*} +O(h) &  O(h)\\ 0 &  [ C_1^\pi ]_4  \end{bmatrix},
\end{equation}
where 
  $   [ C_1^\pi ]_4 \in \mathbb R^{(\ft m_0-\ft m^*)\times (\ft m_0-\ft m^*)}$ is an invertible matrix (since $ C_1^\pi$ is invertible, see indeed Proposition~\ref{pr.indep-psix})
  and,
according to \eqref{eq.borne-CC}, $ [ C_1^\pi ]_4=O(1)$ and  $ [ C_1^\pi ]^{-1}_4=O(1)$
in the limit $h\to 0$.  
Finally,  let us  define the $\ft m_0\times \ft m_0$ matrix $C_1$ by
 \begin{equation}\label{eq.C1-block}
 C_1 := \begin{bmatrix}
I_{\ft m^*}   &0\\ 0 &   [ C_1^\pi ]_4  \end{bmatrix},
\end{equation}
so that, in the limit $h\to 0$, it holds
\begin{equation}\label{eq.c1fan'}
C_1=O(1)\ \ \text{and}\ \ C_1^{-1}=O(1)
\end{equation}
and
\begin{equation}\label{eq.c1fan}
\Vert  C_1^{-1} \, (I_{\ft m_0}+O(h)) \,  C_1^\pi  \Vert=  1+O(h)\  \text{ and }\ \Vert   ( C_1^\pi)^{-1} (I_{\ft m_0}+O(h))  \, C_1 \Vert=1+O(h).
\end{equation}

\noindent
\textbf{Step 2.} Let us recall that in the limit $h\to 0$,
the $\ft m_{0}$ smallest eigenvalues of 
$\Delta_{f,h}^{D}$ are the squares of the singular values of the matrix 
$ S^{\mathcal B_0,\mathcal B_1}=\,^tC_1^\pi\, S^\pi \,  C_0^\pi\in  \mathbb R^{\ft m_0 \times \ft m_0}$,
 where $C_0^\pi$, $C_1^{\pi}$, and $S^\pi$
 are defined in \eqref{eq.Spi} and in \eqref{eq.Spi2}.
Moreover, the relation \eqref{eq.SpiS}
leads to the factorization (see \eqref{eq.Sij} for the definition of the matrix $S$)
$$S^{\mathcal B_0,\mathcal B_1}=\,  ^t\big(  C_1^{-1} (I_{\ft m_0}+O(h))C_1^\pi   \big)    \,  ^tC_1\, S  \,  C_0   \, \big(C_0^{-1}C_0^\pi).$$
Using \eqref{eq.c0fan}, \eqref{eq.c1fan}, and Lemma~\ref{le.fan}, it follows that
\begin{equation}
\label{eq.part-basis-sing}
 \forall j\in\{1,\dots,{\ft m_{0}}\},\ \ \      \eta_{ j}(S^{\mathcal B_0,\mathcal B_1})= \eta_{ j}(  ^tC_1\, S  \,  C_0  )\, \big(1+O(h)\big).
 \end{equation}
 Hence, the  $\ft m_0$ smallest eigenvalues of $\Delta_{f,h}^D$ are, up to a multiplicative term
 of order $\big(1+O(h)\big)$, the squares of the singular values of the matrix $\, ^tC_1\, S \,  C_0$.\medskip
 
\noindent
In order to prepare the precise computation of these singular values made
in the following step, let us first suitably decompose the matrices
taking part into $\, ^tC_1\, S \,  C_0$.
To this end, let us introduce 
$$
\ft k^{*}\in\{1,\dots,\ft m^{*}\}
$$
and
 write the diagonal matrix $D$ defined by~\eqref{eq.DiagD} and~\eqref{eq.pj} as follows:
 \begin{equation}\label{eq.Dblock}
D= \begin{bmatrix}
 D_1 & 0\\ 0 &   D_2  \end{bmatrix},
\end{equation}
where $D_1$ is the square diagonal matrix of size $\ft k^*$  defined by
 \begin{equation}\label{eq.D1-block}
D_1:=\text{Diag}\big  (h^{p_1} \, e^{-\frac 1h (f(\mbf j(x_1))-f(x_1))}, \ldots, h^{p_{\ft k^*}} \, e^{-\frac 1h (f(\mbf j(x_{\ft k^*}))-f(x_{\ft k^*}))     }\big )
\end{equation}
and $D_2$ is the square diagonal matrix of size $\ft m_0-\ft k^*$  defined by
 \begin{equation}\label{eq.D2-block}
 D_2:=\text{Diag}\big  (h^{p_{\ft k^*+1}} \, e^{-\frac 1h (f(\mbf j(x_{\ft k^*+1}))-f(x_{\ft k^*+1}))}, \ldots, h^{p_{\ft m_0}} \, e^{-\frac 1h (f(\mbf j(x_{\ft m_0}))-f(x_{\ft m_0}))     }\big ).
 \end{equation} 
Moreover, according to \eqref{eq.deltaij-tt},
the matrices $S=\big(\Vert d_{f,h}\psi_j \Vert_{\Lambda^1L^2(\Omega)}\,\langle \Theta_i,\Theta _j\rangle_{\Lambda^1L^2(\Omega)} \big)_{i,j}$ and $T=SD^{-1}$ defined in \eqref{eq.Sij}
and in \eqref{eq.T} have the block structure
 \begin{equation}\label{eq.Sblock}
S = \begin{bmatrix}
 S_1 & 0\\ 0 &   S_2  \end{bmatrix} \ \text{ and }  \ T= \begin{bmatrix}
 T_1 & 0\\ 0 &   T_2  \end{bmatrix},
\end{equation}
where:
\begin{itemize}
\item[--] $T_1$ and $S_1$ are  square diagonal matrices of size $\ft k^*$ defined by
 \begin{equation}\label{eq.S1block}
 S_1:= \text{ Diag}( S_{1,1},\ldots, S_{ \ft k^*,\ft k^*}   )\quad\text{and}\quad 
 T_1:=S_1\, D_1 ^{-1},
 \end{equation}
\item[--] $T_2, S_2\in \mathbb R^{(\ft m_0-\ft k^*)\times (\ft m_0-\ft k^*)}$
and, according to Lemma~\ref{le.invT},
\begin{equation}\label{eq.T2}
\text{$T_{2}=S_2\, D_2^{-1}$ is invertible\ \  and}\ \  T_{2}^{-1}=O(1).
\end{equation}
\end{itemize}
Using in addition~\eqref{eq.C0-block} and~\eqref{eq.C1-block}, the matrices $C_0$, $C_1$ and thus $ ^tC_1 S   C_0$ have the block  structures
\begin{equation}\label{eq.C1SC0}
 C_0= \begin{bmatrix}
I_{\ft k^*}   & U\\ 0 &  V  \end{bmatrix}
,\ \ C_1= \begin{bmatrix}
I_{\ft k^*}   & 0\\ 0 &  W  \end{bmatrix},
\ \ \text{and thus}\ \  \, ^tC_1\, S \,  C_0 = \begin{bmatrix}
S_1 &  S_1U\\ 0 &  \,^tW S_2V  \end{bmatrix},
\end{equation}
where, 
according to \eqref{eq.c0fan} and \eqref{eq.c1fan'}, it holds in the limit $h\to 0$:
\begin{equation}\label{eq.C0C1-block-kstar2}
U,V=O(1)\,, \ V^{-1}=O(1) \ \ \text{ and } \ \ W,W^{-1}=O(1).
\end{equation}
Note lastly  that when $\ft k^*=\ft m^*$, one has $U=[ C_0^\pi ]_2$, $V=[ C_0^\pi ]_4$, and $W=[ C_1^\pi ]_4$. \medskip

\noindent
\textbf{Step 3.} 
We are now in position  to prove Theorem~\ref{th.thm3}. To this end,
we will
 compute
the smallest singular values of the  matrix $\, ^tC_1\, S \,  C_0$ that we have seen to be, up
to a multiplicative error term of order $1+O(h)$, the square roots of the smallest
eigenvalues of $\Delta_{f,h}^{D}$ (see indeed \eqref{eq.part-basis-sing}).\\
In the following, one uses the block decompositions exhibited in 
\eqref{eq.Dblock}--\eqref{eq.C1SC0}
%
and,  for $\ell \in \mathbb N$, one denotes by $\| \cdot \|_2$  the Euclidean norm on  $ \mathbb R^\ell$.  
Moreover, for every $h$ small enough, one chooses the ordering of the set $\{x_{1},\dots,x_{\ft m^{*}}\}$,
depending on $h$, such that
$$
\text{the sequence $\big(S_{j,j}\big)_{j\in\{1,\dots,\ft m^{*}\}}$ is increasing.}
$$
According 
to   \eqref{eq.part-basis-sing}, \eqref{eq.S1block}, and to Proposition~\ref{pr.interS}, it then suffices 
to show that 
there exists $c>0$
 such that  it holds in the limit 
$h\to 0$, 
 \begin{equation}
\label{eq.=S}
\forall \ell \in\{1,\dots,\ft m^{*}\}\,,\ \ \ \eta_{\ft m_0-\ell+1}(\, ^tC_1\, S \,  C_0)=
S_{\ell,\ell}\,\big(1+O(e^{-\frac ch})\big).
\end{equation}
To this end, we recall that
 by the Max-Min principle, one has 
 for every 
$\ell \in\{1,\dots,\ft m_{0}\}$,
\begin{align}
\label{eq.sing-val2-max-min}
\eta_{\ft m_{0}-\ell+1}(\, ^tC_1\, S \,  C_0)& = \max_{E\subset \mathbb R^{\ft m_{0}}   , \dim E=\ell-1} \  
\min_{y\in E^{\perp} \, ;\, \|y\|_2=1} \big\| \, ^tC_1\, S \,  C_0  y\big\|_2\\
\label{eq.sing-val2-min-max}
&=
\min_{E\subset \mathbb R^{\ft m_{0}}   , \dim E=\ell}\ \max_{y\in E \, ;\, \|y\|_2=1} \big\| \, ^tC_1\, S \,  C_0y\big\|_2.
\end{align}

\noindent
To obtain the upper bound in \eqref{eq.=S}
for some arbitrary
$\ell \in\{1,\dots,\ft m^{*}\}$,
we apply~\eqref{eq.sing-val2-min-max}  which gives,
according  to
\eqref{eq.C1SC0} applied with $\ft k^{*}=\ell$ and to  \eqref{eq.S1block}:
\begin{align}
\nonumber
\eta_{\ft m_0-\ell+1}(\, ^tC_1\, S \,  C_0)\ \leq \ 
\max_{y\in \mathbb R^{\ell}\, ;\, \|y\|_2=1} \big\|\, ^tC_1\, S \,  C_0\, (y,0,\dots,0)\big\|_2&\ =\ \max_{y\in \mathbb R^{\ell} \, ;\, \|y\|_2=1} \big\| S_{1} y\big\|_2\\
\label{eq.upper-etam0} 
&\ =\ S_{\ell,\ell}.
\end{align} 

\noindent
Let us now prove the lower bound  in \eqref{eq.=S} for 
some arbitrary $\ell\in\{1,\dots, \ft m^{*}\}$.
For that purpose, let us
 introduce 
 $y^* \in \mathbb R^{  \ft m_{0} } $ such that $\|y^*\|_2=1$,
$y^{*}\in (\mathbb R^{\ell-1}\times\{0,\dots,0\})^{\perp}$, and
$$
\big\| \, ^tC_1\, S \,  C_0\, y^*\big\|_2=\min\limits_{y\in(\mathbb R^{\ell-1}\times\{0,\dots,0\})^{\perp} \, ;\, \|y\|_2=1} \big\|\, ^tC_1\, S \,  C_0\, y\big\|_2.
$$ 
Note that according to \eqref{eq.sing-val2-max-min}, it holds in particular
\begin{equation}
\label{eq.=S'}
\eta_{\ft m_0-\ell+1}(\, ^tC_1\, S \,  C_0)\ \geq\ \big\| \, ^tC_1\, S \,  C_0\, y^*\big\|_2.
\end{equation}
Let us also introduce
$\ft k^{*}\in\{\ell,\dots,\ft m^{*}\}$
 such that 
\begin{equation}
\label{eq.k*}
f(\mbf j(x_{\ell}))-f(x_{\ell}) = f(\mbf j(x_{\ft k^*}))-f(x_{\ft k^*}) >  \max\limits_{ j\in \{\ft k^*+1,\ldots,\ft m_0\}} f(\mbf j(x_{j}))-f(x_{j}).
\end{equation}
 Note that this is indeed possible
by  the first item of Theorem~\ref{th.thm3}.
Let us then write $y^*=( y^*_{a},y^*_{b})$, where $y^*_{a}\in \mathbb R^{\ft k^{*}} $ and~$y^*_{b}\in\mathbb R^{\ft m_{0}-\ft k^{*}} $,
and let us prove that there exists $c>0$ such that when $h\to 0$, 
\begin{equation}
\label{eq.ybeta}
\|y^*_{b}\|_2= O \big(e^{-\frac{c}{h}} \big).
\end{equation} 
According to \eqref{eq.=S'}, \eqref{eq.C1SC0} applied with $\ft k^{*}$, 
and to the triangular inequality, one has
\begin{align*}\eta_{\ft m_0-\ell+1}(\, ^tC_1\, S \,  C_0)\geq  \big  \|\, ^tC_1\, S \,  C_0\, (y^*_{a},y^*_{\beta} )\big \|_2
&\geq \big\|\, ^tC_1\, S \,  C_0\, (0,y^*_{b})\big\|_2-\big\|\, ^tC_1\, S \,  C_0\, (y^*_{a},0)\big\|_2\\
&= \big\|\, ^tC_1\, S \,  C_0\, (0,y^*_{b})\big\|_2-\big\|S_{1} y^*_{a}\big\|_2.
\end{align*}
Using in addition \eqref{eq.upper-etam0} and \eqref{eq.S1block} with $\ft k^{*}$,
it follows that in the limit $h\to 0$:
 \begin{equation}
 \label{eq.yb}
 \big\|\, ^tC_1\, S \,  C_0\, (0,y^*_{b})\big\|_2\ \le\  S_{\ell,\ell} + \|S_{1}\|\, \| y_a^*\|_{2} \ \leq\ 2\,S_{\ft k^{*},\ft k^{*}}.
 \end{equation}
Moreover, according to~\eqref{eq.C1SC0},  one has
\begin{align*}
\big\|\, ^tC_1\, S \,  C_0\, (0,y^*_{b})\big\|_2=\Big(  \big\|\,S_{1} U\, y^*_{b}\big\|_2^2  + \big\|\,    ^tWS_2V\, y^*_{b} \big\|_2^2 \Big)^{\frac 12}
&\ge \big\|\,    ^tWS_2V\, y^*_{b} \big\|_2,
\end{align*}
where, using \eqref{eq.T2} and \eqref{eq.C0C1-block-kstar2},
it holds for some $C>0$ in the limit $h\to0$,
$$
\big\|\,    ^tWS_2V\, y^*_{b} \big\|_2=
\big\|\,    ^tW T_{2}D_2V\, y^*_{b} \big\|_2 
\geq \frac 1C\|D_{2}^{-1}\|^{-1} \|y^*_{b}\|_{2}.
$$
It then follows from \eqref{eq.yb} that in the limit $h\to 0$,
it holds
\begin{align*}
\|y^*_{b}\|_{2}&\leq 2\,C\,\|D_{2}^{-1}\|\,S_{\ft k^{*},\ft k^{*}},
\end{align*}
which leads to \eqref{eq.ybeta} according to 
item 2 in Proposition~\ref{pr.interS},
\eqref{eq.D2-block}, and to 
\eqref{eq.k*}.\medskip

\noindent
Then, using
\eqref{eq.=S'},
\eqref{eq.C1SC0} with $\ft k^{*}$, and~\eqref{eq.ybeta} together with the fact that  $ U =O(1)$ (see~\eqref{eq.C0C1-block-kstar2}), we obtain the existence of $c>0$ such that  it holds in the limit $h\to 0$,
\begin{align*}
\eta_{{\ft m_{0}}-\ell+1}(\, ^tC_1\, S \,  C_0) \geq \big\| \, ^tC_1\, S \,  C_0\, y^*\big\|_2&\ge \big\Vert S_{1}y^*_{a}\big\Vert_{2} -  \big\Vert S_{1}U y_b^*\big\|_2\\
&=  \big\Vert S_{1}y^*_{a}\big\Vert_{2} -  \|S_{1}\|\,O\big (e^{-\frac ch}\big ).
\end{align*}
Hence, using in addition
$\|y^*_{a}\|_{2}= 1+O \big(e^{-\frac{c}{h}} \big)$
(which follows from \eqref{eq.ybeta} and~$\|y^*\|_2=1$),
$y_{a,1}^{*}=\cdots=y_{a,\ell-1}^{*}=0$
(since
$y^{*}\in (\mathbb R^{\ell-1}\times\{0,\dots,0\})^{\perp}$), 
\eqref{eq.S1block}, 
item 2 in Proposition~\ref{pr.interS}, and
\eqref{eq.k*},
it holds in the limit $h\to 0$,
$$
\eta_{{\ft m_{0}}-\ell+1}(\, ^tC_1\, S \,  C_0) \geq S_{\ell,\ell} \big(1+O (e^{-\frac {c}{h}})\big)
-S_{\ft k^{*},\ft k^{*}}\,O\big (e^{-\frac ch}\big )
\geq S_{\ell,\ell}\,\big(1+O(e^{-\frac c{2h}} )\big),
$$
which concludes the proof of \eqref{eq.=S}. The proof of Theorem~\ref{th.thm3} is thus complete.
 \end{proof}

\noindent
\textbf{Acknowledgements}. 
This work was partially supported by the grant PHC AMADEUS 2018 PROJET N$^{\text{o}}$ 39452YK.

\small{
\bibliography{petites_vp_fonction-uptdated} 
\bibliographystyle{plain}

}

\end{document}